\documentclass[10pt,letterpaper,reqno]{amsart}

\usepackage[dvips]{epsfig}
\usepackage{amscd}
\usepackage{amssymb}
\usepackage{amsthm,amsxtra}
\usepackage{amsmath}
\usepackage{mathrsfs}
\usepackage{latexsym}
\usepackage{alltt}
\usepackage{color}
\usepackage{relsize}

\usepackage{upref}
\usepackage{hyperref}
\theoremstyle{plain}
\newtheorem{thm}{Theorem}[section]
\theoremstyle{plain}

\newtheorem{lem}[thm]{Lemma}
\newtheorem{prop}[thm]{Proposition}
\newtheorem{cor}[thm]{Corollary}

\newtheorem{property}[thm]{Property}
\theoremstyle{definition}
\newtheorem{defi}{Definition}[section]
\newtheorem{rem}{Remark}
\newtheorem{assum}[thm]{Assumption}
\newtheorem{main}{Main Result}

\renewcommand{\d}{\/\mathrm{d}\/}

\def\L{T\wedge\rho_N}
\def\Ln{T\wedge\rho^n_N}
\def\Lm{T\wedge\rho_N^{n,m}}

\def\v{\mathbf{v}}

\def\Q{\mathbf{Q}}

\def\\tau{\mathbf{\tau}}
\def \Rn {\rho_N^{n,m}}

\newcommand{\R}{\mathbb{R}}

\renewcommand{\d}{\/\mathrm{d}\/}

{%
\setcounter{enumi}{0}

\begin{enumerate}}
{\end{enumerate} }

{
\setcounter{enumi}{0}

\begin{enumerate}}
{\end{enumerate} }

\numberwithin{equation}{section} 
\allowdisplaybreaks

\begin{document}

\title[Strong Solutions of Viscoelastic Flows]{Strong Solutions of Stochastic Models for Viscoelastic Flows of Oldroyd Type}

\author[Utpal Manna]{Utpal Manna}

\address{%
School of Mathematics\\
Indian Institute of Science Education and Research (IISER) Thiruvananthapuram\\
Thiruvananthapuram 695016\\
Kerala, INDIA}

\email{manna.utpal@iisertvm.ac.in}

\author[Debopriya Mukherjee]{Debopriya Mukherjee}

\address{%
School of Mathematics\\
Indian Institute of Science Education and Research (IISER) Thiruvananthapuram\\
Thiruvananthapuram 695016\\
Kerala, INDIA}

\email{debopriya13@iisertvm.ac.in}

\begin{abstract}
In this work we study stochastic Oldroyd type models for viscoelastic fluids in $\R^d, d= 2, 3$. We show existence and uniqueness of strong local maximal solutions when the initial data are in $H^s$ for $s>d/2, d= 2, 3$. Probabilistic
estimate of the random time interval for the existence of a local solution is expressed in terms of expected values of the initial data.   
\end{abstract}



\maketitle
\tableofcontents

\section{Introduction}
Over the past few years, there have been many works devoted to viscoelastic
fluids in dimensions two and three. Most of these works are concerned about local existence of strong solutions, global existence of weak solutions, necessary condition for blow-up (in the spirit of well-known Beale-Kato-Majda criterion \cite{BKM}) and global well-posedness for smooth solutions with small initial data. 

In this work, we focus upon the classical Oldroyd type models for viscoelastic fluids (see, Oldroyd \cite{Old}) in $\R^d, d= 2, 3$
\begin{align}
\frac{\partial \mathbf{v}}{\partial
t}+(\mathbf{v}\cdot\nabla)\v-\nu\Delta \v+\nabla
p &=\mu_1 \nabla \cdot \mathbf{\tau}  \quad \mbox{in} \quad \mathbb{R}^d \times (0,T),    \label{e1}\\
\frac{\partial \mathbf{\mathbf{\tau}}}{\partial
t}+(\v \cdot \nabla)\mathbf{\tau} +a\mathbf{\tau} +\Q(\mathbf{\tau}, \nabla \v)&=\mu_2 \mathcal{D}(\v) \quad \mbox{in} \quad \mathbb{R}^d \times (0,T), \label{e2}\\
\nabla\cdot\v&=0 \quad \mbox{in} \quad \mathbb{R}^d \times (0,T),\label{e3}\\
\v(0, \cdot)=\v_0,\,\,\mathbf{\tau}(0, \cdot)&=\mathbf{\tau}_0 \quad \mbox{in} \quad \mathbb{R}^d .\label{e4}
\end{align}
Here $\v$ is the velocity vector field which is assumed to be divergence free, $\mathbf{\tau}$ is the non-Newtonian part of the stress tensor (i.e., $\mathbf{\tau}(x,t)$ is a $(d,d)$ symmetric matrix), $p$ is the pressure of the fluid, which is a scalar. The parameters $\nu$ (the viscosity of the fluid), $a$ (the reciprocal of the relaxation time), $\mu_1$ and $\mu_2$ (determined by the dynamical viscosity of the fluid, the retardation time and $a$) are assumed to be non-negative. $\mathcal{D}(\v)$ is called the deformation tensor and is the symmetric part of the velocity gradient 
$$\mathcal{D}(\v)=\frac{1}{2}(\nabla \v+ \nabla^t \v).$$
$\Q$ is a quadratic form in $(\mathbf{\tau}, \nabla\v)$. As remarked in Chemin and Masmoudi \cite{CM}, since the equation for the stress tensor should be invariant under coordinate transformation, $\Q$ cannot be most general quadratic form, and for Oldroyd fluids one usually chooses 
$$\Q(\mathbf{\tau},\nabla \v)=\mathbf{\tau} \mathcal{W}(\v)-\mathcal{W}(\v) \mathbf{\tau} -b\left(\mathcal{D}(\v)\mathbf{\tau}+\mathbf{\tau} \mathcal{D}(\v)\right),$$
where $b \in [-1,1]$ is a constant and  $\mathcal{W}(\v)=\frac{1}{2}(\nabla \v- \nabla^t \v)$ is the vorticity tensor, and is the skew-symmetric part of velocity gradient.

There is growing literature devoted to these systems and it is almost impossible to provide a complete review on the topic. We shall restrict ourselves to a few significant works which are relevant to our paper. 
\par
\noindent
It is straightforward to observe that the formal $\mathrm{L}^2$-energy estimate of the above system \eqref{e1}-\eqref{e4} is the following:
\begin{align*}
\dfrac{1}{2}\dfrac{d}{dt} (\mu_2\|\v(t)\|_{\mathrm{L}^2}^2 + \mu_1\|\tau(t)\|_{\mathrm{L}^2}^2)+\nu\mu_2 \|\nabla\v(t)\|_{\mathrm{L}^2}^2+a\mu_1\|\tau(t)\|_{\mathrm{L}^2}^2\leq |b|~\|\nabla \v(t)\|_{\mathrm{L}^{\infty}}\|\tau(t)\|_{\mathrm{L}^2}^2. 
\end{align*}
Since by the Brezis-Wainger type logarithmic Sobolev inequality, $\mathrm{L}^{\infty}$-norm of gradient of velocity field can be bounded by that of vorticity field for the Sobolev exponent strictly bigger than $d/2+1$, the difficulty here arises in getting an $\mathrm{L}^{\infty}$ estimate on the vorticity. Indeed, at first glance it would seem hopeless because the vorticity equation involves a transport term as well as a nonlocal term. However, one needs to perform a losing estimate (see, Chemin-Masmoudi \cite{CM}) for the transport equation satisfied by $\tau$ that allow us to obtain a Beale-Kato-Majda (\cite{BKM}) type sufficient condition of non-breakdown. 

Due to the parabolic-hyperbolic coupling and the special structure of $\Q$, the corresponding stationary problem is also interesting and was studied by Renardy \cite{Re}. The existence and uniqueness of local strong solutions
in $H^m$ ($m$ integer) have been established by Guillop\'e and Saut \cite{GS1}. Further, these solutions are global if the coupling between the two equations is weak as well as the initial data are small  Guillop\'e and Saut \cite{GS2}. Existence of $\mathrm{L}^s-\mathrm{L}^r$ solutions has been treated by Fernandez Cara,
Guill\'en, and Ortega \cite{FGO}. Global existence of weak solutions has been shown in the corotational case (with $b=0$) by Lions and Masmoudi \cite{LM}. Some recent works have been devoted to the proof of global well-posedness in the case of small data (e.g. see Masmoudi et al. \cite{CM, LMZ}, Lin et al. \cite{LLZ}, Lei et al. \cite{Lei1, Lei2}).

Let us mention the connection between the deterministic system under consideration in this work and certain other systems showing a critical coupling. A lot of works have been devoted recently to the study of
two and three dimensional Boussinesq system and magnetohydrodynamics (MHD) system with partial dissipation and Ericksen-Leslie nematic liquid crystal model. 
In particular, a critical coupling for Boussinesq system has been studied in Hmidi et al. \cite{H1}-\cite{H4}, Manna and Panda \cite{MP}, for MHD system in Caflisch et al. \cite{CKS}, for liquid crystal model in   Lin and Liu \cite{LL}, to name a few. The coupling in the Boussinesq system is simpler than the one in MHD system (or in liquid crystal model or viscoelastic fluid of Oldroyd type considered in this work) in the sense that the vorticity
equation is forced by the gradient of the temperature but then the temperature solves an unforced convection-diffusion equation. Consequently, in the case of
critical coupling, if one can find a combination of the vorticity and the temperature
that has better regularity properties, it is rather easy to deduce an estimate on each
individual quantity Elgindi and Rousset \cite{ER}. This is not the case for the MHD system, the liquid crystal model or for the Oldroyd model (even if $\Q=0$) since they are strongly coupled. Moreover, due to the special structure of $\Q$, the Oldroyd model under consideration possesses additional difficulty and it is evident from the lack of $\mathrm{L}^2$-energy estimate without prior assumption on the bound of $\mathrm{L}^{\infty}$-norm of the gradient of velocity field. 

Literatures related to analysis of the above critical coupled systems perturbed by random forcing are very limited and quite recent (e.g. see Yamazaki \cite{Ya} for Boussinesq system with zero dissipation, Manna et al. \cite{MaMo} for non-resistive MHD system, Brze\'{z}niak et al. \cite{BHR} for liquid crystal model). To the best of author's knowledge, there is no literature available on the random perturbation of the general non-linear viscoelastic fluid of Oldroyd type \eqref{e1}-\eqref{e4}. However, in Barbu et al. \cite{Ba} and Razafimandimby \cite{Ra}, existence and asymptotic behaviour of a linear visoelastic fluid equation driven by additive or multiplicative Wiener
stochastic processes are studied. This equation is an integro-differential equation (of Volterra type) consisting of the Navier-Stokes equation and a hereditary (or memory) term as the integral of a linear kernel, but doesn't possess any critical nonlinear coupling as in \eqref{e1}-\eqref{e2}. Therefore, the standard techniques of the stochastic Navier-Stokes equation can be borrowed to establish the well-posedness and regularity of solutions.

In this context, we should make a note that the Oldroyd type viscoelastic fluid considered in this paper is purely a macroscopic model. Usually at the macroscopic level, in the case of
non-Newtonian fluids such as polymeric fluids, such an equation links the stress tensor to the velocity field either through a partial differential equation (as in equations \eqref{e1}-\eqref{e2}) or through an integral relation (e.g. in linear Oldroyd model \cite{Ba}, \cite{Ra}). Recently some works have been devoted for the understanding of fluid behaviour both in the macroscopic and microscopic regimes. In order to build a micro-macro model, one needs to go down to the microscopic scale and make use of kinetic theory to obtain a mathematical model for the evolution of the microstructures of the fluid (e.g. configurations of the polymer
chains in the case of polymeric fluid). In mathematical terms, this micro-macro approach translates into a coupled multiscale system (simplest example of such a model is the dumbbell model) in which the polymers are modelled as dumbbells each of which consists of two beads connected by a spring (see \cite{JLL1} and \cite{Ot} for detailed introduction on the subject). In Jourdain et al. \cite{JLL2}, the authors analyse a stochastic finite extensible nonlinear elastic dumbbell model, and prove a local-in-time existence and uniqueness
result. This work has been further extended in Jourdain et al. \cite{JLLO}, where long-time behavior of the solution in various settings
(shear flow, general bounded domain with homogeneous Dirichlet boundary conditions
on the velocity, general bounded domain with non-homogeneous Dirichlet
boundary conditions on the velocity) have been shown. There is also a recent trend in the community of researchers performing numerical simulations of such complex flows, where stochastic micro-macro models are considered as a numerical tool for simulating the dynamic behavior of polymeric fluids (see, \cite{JLL1}, \cite{Li}).

In the present work, we are interested in the mathematical analysis of a stochastic version of \eqref{e1}-\eqref{e4}. Our work is motivated by the importance of external perturbation on the dynamics of the velocity field for fluids with memory. The viscoelastic property demands that the material must return to its original shape after any deforming external force has been removed (i.e., it will show an elastic response) even though it may take time to do so. Hence the equation modelling stress tensor (i.e. the equation for $\mathbf{\tau}$) is invariant under coordinate transformation. Therefore the question of how to incorporate a suitable perturbation modelling the stress tensor without destroying its invariance property is a delicate one. 

Hence for a full understanding of the effect of
fluctuating forcing field on the behaviour of the viscoelastic fluids, one needs to take into account the dynamics of $\v$ and $\mathbf{\tau}$. To initiate this kind of investigation we propose a mathematical study
of the following system of equations which basically describes an approximation of the system governing the viscoelastic fluids under the influence of fluctuating external forces.
\begin{align}
&d \mathbf{v}(t)+[(\mathbf{v}(t)\cdot\nabla)\v(t)-\nu\Delta \v(t)+\nabla
p]dt =\mu_1 \nabla \cdot \mathbf{\tau}(t)dt+\sigma(t,\v(t))\d W_1(t)+\int_Z G(\v(t-), z)\tilde{N}_{1} (\d t,\d z),     \label{se01}\\
&d \mathbf{\tau}(t)+[(\v(t) \cdot \nabla)\mathbf{\tau} (t) +a\mathbf{\tau}(t) +\Q(\mathbf{\tau}(t), \nabla \v(t))]dt=\mu_2 \mathcal{D}(\v(t))dt+ (h \otimes \mathbf{\tau}(t))\circ d W_2(t) , \label{se02}\\
&\nabla\cdot\v=0 ,\label{se03}\\
&\v(0,\cdot)=\v_0,\,\,\mathbf{\tau}(0,\cdot)=\mathbf{\tau}_0  .\label{se04}
\end{align}
where $\,W_1 $ is a Hilbert space valued Wiener process with nuclear operator $Q_1$, $\circ~ d W_2(t)$ stands for the Stratonovich differential where $W_2$ is a real-valued Wiener processes, $\tilde{N}$ is a compensated Poisson random measure and $h$ is a bounded function. The tensor product $h \otimes \mathbf{\tau}(t)$ denotes usual matrix multiplication. Detailed and precise descriptions of the model are provided in the subsequent Sections.

In this paper, we study the existence and uniqueness of local (maximal) strong solutions to the incompressible, viscoelastic fluids of Oldroyd type \eqref{se01}-\eqref{se04} in both two and three dimensions. The analysis here is significantly different from the classical one for stochastic evolution equations due to the lack of diffusion in the $\tau$ equation \eqref{se02} and  structure of $\Q$. To be a little more precise, one of the key difficulties in proving local existence with diffusion only in the $\v$ equation stems from the nonlinear terms. Since $H^s$ is an algebra for $s>d/2$, so one obtains 
$$\vert \left\langle(\v\cdot\nabla)\mathbf{\tau}, \mathbf\varphi)\right\rangle_{H^s}\vert\leq \|\v\|_{H^s} \|\nabla\mathbf{\tau}\|_{H^s}\|\mathbf{\varphi}\|_{H^s}.$$
Thus we must estimate $\|\nabla\mathbf{\tau}\|_{H^s}$, and if we start with $\mathbf{\tau}_0\in H^s$ we do not have any control over the $H^s$ norm of $\nabla\mathbf{\tau}$ because there is no smoothing for $\mathbf{\tau}$.
Due to the same reasons, the semigroup method to mild solutions may not
work in this case and also the local $m$-accretivity property is not available due to the absence of a diffusive term. To the best of the authors knowledge, this work
appears to be the first systematic treatment for the existence and uniqueness of the local/maximal strong solution of the stochastic viscoelastic fluid of Oldroyd type in its most general form. Global well-posedness for smooth solutions with small initial data of the viscoelastic fluids under the influence of fluctuating external forces will be addressed by the authors in near future.

The organization of the present article is as follows. In the Section 2, we introduce some notation used throughout this paper and certain known but useful results concerning fractional order Sobolev spaces, commutator estimates and stochastic analysis. After stating the hypotheses on the random noise coefficients in Section 3, we provide a full statement of the main result of this work. In the very Subsection, we also briefly outline the strategy of the proof of the main result.  In Section 4, we consider an approximate system and establish an energy estimate in $H^s, s>d/2$. We also provide a probabilistic estimate of the stopping time. Section 5 is devoted in proving the (strong) convergence of the approximate solutions. The existence and uniqueness of local strong solution is provided in Section 6, and that of local maximal solution is proved in Section 7. In the Appendix we prove
several results which are used to establish the strong convergence of the approximate solutions in Section 5. 

\section{Preliminaries}
\subsection{Notations} 
Throughout the paper we use the same notation $H^{s}(\mathbb R^{d})$ for both vector-valued and tensor valued functions. For notational convenience we define
$$H^{s}(\mathbb R^{d}):=H^{s}(\mathbb R^{d};\mathbb R^{d})=(H^{s}(\mathbb R^{d}))^d \quad \mbox{and} \quad H^{s}(\mathbb R^{d}):=H^{s}(\mathbb R^{d};\mathbb R^{d \times d}).$$ 
Likewise, $\mathrm{L}^2(\R^d)$ is used for both vector-valued and tensor valued functions.
\subsection{Fractional Order Sobolev Spaces}
For $s \in \mathbb R,$ let  $J^s$ denote the Bessel potential of order $s$ which is equivalent to the operator $(I - \Delta)^{s/2}$, where $\Delta$ is the Laplace operator, and is defined via the Fourier transform $\mathcal F$ as follows
\[ {\mathcal F} \left[ J^{s} f \right](\xi) = (1 + |\xi|^2)^{s/2} \widehat f(\xi) .\]

\noindent
The inner product on $H^{s}(\mathbb R^{d})$ is given by
\begin{align*}
(f, g)_{H^{s}} = \left( (1 + |\xi|^2)^{s/2} \widehat f(\xi), (1 + |\xi|^2)^{s/2} \widehat g(\xi) \right)_{\mathrm{L}^2} &= ({\mathcal F} \left[ J^s f \right](\xi), {\mathcal F} \left[ J^s g \right](\xi))_{\mathrm{L}^2} \nonumber \\ &= \left( J^s f, J^s g \right)_{\mathrm{L}^2},
\end{align*}
and the norm on ${H}^{s}(\mathbb R^{d})$ is defined by
\begin{align}\label{hs}
\| f\|_{H^{s}} = \left( \int_{\mathbb R^{d}} \left[ (1 + |\xi|^2)^{s/2} |\widehat f(\xi)| \right]^2 \right)^{1/2} &= \left\| (1 + |\xi|^2)^{s/2} \widehat f(\xi)\right\|_{\mathrm{L}^2} = \| J^s f\|_{\mathrm{L}^2},
\end{align}
\begin{rem}\label{prohs}
If $s > d/2$, then each $f \in H^{s}(\mathbb R^{d})$ is bounded and continuous and hence 
\[ \|f\|_{\mathrm{L}^{\infty}(\mathbb R^{d})} \leq C \|f\|_{H^s(\mathbb R^{d})}, \left. \right. for\left. \right. s > d/2.\]
Also, note that $H^s$ is an algebra for $s > d/2$, i.e., if $f, g \in H^{s}(\mathbb R^{d})$, then $fg \in H^{s}(\mathbb R^{d})$, for $s > d/2$. Hence, we have
\[ \| fg\|_{H^s} \leq C \| f\|_{H^s} \| g\|_{H^s}, \left. \right. for\left. \right.s > d/2.\]
\end{rem}

\begin{rem}\label{div}
Fix $s > d/2$ and let $f,g \in H^{s}$ with $\nabla \cdot f=0.$ Then
\[ \| (f \cdot \nabla) g \|_{H^{s-1}} \leq C \|f\|_{H^s} \|g\|_{H^s}.\]
\end{rem}

\begin{proof}
Now $f$ is divergence free, $(f \cdot \nabla) g = \nabla \cdot(f \otimes g)$.
And $H^s$ is an algebra for $s >d/2$,
\[\| (f \cdot \nabla) g \|_{H^{s-1}} = \| \nabla \cdot(f \otimes g)\|_{H^{s-1}} \leq C \| f \otimes g\|_{H^s} \leq C \|f\|_{H^s} \|g\|_{H^s}.\]
\end{proof}

Now we state a more generalised Lemma of  Theorem 2.4.5 of Kesavan \cite{Ks}.
\begin{lem}\label{Si}
$($Sobolev Inequality$)$ For $f \in H^{s}(\mathbb R^{d})$, we have
\[ \|f\|_{L^{q}(\mathbb R^{d})} \leq C_{d, s, q} \|f\|_{H^s(\mathbb R^{d})}\]
provided that q lies in the following range
\begin{enumerate}
\item[(i)]
if $s < d/2$, then $2 \leq q \leq \frac{2d}{d - 2s}$,
\item[(ii)]
if $s = d/2$, then $2 \leq q < \infty$,
\item[(iii)]
if $s > d/2$, then $2 \leq q \leq \infty$.
\end{enumerate}
\end{lem}

\begin{rem}\label{r23}
We deduce the following result using Lemma \ref{Si} and this estimate will be useful in several calculations. In two dimensions, we exploit H\"{o}lder's inequality with exponents 2/$\epsilon$ and 2/(1-$\epsilon$), and Sobolev inequality for 0 $<$ $\epsilon$ $<$ $s-1$ to obtain
\[ \|fg\|_{\mathrm{L}^2} \leq \|f\|_{\mathrm{L}^{2/ {\epsilon}}}\|g\|_{\mathrm{L}^{2/1-{\epsilon}}} \leq C \|f\|_{H^{1-{\epsilon}}} \|g\|_{H^{\epsilon}} \leq C \|f\|_{H^1} \|g\|_{{H}^{s-1}}.\]
In three dimensions, we again exploit H\"{o}lder's inequality with exponents 6 and 3, and Sobolev inequality to obtain
\[ \|fg\|_{\mathrm{L}^2} \leq \|f\|_{\mathrm{L}^6} \|g\|_{\mathrm{L}^3}  \leq C \|f\|_{H^1} \|g\|_{H^{1/2}} \leq C \|f\|_{H^1} \|g\|_{{H}^{s-1}}.\]
Note that, for both the two and three dimensions, we obtain the same bounds.
\end{rem}

\begin{lem}\label{iss}
$($Interpolation in Sobolev spaces$)$. Given $s > 0,$ there exists a constant C depending on s, so that for all $f \in H^{s}(\mathbb R^{d})$ and $0 < s' < s,$
\[\|f\|_{H^{s'}} \leq C \|f\|_{\mathrm{L}^2}^{1-s'/s}\|f\|^{s'/s}_{H^{s}}.\]
\end{lem}

\noindent For details see Theorem 9.6, Remark 9.1 of Lions and Magenes \cite{LM}.

\subsection{Fourier Truncation Operator} \label{FTO}
 Let us define the Fourier truncation $\mathcal{J}_{n}$
as follows:
$$\widehat{\mathcal{J}_{n}f}(\xi)=\mathlarger{\mathbf{1}}_{B(0,n)}(\xi)\widehat{f}(\xi),$$ where $B(0,n)$, a  ball of radius
$n$ centered at the origin and $\mathlarger{\mathbf{1}}_{B(0,n)}$ is
the indicator function. We list the following properties of $\mathcal{J}_n$ [see Chemin \cite{C}, Fefferman et.
al.\cite{FMRR}, Manna et al. \cite{MaMo}].

\begin{align}
& 1.\qquad \|\mathcal{J}_{n}f\|_{H^s(\mathbb{R}^d)} 
\leq \|f\|_{H^s(\mathbb{R}^d)}.\label{intro1}\\
& 2. \qquad \|\mathcal{J}_{n}f-f\|_{H^s}\leq
c\left(\frac{1}{n}\right)^{k}\|f\|_{H^{s+k}}.\label{intro2a}\\
& 3. \qquad \|(\mathcal{J}_{n}-\mathcal{J}_{m})f\|_{H^s(\mathbb{R}^d)}
\leq
\max\left\{\left(\frac{1}{n}\right)^{k},\left(\frac{1}{m}\right)^{k}\right\}\|f\|_{H^{s+k}(\mathbb{R}^d)}.\label{sr}
\end{align}

\subsection{Commutator Estimates} \label{comm.est}
Let us now state the celebrated commutator estimate due to Kato and Ponce
\cite{KaPo} (Lemma XI).
\begin{lem} If $s>0$ and
$1<p<\infty$, then
\begin{align}\label{kato}\|J^s(fg)-f(J^sg)\|_{\mathrm{L}^p}\leq C_p\left(\|\nabla f\|_{\mathrm{L}^{\infty}}\|J^{s-1}g\|_{\mathrm{L}^p}+\|J^sf\|_{\mathrm{L}^p}\|g\|_{\mathrm{L}^{\infty}}\right).\end{align}
\end{lem}

\begin{rem} \label{rem_kato} From (\ref{kato}), it can be easily seen that for
$s> 0$ and $1<p<\infty$, the nonlinear term satisfy the estimate:
\begin{align}\label{kato1}
\|J^s[(f\cdot\nabla)g]-(f\cdot\nabla)(J^s g)\|_{\mathrm{L}^p}\leq
C_p\left(\|\nabla f\|_{\mathrm{L}^{\infty}}\|J^{s-1}\nabla g\|_{\mathrm{L}^p}+\|J^sf\|_{\mathrm{L}^p}\|\nabla g\|_{\mathrm{L}^{\infty}}\right).
\end{align}
\end{rem}
\par
\noindent
Assuming the fact that $f$ is divergence free, we have
$((f\cdot\nabla)(J^sg),J^sg)_{\mathrm{L}^2}=0.$ Hence for $p=2$, we have the following estimate.
\begin{cor}\label{cor_kato}
For $s>0$, there exists a constant $c=c(d,s)$ such that, for all
$ f,g\in H^s(\mathbb{R}^d)$ and $\nabla\cdot f=0$, we have
\begin{align}
|(J^s[(f\cdot\nabla)g],J^s g)_{\mathrm{L}^2}|\leq
c\left(\|\nabla f\|_{\mathrm{L}^{\infty}}\|g\|_{H^s}+\|f\|_{H^s}\|\nabla g\|_{\mathrm{L}^{\infty}}\right)\|g\|_{H^s}.
\end{align}
\end{cor}
\begin{rem}\label{kato10}
For $s>d/2+1$, we have $\|\nabla f\|_{\mathrm{L}^{\infty}}\leq
c\|\nabla f\|_{H^{s-1}}$ and hence we get
\begin{align}
|(J^s[(f\cdot\nabla)g],J^sg)_{\mathrm{L}^2}|&\leq
c\left(\|\nabla f\|_{H^{s-1}}\|g\|_{H^s}+\|f\|_{H^s}\|\nabla g\|_{H^{s-1}}\right)\|g\|_{H^s}\nonumber\\&\leq
c\|f\|_{H^s}\|g\|_{H^s}^2.
\end{align}
\end{rem}

The next result is a partial generalization of the commutator
estimates of Kato and Ponce \cite{KaPo} given in Theorem $1.2$ of Fefferman et al.
\cite{FMRR}.
\begin{thm} \label{com_thm}
Given $s>d/2$, there is a constant $c=c(d,s)$ such that, for all
$f,g$ with $\nabla f,g\in H^s(\mathbb{R}^d)$ such that $\nabla \cdot f=0,$ we have
\begin{align}
\|J^s\left[(f\cdot\nabla)g\right]-(f\cdot\nabla)(J^sg)\|_{\mathrm{L}^2}\leq
c\|\nabla f\|_{H^s}\|g\|_{H^s}.
\end{align}
\end{thm}
\begin{cor}[Corollary 2.1, Fefferman et al. \cite{FMRR}] \label{com_cor}
Given $s>d/2$, there is a constant $c=c(s,d)$ such that, for all
$f,g$ with $\nabla f,g\in H^s(\mathbb{R}^d)$ and
$\nabla\cdot f=0$, we have
\begin{align}
\left|\left(J^s\left[(f\cdot\nabla)g\right],J^sg\right)_{\mathrm{L}^2}\right|
\leq c\|\nabla f\|_{H^s}\|g\|_{H^s}^2.
\end{align}
\end{cor}

\begin{lem} { Variants of Commutator estimates:}
Let $s>0$ and $1<p<\infty$ and  $ \mathlarger{p_2} ,p_3 \in (1, \infty)$ 
be such that 
$ \mathlarger{\frac{1}{p} \geq \frac{1}{p_1}+\frac{1}{p_2}},$ \qquad $\mathlarger{\frac{1}{p}} \geq \mathlarger{\frac{1}{p_3}+\frac{1}{p_4}}.$
Then
\begin{align}\label{kato.vari}
\|J^s(fg)\|_{\mathrm{L}^p}\leq C \left(\|
f\|_{\mathrm{L}^{p_1}}\|J^{s}g\|_{\mathrm{L}^{p_2}}+\|J^sf\|_{\mathrm{L}^{p_3}}\|g\|_{\mathrm{L}^{p_4}}\right).
\end{align}
\end{lem}
For details, see Lemma $2.6$ of Bessaih and Ferrario \cite{BF}.

\begin{property} \label{propQ}
As a consequence of the commutator estimate \eqref{kato}, the bilinear map $\Q$ satisfies the following (tame) estimate,
$$\|\Q(\mathbf{\tau},\nabla \v)\|_{H^s} \leq C \left( \|\mathbf{\tau}\|_{\mathrm{L}^\infty}\|\nabla \v\|_{H^s}+ \|\nabla \v\|_{\mathrm{L}^\infty}\|\mathbf{\tau}\|_{H^s}\right.).$$
\end{property}

\subsection{Basic Stochastic Analysis} \label{stoc.an}

In this Subsection we are going to introduce some definitions and properties of the Hilbert space valued stochastic processes. For further details, one can refer M\'{e}tivier
\cite{Me}, {Da Prato and Zabczyk \cite{DaZ}, Gawarecki and Mandrekar \cite{GaMa}.

Let $U$ and $H$ be two separable Hilbert spaces. A nonnegative operator $Q\in\mathcal{L}(U,U)$ is of trace
class if and only if
for an orthonormal basis $\{e_j\}$ on $U,$ $$\mathlarger{\sum_{j=1}^{\infty}}(Qe_j,e_j)_{U}<\infty.$$
A bounded linear operator $Q: U\to H$ is said to be
Hilbert-Schmidt if $\mathlarger{\sum_{k=1}^{\infty}}\|Q
e_k\|_{H}^2<\infty.$ The set $\mathcal{L}_2(U,H)$ of all
Hilbert-Schmidt operators from $U$ into $H$, equipped with the norm
$\|Q\|_{\mathcal{L}_2(U,H)}=\left(\mathlarger{\sum_{k=1}^{\infty}}\|Q
e_k\|_{H}^2\right)^{\frac{1}{2}}$ is a separable Hilbert space. \newline
Further assume that $Q$ is symmetric, positive, trace class
operator on $U.$ \newline
Let $\mathcal{L}_Q(U,H)=\mathcal{L}_2(U_0,H)$ denote the space of
all Hilbert-Schmidt operator from $U_0$ to $H$ where $U_0 = Q^{\frac{1}{2}}U.$


\qquad \qquad \newline
Let $\mathbb{M}$ be the totality of non-negative (possibly infinite)
integral valued measures on $(H,\mathscr{B}(H))$ and
$\mathscr{B}_{\mathbb{M}}$ be the smallest $\sigma$-field on
$\mathbb{M}$ with respect to which all $N\in\mathbb{M}\to N(B)\in
\mathbb{Z}^+\cup\{\infty\}, B\in\mathscr{B}(H)$, are measurable.

\begin{defi}
An $(\mathbb{M},\mathscr{B}(\mathbb{M}))$-valued random variable $N$
is called a \textit{Poisson random measure}
\begin{enumerate}
\item if for each $B\in\mathscr{B}(H), N(B)$ is Poisson distributed.
i.e., $\mathbb{P}(N(B) = n)
=\mathlarger{\frac{\eta(B)e^{-\eta(B)}}{n!}},n=0,1,2,\ldots,$ where
$\eta(B)=\mathbb{E}(N(B)), B\in\mathscr{B}(H);$
\item if $B_1,B_2,\ldots,B_n\in\mathscr{B}(H)$ are disjoint, then
$N(B_1), N(B_2),\ldots,N(B_n)$ are mutually independent.
\end{enumerate}
\end{defi}

\begin{defi}
A c\`{a}dl\`{a}g process $(\mathbf{X}_t)_{t\geq 0}$ is a stochastic process  for which the paths $t\mapsto \mathbf{X}(t)$ are right continuous with
left limits everywhere,with probability one. \newline
$(\mathbf{X}_t)_{t\geq 0}$, is called a L\'{e}vy
process if it has stationary independent increments and is
stochastically continuous.
\end{defi}

For a $H$-valued L\'{e}vy process $(\mathbf{X}_t)_{t\geq 0},$  let us define $$N(t, Z)= N(t, Z,
\omega)= \#\left\{s\in (0, \infty): \triangle\mathbf{X}_s(\omega)\in
Z\right\}, t>0, Z\in\mathscr{B}(H\backslash\{0\}), \omega\in\Omega$$
as the \emph{Poisson random measure associated with the L\'{e}vy
process} where $\triangle \mathbf{X}_t(\omega)(=\mathbf{X}_t(\omega) - \mathbf{X}_{t-}(\omega))$ denotes the corresponding jump for every
$\omega\in\Omega$.

The differential form of the measure $N(t,Z,\omega)$ is written as
$N(dt, dz)(\omega)$. We call $\tilde{N}(dt, dz) = N(dt, dz) -
\lambda(dz)dt $ a \emph{compensated Poisson random measure (cPrm)},
where $\lambda(dz)dt $ is known as \emph{compensator} of the
L\'{e}vy process $(\mathbf{X}_t)_{t\geq 0}$. Here $dt$ denotes the
Lebesgue measure on $\mathscr{B}(\mathbb{R}^{+})$, and $\lambda(dz)$
is a $\sigma$-finite L\'{e}vy measure on $(Z, \mathscr{B}(Z))$.

\begin{defi} [see Mandrekar and R\"{u}diger \cite{MR}]
Let $H$ and $F$ be separable Hilbert spaces. Let $F_t:=
\mathscr{B}(H)\otimes\mathscr{F}_t$ be the product $\sigma$-algebra
generated by the semi-ring $ \mathscr{B}(H)\times\mathscr{F}_t$ of
the product sets $Z\times F,$ $Z\in\mathscr{B}(H),$ $F\in
\mathscr{F}_t$ (where $\mathscr{F}_t$ is the filtration of the
additive process $(\mathbf{L}_t)_{t\geq 0}$). Let $T>0$, define
\begin{align*}
\mathbb{H}(Z) = & \big\{g : \mathbb{R}^+ \times Z \times \Omega
\rightarrow F,\text{ such that g is}\;\; F_T/\mathscr{B}(F) \;\;
measurable\;\;and \nonumber\\& \;\;\qquad g(t,z,\omega)\;\;is\;\;
\mathscr{F}_t - adapted\;\;\forall z\in Z, \forall t\in (0,T]\big\}.
\end{align*}
For $p\geq1$, let us define,
$$\mathbb{H}^{p}_\lambda([0,T]\times Z;F) = \left\{g\in \mathbb{H}(Z) :
 \int_0^T\int_Z\mathbb{E}[\|g(t,z,\omega)\|^p_F]\lambda(dz)dt < \infty \right\}.$$
\end{defi}

Let us denote $D([0, T ];H)$ as the space of all c\`adl\`ag paths from $[0, T ]$
into Hilbert space $H.$ The space
$D([0,T];H)$ is endowed with the Skorokhod $J$-topology.
For more details see Chapter
2, Metivier \cite{Me} and Chapter 3, Billingsley \cite{BiP}.

\begin{defi}[Quadratic variation process and Meyer process] Let $M$ be a square integrable martingale with right continuous paths with values in a separable Hilbert space $H$  on $(\Omega,\mathscr{F},(\mathscr{F}_{t})_{t\geq 0},\mathbb{P}).$ Then there exists two real right continuous
increasing processes $[M]$ and $\triangleleft M\triangleright$ with
$0=[M]_0=\triangleleft M\triangleright_0$ such that
\begin{align}
\|M_t\|_H^2=\|M_0\|_H^2+2\int_0^t(M_{s-},\d M_s)_H+[M]_t.
\end{align}
$\triangleleft M\triangleright$ is the unique real right continuous
increasing predictable process such that
\begin{align}
\|M_t\|_H^2-\|M_0\|_H^2-\triangleleft M\triangleright_t \textrm{ is
a martingale.}
\end{align}
Here $[M]$ is called the quadratic variation of $M$ and
$\triangleleft M\triangleright$ the Meyer process of $M$.
\end{defi}
\begin{rem}
If $M$ is continuous, then we have $\triangleleft
M\triangleright=[M]$.
\end{rem}

\begin{rem} \label{nlam}  [Section 2.3, M\'{e}tivier \cite{Me})] 
Let $H$ be Hilbert spaces and let $Q:H\to H$ be a trace class
operator. Let $$\d u(t)=\sigma(t,u)\d
W(t)+\mathlarger{\int_Z}g(u(t-),z)\tilde{N}(\d t,\d z),$$ where
$W(\cdot)$ is a $H$-valued Wiener process,
$\sigma(\cdot,\cdot):[0,t)\times H\to\mathcal{L}_Q(H,H)$, $Z$ is a
measurable subspace of $H$, $g(\cdot,\cdot) :H\times Z\to H$ and
$\tilde{N}(\cdot,\cdot)$ is the compensated Poisson random measure.
Then, 
\begin{align} \label{mt1}
 \mathbb{E}\left[\int_0^t\|\sigma(s,u)\|_{\mathcal{L}_Q(H,H)}^2\d
s+\int_0^t\int_Z\|g(u,z)\|_H^2N(\d s,\d
z)\right] \notag\\
=\mathbb{E}\left[\int_0^t\|\sigma(s,u)\|_{\mathcal{L}_Q(H,H)}^2\d
s+\int_0^t\int_Z\|g(u,z)\|_H^2\lambda(\d z)\d s\right].
\end{align}
\end{rem}

\begin{lem}(Burkholder-Davis-Gundy Inequality)\label{burk}
Let $M$ be a Hilbert space valued c\`{a}dl\`{a}g martingale with
$M_0=0$ and let $p\geq 1$ be fixed. Then for any
$\mathscr{F}$-stopping time $\mathbf{\tau}$, there exists constants $c_p$ and
$C_p$ such that
$$\mathbb{E}\left\{[M]_{\mathbf{\tau}}^{p/2}\right\}\leq c_p\mathbb{E}\left\{\sup_{0\leq t\leq\mathbf{\tau}}\|M_t\|_H^p\right\}\leq C_p
\mathbb{E}\left\{[M]_{\mathbf{\tau}}^{p/2}\right\}$$ for all $\mathbf{\tau}$, $0\leq
\mathbf{\tau}\leq \infty$, where $[M]$ is the quadratic variation of process
 $M$. The constants
are universal (independent of $M$).
\end{lem}
\noindent For proof see Theorem $1.1$ of Marinelli and R\"{o}ckner
\cite{MCRM}. For real-valued c\`{a}dl\`{a}g martingales see Theorem
$3.50$ of Peszat and Zabczyk \cite{PZ}.

\section{The Stochastic Model and Statement of the Main Results}
Assume that $\big(\Omega, \mathcal{F}, \mathbb{F}, \mathbb{P}\big)$ is a filtered probability space, where $\mathbb{F}=\big(\mathcal{F}_t)_{t\geq 0}$ is the filtration, and this probability space satisfies the so called usual conditions, i.e.
	\begin{trivlist}
		\item[(i)] $\mathbb{P}$ is complete on $(\Omega, \mathcal{F})$,
		\item[(ii)] for each $t\geq 0$, $\mathcal{F}_t$ contains all $(\mathcal{F},\mathbb{P})$-null sets,
		\item[(iii)] the filtration $\mathbb{F}$ is right-continuous.
	\end{trivlist}
We consider the following stochastic viscoelastic equations \eqref{e1}-\eqref{e4}:
\begin{align}
&d \mathbf{v}(t)+[(\mathbf{v}(t)\cdot\nabla)\v(t)-\nu\Delta \v(t)+\nabla
p]dt =\mu_1 \nabla \cdot \mathbf{\tau}(t)dt+\sigma(t,\v(t))\d W_1(t)+\int_Z G(\v(t-), z)\tilde{N}_{1} (\d t,\d z),     \label{se1}\\
&d \mathbf{\tau}(t)+[(\v(t) \cdot \nabla)\mathbf{\tau} (t) +a\mathbf{\tau}(t) +\Q(\mathbf{\tau}(t), \nabla \v(t))]dt=\mu_2 \mathcal{D}(\v(t))dt+ (h \otimes \mathbf{\tau}(t))\circ d W_2(t) , \label{se2}\\
&\nabla\cdot\v=0 ,\label{se3}\\
&\v(0,\cdot)=\v_0,\,\,\mathbf{\tau}(0,\cdot)=\mathbf{\tau}_0,\label{se4}
\end{align}
where $\,W_1 $ is $H^s$-valued Wiener process with nuclear operator $Q_1$ and $(h \otimes \mathbf{\tau}(t))\circ d W_2(t)$ is understood in Stratonovich sense and $W_2=(W_2(t))_{t \geq 0}$ is a real-valued Wiener processes on $\left(\Omega,\mathscr{F},(\mathscr{F}_t)_{t\geq
0},\mathbb{P}\right)$. $h$ is an element of $\mathrm{L}^\infty (\mathbb{R}^{d \times d})$. The tensor product $h \otimes \mathbf{\tau}(t)$ denotes usual matrix multiplication, i.e.,
$(h \otimes \mathbf{\tau}(t))_{i,j}=\sum _{k=1}^{d} h_{ik} \mathbf{\tau} _{kj}(t)\quad ;\,\,i,j=1,2, \cdots ,d.$
$\tilde{N} (\d t,\d z)$ is a compensated Poisson random measure. Further it is assumed that $W_1,W_2$ and $\tilde{N}$ mutually independent processes.

Define a linear operator $\mathcal{S}$ from $\mathrm{L}^2$ into itself by
$\mathcal{S}(\mathbf{\tau})=h \otimes \mathbf{\tau}.$ Note that $\mathcal{S}$ is bounded and satisfies 
$\|\mathcal{S}(\mathbf{\tau})\|_{\mathrm{L}^2}=\|h \otimes \mathbf{\tau}\|_{\mathrm{L}^2} \leq \|h\|_{\mathrm{L}^{\infty}}\|\mathbf{\tau}\|_{\mathrm{L}^2}.$

Using the relation between Stratonovich and It\^{o} differentials (e.g. see Kuo \cite{Kuo}, page 122 ) one has
$$\mathcal{S}(\mathbf{\tau}) \circ dW_2=\frac{1}{2} \mathcal{S}^2(\mathbf{\tau}) dt+ \mathcal{S}(\mathbf{\tau}) dW_2,$$
where $\mathcal{S}^2(\mathbf{\tau})=\mathcal{S} \circ \mathcal{S}(\mathbf{\tau})=\mathcal{S} \circ (h \otimes \mathbf{\tau})= h \otimes(h \otimes \mathbf{\tau}),$ for any $\mathbf{\mathbf{\tau}} \in \mathrm{L}^2$ and $\circ $ is usual composition notation.
Further note that 
\begin{align} \label{propS.2}
\|\mathcal{S}^2(\mathbf{\tau})\|_{\mathrm{L}^2} \leq \|h\|_{\mathrm{L}^\infty} \|h \otimes \mathbf{\tau}\|_{\mathrm{L}^2} \leq \|h\|^2_{\mathrm{L}^{\infty}} \|\mathbf{\tau}\|_{\mathrm{L}^2}.
\end{align}

\begin{rem} \label{esti.S}
For $s > d/2$, let $h$ be an element of $H^s(\mathbb{R}^{d \times d}) \subset \mathrm{L}^\infty (\mathbb{R}^{d \times d}),$ and we extend the definition of $\mathcal{S}$ from $H^s$ into itself by $\mathcal{S}(\mathbf{\tau})=h \otimes \mathbf{\tau}$ and $\|\mathcal{S}(\mathbf{\tau})\|_{H^s} \leq \|h\|_{H^s} \|\mathbf{\tau}\|_{H^s}$ and $\|\mathcal{S}^2(\mathbf{\tau})\|_{H^s} \leq \|h\|_{H^s} \|h \otimes \mathbf{\tau}\|_{H^s} \leq \|h\|^2_{H^s}  \|\mathbf{\tau}\|_{H^s}.$
\end{rem}

\subsection{Analysis of the Noise Terms}

Let us assume the following properties of noise co-efficients $\sigma$ and $G$ namely joint continuity, linear growth and Lipscitz condition.
\begin{assum} \label{hypo}
For all $s\geq 0,$ the noise co-efficient $\sigma$ and $G$ satisfy
\begin{itemize}
    \item[(A.1)]  The function $\sigma \in C([0, T] \times H^{s}; \mathcal{L}_Q(\mathrm{L}^2,H^s))$
    and $G \in \mathbb{H}^2_{\lambda}([0, T] \times Z;
    H^s(\mathbb{R}^d))$.
    \item[(A.2)]  (Growth Condition)
For all $\v \in H^s(\mathbb{R}^d)$ and for all $t\in[0,T]$, there
exist positive constant $K$ such that
    $$\|\sigma(t, \v)\|^2_{\mathcal{L}_Q(\mathrm{L}^2,H^s)} + \int_{Z} \|G(\v, z)\|^2_{H^s}\lambda(\d z)
    \leq K \left(1 +\|\v\|_{H^s}^2 \right).$$

    \item[(A.3)]  (Lipschitz Condition)
    For all $t \in [0, T]$ and for all $\v_1, \v_2,\in H^s(\mathbb{R}^d)$,  there exist positive constant $L$ such that
    \begin{align*}
    \|\sigma(t, \v_1) - \sigma(t,
    \v_2)\|^2_{\mathcal{L}_Q(\mathrm{L}^2,H^s)}&
    + \int_{Z} \|G(\v_1, z)-G(\v_2, z)\|^2_{H^s}\lambda(\d z)
    \\&\leq L \|\v_1 -
    \v_2\|_{H^s}^2.
    \end{align*}
\end{itemize}
\end{assum}

\begin{rem}
	For $m>n$, by using (\ref{sr}), one can note that for any $l\geq 0$ and $\epsilon>0$,
	\begin{align}\label{sr2}
	\|(\mathcal{J}_n-\mathcal{J}_{m})\sigma(t,\v)\|_{\mathcal{L}_Q(\mathrm{L}^2,H^l)}^2&=
	\sum_{j=1}^{\infty}\lambda_j\|(\mathcal{J}_n-\mathcal{J}_{m})\sigma(t,\v)e_j\|_{H^l}^2\nonumber\\&
	\leq\frac{C}{n^{\epsilon}}\sum_{j=1}^{\infty}\lambda_j\|\sigma(t,\v)e_j\|_{H^{l+\epsilon}}^2
	=\frac{C}{n^{\epsilon}}
	\|\sigma(t,\v)\|_{\mathcal{L}_Q(\mathrm{L}^2,H^{l+\epsilon})}^2.
	\end{align}
\end{rem}
We now introduce the concept of local strong solution and maximal local strong solution of (\ref{se1})-(\ref{se4}). Throughout we will assume that $T$ is a fixed positive number.

\begin{defi}[Local Strong Solution]
We say that the triplet $(\v,\mathbf{\tau},\rho_{\infty})$ is a \textup{local strong
(path-wise) solution} to \eqref{se1}-\eqref{se2} if
\begin{enumerate}
\item [(i)] the symbol $\rho_{\infty}$ is a stopping time such that $\rho_{\infty}\leq T$ a.s., and there exists a non-decreasing sequence $\{\rho_N, N\in\mathbb{N}\}$ of stopping times with $\rho_N \uparrow\rho_{\infty}$  a. s. as $N\uparrow\infty$,
\item [(ii)] for $s>d/2$ and $t\in [0,\rho_{\infty})$, the symbols $\v$ and $\mathbf{\tau}$ denote progressively measurable stochastic processes such
that 
$$\v\in\mathrm{L}^2(\Omega;D(0,t;H^s(\mathbb{R}^d))\cap\mathrm{L}^2(0,t;H^{s+1}(\mathbb{R}^d))) \quad \mbox{with} \quad \nabla \cdot \v=0$$  $$ \quad \mbox{and} \quad \tau\in\mathrm{L}^2(\Omega;C(0,t;H^s(\mathbb{R}^d))).$$ 
Moreover, for any $t\in [0, T]$, $N\in\mathbb N$, and for any $\phi_i \in H^s(\mathbb{R}^d)\,;\,i=1,2$ with $\nabla\cdot \phi_1 =0,$ $\v$ and $\tau $ satisfy the following equations with probability 1:
\begin{align} \label{exist.1.def}
\left(J^s\v(t \wedge \rho_N),J^s\phi_1\right)_{\mathrm{L}^2}&=\left(J^s\v_0,J^s\phi_1\right)_{\mathrm{L}^2}+\int_0^{t \wedge \rho_N} \left(\nu\Delta
J^s\v-J^s[(\v\cdot\nabla)\v]+\mu_1\nabla\cdot J^s\tau,J^s \phi_1\right)_{\mathrm{L}^2}\d s\nonumber\\
&\quad+\int_0^{t \wedge \rho_N} \left(J^s\sigma(s,\v(s))\d W_1(s),J^s\phi_1\right)_{\mathrm{L}^2}\nonumber\\
&\quad+\int_0^{t \wedge \rho_N} \int_Z\left(J^sG(\v(s-),z),J^s\phi_1\right)_{\mathrm{L}^2}\tilde{N}_1(\d s,\d z),
\end{align}
\begin{align}
\left(J^s\tau(t \wedge \rho_N),J^s\phi\right)_{\mathrm{L}^2}&=\left(J^s\tau_0,J^s\phi_2\right)_{\mathrm{L}^2}
+\mu_2 \int_0^{t \wedge \rho_N} \left( J^s\mathcal{D}(\v),J^s\phi_2\right)_{\mathrm{L}^2}\d s\nonumber\\
&\quad-\int_0^{t \wedge \rho_N} \left(J^s[(\v\cdot\nabla)\tau]+J^s \Q(\tau,\nabla\v)+aJ^s\tau,J^s\phi_2\right)_{\mathrm{L}^2}\d s\nonumber\\
&\quad+\frac{1}{2}\int_0^{t \wedge \rho_N} \left(J^s \mathcal{S}^2(\tau),J^s\phi_2\right)_{\mathrm{L}^2}\d s+\int_0^{t \wedge \rho_N}\left(J^s \mathcal{S}(\tau(s)),
J^s\phi_2\right)_{\mathrm{L}^2}\d W_2(s).\label{exist.2.def}
\end{align}
\end{enumerate}
\end{defi}

\begin{defi}[Maximal Local Strong Solution]
Let $(\v,\tau, \rho_{\infty})$ be a local solution to \eqref{se1}-\eqref{se2} such that
$$\sup_{0\leq s\leq \rho_{\infty}}\|\v(s)\|^2_{H^s}+\sup_{0\leq s\leq \rho_{\infty}}\|\tau(s)\|^2_{H^s}+\int_0^{\rho_{\infty}}\|\nabla\v(s)\|^2_{H^s}\d s=\infty,$$
 on the set $\{\omega\in\Omega:\rho_{\infty}(\omega)\leq T\}$, then the local process $(\v,\tau, \rho_{\infty})$ is called a maximal local solution. If $\rho_{\infty}<T$, the stopping time $\rho_{\infty}$ is called the explosion time of the stochastic processes $(\v,\tau)$.\newline
 A maximal local solution $(\v^1,\tau^1, \rho^1_{\infty})$ is said to be unique if for any other maximal local solution $(\v^2,\tau^2, \rho^2_{\infty})$, we have $\rho^1_{\infty}=\rho^2_{\infty}$ and $\v^1(t)=\v^2(t), \tau^1(t)=\tau^2(t)$ for any $0\leq t\leq \rho^1_{\infty}$ with probability 1.
\end{defi}
\subsection{Statement of the Main Result}
 The main results of this work are stated below and are proven in the subsequent Sections.
\begin{main}\label{MR}
	Let $(\Omega,\mathscr{F},(\mathscr{F}_t)_{t\geq 0},\mathbb{P})$ be a
	given filtered probability space and let
	$\v_0,\mathbf{\tau}_0\in\mathrm{L}^2(\Omega;H^s(\R^d))$ with $\nabla\cdot \v_0 =0,$ $s>d/2$, be
	$\mathscr{F}_0$-measurable. 
	 Then there exists a unique local in
	time strong solution $(\v,\mathbf{\tau},\rho_N)$ to the stochastic viscoelastic system \eqref{se1}-\eqref{se4}, 
where \begin{align*} 
	\rho_N=\inf_{t\geq
		0}\left\{t:\mu_2\|\v(t)\|_{H^s}^2+\mu_1\|\mathbf{\tau}(t)\|_{H^s}^2
	+2 \mu_2\nu \int_0^t \| \nabla\v (r)\|_{H^s}^2dr>N\right\},
	\end{align*}
	such that
	\begin{enumerate}
	 \item $\mathbb{E}\left[\mathlarger{\sup_{0\leq t\leq
			T\wedge\rho_N}}\mu_2\|\v(t)\|^2_{H^s}+\mathlarger{\sup_{0\leq t\leq
			T\wedge\rho_N}}\mu_1\|\mathbf{\tau}(t)\|^2_{H^s}+2\mu_2\nu\mathlarger{\int_0^{T\wedge\rho_N}}\|\nabla\v(t)\|^2_{H^s}\d
	t\right]<\infty$ for $T>0$,
	\item $\rho_N$ is a predictable strictly positive stopping time
	satisfying 
	\begin{align*}
	\mathbb{P}\left(\rho_N>\delta\right)
	\geq 1-2\delta e^{(\tilde{C}+C_2 \delta)} \Big(2\mathbb{E}\left(\mu_2\|\v_0\|_{H^s}^2+\mu_1\|\mathbf{\tau}_0\|_{H^s}^2\right)+18K\mu_2\delta\Big)	
	\end{align*}
	for any $\delta\in (0,1)$, and for some positive constant $\tilde{C}$ independent of $\delta$.
	\item $\v\in\mathrm{L}^2(\Omega;\mathrm{L}^{\infty}(0,T\wedge\rho_N; H^s(\R^d))\cap\mathrm{L}^2(0,T\wedge\rho_N; H^{s+1}(\R^d)))$,
	
	and $\indent\mathbf{\tau}\in\mathrm{L}^2(\Omega;\mathrm{L}^{\infty}(0,T\wedge\rho_N; H^s(\R^d))),$
	\item paths of the $\mathscr{F}_t$-adapted processes
	$(\v,\rho_N)$ and $(\mathbf{\tau},\rho_N)$ are c\`{a}dl\`{a}g and continuous respectively.
	\end{enumerate}
	Moreover, there exists a unique maximal local strong solution $(\v,\mathbf{\tau},\rho_{\infty})$ to the system \eqref{se1}-\eqref{se4} such that 
	\begin{align*}
	\mathbb{P}\left(\rho_{\infty}>\delta\right)
\geq 1-2\delta e^{(\tilde{C}+C_2 \delta)} \Big(2\mathbb{E}\left(\mu_2\|\v_0\|_{H^s}^2+\mu_1\|\mathbf{\tau}_0\|_{H^s}^2\right)+18K\mu_2\delta\Big)	
	\end{align*}
	where $\rho_{\infty}(\omega):=\lim_{N \rightarrow \infty} \rho_{N}(\omega)$ for almost all $\omega \in \Omega.$
\end{main} 

\subsection{Strategy of the Proof}
We prove the main results in a few steps, which are outlined below.
\par
\noindent
\emph{Step (i)} We first show that the solutions $(\v_n,\mathbf{\tau}_n)$ of smoothed version of the
(Fourier) truncated stochastic Oldroyd system \eqref{se1}-\eqref{se4} exist and the
$H^s$-norm of $(\v_n,\mathbf{\tau}_n)$ are uniformly bounded up to  a stopping
time $\rho_N^n$. We also provide a probabilistic estimate of the stopping time $\rho_N^n$ (see Theorem \autoref{positive1}).
\par
\noindent
\emph{Step (ii)} We then show that family of strong solutions $\{(\v_n,\mathbf{\tau}_n)\}_{n\in\mathbb N}$ is Cauchy in
$\mathrm{L}^2(\Omega;\mathrm{L}^{\infty}(0,T\wedge\xi_N;\mathrm{L}^2(\R^d)))$, where $\xi_N :=\lim_{n\to\infty} \rho_n^N$ (see Theorem \autoref{cauchy} and Remark \autoref{stopping}).
\par
\noindent
\emph{Step (iii)} By using Sobolev interpolation, we prove
$(\v_n,\mathbf{\tau}_n)\to(\v,\mathbf{\tau})$ strongly in
$\mathrm{L}^2(\Omega;\mathrm{L}^{\infty}(0,T\wedge\xi_N;H^{s'}(\R^d)))$
for any $0<s'<s$.  Then we show in Theorem \ref{existence}:
\begin{enumerate}
	\item [(a)] $\rho_N$ as the pointwise limit of $\rho_N^n$ and identify $\xi_N$ as $\rho_N$.
	\item [(b)] $(\v,\mathbf{\tau})$ solve \eqref{se1}-\eqref{se4} as an
	equality in
	$\mathrm{L}^1(\Omega;\mathrm{L}^2(0,T\wedge\rho_N;H^{s'-1}(\R^d)))$,\item
	[(c)] $(\v,\mathbf{\tau},\rho_N)$ is a local in time strong solution such that
	\begin{align*}&\v\in\mathrm{L}^{2}(\Omega;\mathrm{L}^{\infty}(0,T\wedge\rho_N;H^s(\R^d))\cap\mathrm{L}^2(0,T\wedge\rho_N;H^{s+1}(\R^d))),\\
	&\mathbf{\tau}\in\mathrm{L}^{2}(\Omega;\mathrm{L}^{\infty}(0,T\wedge\rho_N;H^s(\R^d))),\end{align*}
	\item [(d)] the $\mathscr{F}_t$-adapted paths of $(\v,\rho_N)$ and $(\mathbf{\tau},\rho_N)$ are
	c\`{a}dl\`{a}g and continuous respectively.
	\end{enumerate}
Finally, in Theorem \ref{uniqueness}, we show that $(\v,\mathbf{\tau},\rho_N)$ is a unique local strong solution.
\par\noindent
\emph{Step (iv)} We then prove, in Theorem \ref{max},
the existence of a unique maximal local strong solution
$(\v,\mathbf{\tau},\rho_{\infty})$ to \eqref{se1}-\eqref{se4} using stopping time arguments, and provide a probabilistic estimate of $\rho_{\infty}$.
\par
Having proved the uniform estimates of $\v_n$ and $\mathbf{\tau}_n$ in $Step ~ (i)$, we could use the classical compactness theorem of Aubin and Lions to extract a subsequence $(\v_{n_k}, \mathbf{\tau}_{n_k})$ that converges strongly to $(\v, \mathbf{\tau})$ in some sense. While this approach is natural for bounded domain, on the whole space one only obtains the requisite strong convergence on compact subsets, and one must then show that the nonlinear terms converge as required. }

Our approach also deviates from the one due to Motyl \cite{Mot} where martingale solution of three dimensional Navier-Stokes equations in unbounded domains has been proved using compactness method (by certain generalisation of the classical Dubinsky Theorem) and Jakubowski version of the Skorokhod Theorem for nonmetric spaces. Recent papers \cite{NV}, \cite{Kim} on stochastic Euler equation discusses the probabilistic estimate of stopping times.

We will present complete details of the proof in the remaining Sections.

\section{Truncated Stochastic Model}
We get the truncated  stochastic equations of \eqref{se1}-\eqref{se4}  on $\mathbb{R}^d$ as:
\begin{align}
&d \mathbf{v}_n=\left[\nu \Delta \v_n -\nabla
p_n -\mathcal{J}_{n} \left[(\v_n\cdot \nabla)\v_n \right] + \mu_1 \nabla \cdot \mathbf{\tau}_n \right] dt
+\mathcal{J}_{n} \sigma(t,\v_n)\d W_1(t)\notag\\
&\qquad\qquad\qquad+\int_Z \mathcal{J}_{n} G(\v_n(t-)) \tilde{N}_1(\d t,\d z),\label{Trun1} \\
&d \mathbf{\mathbf{\tau}}_n=-\left[\mathcal{J}_{n}(\v_n \cdot \nabla)\mathbf{\tau}_n  +a\mathbf{\tau}_n +\mathcal{J}_{n} \Q(\mathbf{\tau}_n, \nabla \v_n)
-\mu_2 \mathcal{D}(\v_n)\right]dt\notag\\
&\qquad\qquad\qquad+ \frac{1}{2}\mathcal{J}_{n}\mathcal{S}^2(\mathbf{\tau}_n)dt+ \mathcal{J}_{n}\mathcal{S}(\mathbf{\tau}_n)d W_2(t) , \label{Trun2}\\
&\nabla\cdot\v_n =0 ,\label{Trun3}\\
&\v_n(0,\cdot)=\v_n(0),\,\,\mathbf{\tau}_n(0,\cdot)=\mathbf{\tau}_n(0)  .\label{Trun4}
\end{align} 

As the truncations are invariant under the flow of the equation, we ensure that $\v_n,\mathbf{\tau}_n$ lie in the space 
\begin{align*}
\mathcal{V}_n := \left\{ g \in L^2(\mathbb R^d) : supp(\widehat g) \subset B(0,n)  \right\}
\end{align*}
with $\nabla \cdot \v_n=0.$

\begin{prop} \label{lip.q}
Let $\v_n,\mathbf{\tau}_n \in H^s(\mathbb{R}^d)$, for $s>d/2$ with $\nabla \cdot \v_n=0.$ Then the
bilinear operator $\mathcal{J}_n\Q(\mathbf{\tau}_n, \nabla \v_n)$ is
locally Lipschitz in $\v_n$ and $\mathbf{\tau}_n$ on the space $\mathcal{V}_n$.
\end{prop}
\begin{proof}
For $s>d/2\geq 1,$ let us use integration by parts, H\"{o}lder's inequality and Sobolev inequality  for $\tau_n \in H^s(\mathbb{R}^d),$ we have
\begin{align*}
&|\Big(\mathcal{J}_n \Q(\tau_n,\nabla \v_n^1)-\mathcal{J}_n \Q(\tau_n,\nabla \v_n^2),\v_n^1-\v_n^2\Big)_{\mathrm{L}^2}|
=|\Big( (\Q(\tau_n,\nabla (\v_n^1- \v_n^2)),\mathcal{J}_n(\v_n^1-\v_n^2)\Big)_{\mathrm{L}^2}|\notag\\
&\leq \|\Q(\tau_n,\nabla (\v_n^1- \v_n^2)\|_{\mathrm{L}^2} \|\mathcal{J}_n(\v_n^1-\v_n^2)\|_{\mathrm{L}^2}
 \leq C \|\tau_n\|_{\mathrm{L}^\infty}\|\nabla (\v_n^1- \v_n^2)\|_{\mathrm{L}^2}\|\v_n^1- \v_n^2\|_{\mathrm{L}^2}\notag \\
&\leq C \|\tau_n\|_{H^s}\|\v_n^1- \v_n^2\|_{H^1}\|\v_n^1- \v_n^2\|_{\mathrm{L}^2}.
\end{align*}
Hence for $\tau_n \in H^s(\mathbb{R}^d),$ and for $s>d/2,$
$$ \|\mathcal{J}_n \Q(\tau_n,\nabla \v_n^1)-\mathcal{J}_n \Q(\tau_n,\nabla \v_n^2)\|_{\mathrm{L}^2} \leq C \|\tau_n\|_{H^s}\| \v_n^1- \v_n^2\|_{H^s}.$$
Hence, $\mathcal{J}_n \Q(\cdot, \cdot)$ is locally Lipschitz in $\v_n.$
Again, for $s>d/2\geq 1,$ using same arguments as before for $\v_n \in H^s(\mathbb{R}^d),$ we have
\begin{align*}
&|\Big(\mathcal{J}_n \Q(\tau_n^1,\nabla \v_n)-\mathcal{J}_n \Q(\tau_n^2,\nabla \v_n),\tau_n^1-\tau_n^2\Big)_{\mathrm{L}^2}|
=|\Big( (\Q(\tau_n^1-\tau_n^2,\nabla \v_n),\mathcal{J}_n(\tau_n^1-\tau_n^2)\Big)_{\mathrm{L}^2}|\notag\\
&\leq \|\Q(\tau_n^1-\tau_n^2,\nabla \v_n)\|_{\mathrm{L}^2} \|\mathcal{J}_n(\tau_n^1-\tau_n^2)\|_{\mathrm{L}^2}
 \leq C \|\tau_n^1-\tau_n^2\|_{\mathrm{L}^\infty}\|\nabla \v_n\|_{\mathrm{L}^2}\|\tau_n^1-\tau_n^2\|_{\mathrm{L}^2}\notag \\
&\leq C \|\tau_n^1-\tau_n^2\|_{H^s}\| \v_n\|_{H^1}\|\tau_n^1-\tau_n^2\|_{\mathrm{L}^2} 
\leq C \|\tau_n^1-\tau_n^2\|_{H^s}\| \v_n\|_{H^s}\|\tau_n^1-\tau_n^2\|_{\mathrm{L}^2}.
\end{align*}
Hence, $\mathcal{J}_n \Q(\cdot, \cdot)$ is locally Lipschitz in $\tau_n.$
\end{proof}

\begin{rem}
\begin{itemize}
\item[1.] Similarly, for $\v_n,\mathbf{\tau}_n \in H^s(\mathbb{R}^d)$, for $s>d/2$ with $\nabla \cdot \v_n=0,$ the
nonlinear operators $\mathcal{J}_n\left[(\v_n \cdot\nabla)\mathbf{\tau}_n\right]$ and $\mathcal{J}_n\left[(\v_n \cdot\nabla)\mathbf{\v}_n\right]$ is
locally Lipschitz in $\v_n$ and $\mathbf{\tau}_n$ on the space $\mathcal{V}_n$.
\item[2.] By using Plancherel's Theorem, we observe that depending on $n,$ $\Delta\v_n$ has a
bounded linear growth in $\mathcal{V}_n$, since
\begin{align*}
\|\Delta\v_n\|_{\mathrm{L}^2(\mathbb{R}^d)}^2=\|\widehat{\Delta\v_n}\|_{\mathrm{L}^2(\mathbb{R}^d)}^2
=\|\widehat{\Delta\v_n}\|_{\mathcal{V}_n}^2=\||\xi|^2\widehat{\v_n}\|_{\mathcal{V}_n}^2 \leq
n^2\|\widehat{\v_n}\|_{\mathcal{V}_n}^2
=n^2\|\v_n\|_{\mathrm{L}^2(\mathbb{R}^d)}^2. 
\end{align*}
\end{itemize}
\end{rem}
\begin{cor}
By Theorem $4.9$ of Mandrekar and R\"{u}diger \cite{MaRu}, Ikeda and Watanbe \cite{IW} there exists a
path-wise unique strong solution $(\v_n,\tau_n)$ to problem
(\ref{Trun1})-(\ref{Trun4}) suct that 
$\v_n \in \mathrm{L}^2(\Omega;D(0,T;\mathcal{V}_n))$,  with $\nabla \cdot \v_n=0$ and $\tau_n \in \mathrm{L}^2(\Omega;C(0,T;\mathcal{V}_n))$, where $T$ depends on
$n$. The solution will exist as long as
$\|\v_n\|_{\mathrm{L}^2(\Omega;H^s(\mathbb{R}^d))}$,
$\|\tau_n\|_{\mathrm{L}^2(\Omega;H^s(\mathbb{R}^d))}$ remain finite.
\end{cor}

\subsection{Energy Estimates}

In this Subsection we first obtain energy estimates of approximate solutions on the time interval
$[0,\Ln]$ for some stopping time $\rho^n_N.$  These estimates solely depend on the regularity of the initial data and noise terms.

\begin{thm}\label{positive1}
Let the initial data $\v_0,\mathbf{\tau}_0\in
\mathrm{L}^2(\Omega;H^s(\mathbb{R}^d))$, with $\nabla\cdot \v_0 =0,$  $s>d/2$ be
$\mathscr{F}_0$-measurable, and the Assumption \ref{hypo} be satisfied. For each $n \in \mathbb{N},$ let $(\v_n,\tau_n)$ be the unique strong solution of \eqref{Trun1}-\eqref{Trun4}. Define the stopping time
\begin{align} \label{stop1} 
\rho^n_N=\inf_{t\geq
0}\left\{t:\mu_2\|\v_n(t)\|_{H^s}^2+\mu_1\|\mathbf{\tau}_n(t)\|_{H^s}^2
+2 \mu_2\nu \int_0^t \| \nabla\v_n (r)\|_{H^s}^2dr>N\right\}.
\end{align} 
Then for any $\delta$ with $0< \delta<1$, 
\begin{align}\label{rhoN}
\mathbb{P}\left(\rho^n_N>\delta\right)
\geq 1-2\delta e^{(\tilde{C}+C_2 \delta)} \Big(2\mathbb{E}\left(\mu_2\|\v_0\|_{H^s}^2+\mu_1\|\mathbf{\tau}_0\|_{H^s}^2\right)+18K\mu_2\delta\Big),	
\end{align}
for some positive constant $\tilde{C}$ independent of $\delta$.
\end{thm}
\par
\noindent
As a consequence of the above Theorem, we have the following result.
\begin{rem} \label{ener.estim}
	For any $T>0$, the quantities
	\begin{align*}
	\mathbb{E}\left[\sup_{0\leq t\leq
		T\wedge\rho^n_N}\|\v_n(t)\|_{H^s}^2\right],
	\mathbb{E}\left[\sup_{0\leq t\leq
		T\wedge\rho^n_N}\|\mathbf{\tau}_n(t)\|_{H^s}^2\right],
	\mathbb{E}\left[\int_0^{T\wedge\rho^n_N}\|\nabla\v_n(t)\|_{H^s}^2\,dt\right]
	\end{align*}
	are uniformly bounded.
\end{rem}

\begin{proof}
\textbf{Step I:} \newline

Applying $J^s$ to both the equations \eqref{Trun1}-\eqref{Trun2} we get,
\begin{align}
&d J^s \mathbf{v}_n=\left[\nu \Delta J^s \v_n -\nabla J^s p_n -\mathcal{J}_{n} J^s \left[(\v_n\cdot \nabla)\v_n \right] 
 + \mu_1 J^s \nabla \cdot \mathbf{\tau}_n \right] dt 
+  \mathcal{J}_{n} J^s \sigma(t,\v_n)\d W_1(t) \notag\\
&\qquad \qquad \qquad \qquad \qquad \qquad \qquad +\int_Z \mathcal{J}_{n} J^s G(\v_n(t-)) \tilde{N}_1(\d t,\d z),\label{Trun1s} \\
&d J^s \mathbf{\tau}_n=-\left[\mathcal{J}_{n} J^s (\v_n \cdot \nabla)\mathbf{\tau}_n  +a J^s \mathbf{\tau}_n 
+\mathcal{J}_{n} J^s \Q(\mathbf{\tau}_n, \nabla \v_n)-\mu_2 J^s \mathcal{D}(\v_n)\right]dt\notag\\
&\qquad\qquad\qquad\qquad\qquad\qquad\qquad+ \frac{1}{2}\mathcal{J}_{n} J^s\mathcal{S}^2(\mathbf{\tau}_n)dt+ \mathcal{J}_{n} J^s\mathcal{S}(\mathbf{\tau}_n)d W_2(t)  . \label{Trun2s}
\end{align} 
Applying It\^{o}'s Lemma [see Brze\'{z}niak et al. \cite{BHZ}, Ikeda and Watanabe \cite{IW}] to the function $\|x\|_{\mathrm{L}^2}^2$ to the process $J^s \v_n$ and exploiting divergence free condition  in \eqref{Trun1s}  we get,
\begin{align*}
d \| \v_n\|_{H^s}^2&=-2 \nu \| \nabla J^s \v_n \|_{\mathrm{L}^2}^2dt -2\left( J^s  (\v_n\cdot \nabla)\v_n,J^s \v_n \right)_{\mathrm{L}^2} dt \\
& \quad + 2\mu_1 \left( J^s \nabla \cdot \mathbf{\tau}_n, J^s \v_n \right)_{\mathrm{L}^2} dt 
+  \| \mathcal{J}_{n} J^s \sigma(t,\v_n)\|_{\mathcal{L}_\Q(\mathrm{L}^2,\mathrm{L}^2)} dt 
 +\int_Z \| \mathcal{J}_{n} J^s G(\v_n(t-),z)\|_{\mathrm{L}^2}^2 N_1(d t,d z)\\
 & \quad +2 \int_Z ( \mathcal{J}_{n} J^s G(\v_n(t-),z),J^s \v_n(t-))_{\mathrm{L}^2} \tilde{N}_1 (d t,d z).
\end{align*}
Exploiting the cut off property \eqref{intro1} the above equality is reduced to:
\begin{align} \label{Ito1}
d \| \v_n\|_{H^s}^2+2 \nu \| \nabla J^s \v_n \|_{\mathrm{L}^2}^2dt  
& \leq -2\left( J^s  (\v_n\cdot \nabla)\v_n,J^s \v_n \right)_{\mathrm{L}^2} dt 
+ 2\mu_1 \left( J^s \nabla \cdot \mathbf{\tau}_n, J^s \v_n \right)_{\mathrm{L}^2} dt \notag\\
& \quad + \| J^s \sigma(t,\v_n)\|^2_{\mathcal{L}_\Q(\mathrm{L}^2,\mathrm{L}^2)} dt \notag\\
& \quad + 2( \mathcal{J}_{n} J^s \sigma(t, \v_n)dW_1(t),J^s \v_n)_{\mathrm{L}^2}
 +\int_Z \| J^s G(\v_n(t-),z)\|_{\mathrm{L}^2}^2 N_1(d t,d z)\notag\\
 &\quad +2 \int_Z ( \mathcal{J}_{n} J^s G(\v_n(t-),z),J^s \v_n(t-))_{\mathrm{L}^2} \tilde{N}_1 (d t,d z).
\end{align}
Again let us apply It\^{o}'s Lemma to the function $\|x\|_{\mathrm{L}^2}^2$ to the process $J^s \mathbf{\tau}_n$ in \eqref{Trun2s}  we obtain,
\begin{align*}
d \|J^s \mathbf{\tau}_n\|_{\mathrm{L}^2}^2
 \leq & -2 \left(J^s (\v_n \cdot \nabla)\mathbf{\tau}_n,J^s \mathbf{\tau}_n\right)_{\mathrm{L}^2}dt 
 -2\left( J^s \mathbf{Q}(\mathbf{\tau}_n, \nabla \v_n),J^s \mathbf{\tau}_n\right)_{\mathrm{L}^2}dt\notag\\
 &+2\mu_2 \left(J^s \mathcal{D}(\v_n),J^s \mathbf{\tau}_n\right)_{\mathrm{L}^2} dt
 + \left(\mathcal{J}_{n} J^s\mathcal{S}^2(\mathbf{\tau}_n),J^s \mathbf{\tau}_n\right)_{\mathrm{L}^2}dt  \notag\\
 &+\| \mathcal{J}_{n} J^s \mathcal{S}(\mathbf{\tau}_n)\|^2_{\mathrm{L}^2} dt
 + 2\left(\mathcal{J}_{n} J^s \mathcal{S}(\mathbf{\tau}_n) ,J^s\mathbf{\tau}_n\right)_{\mathrm{L}^2}d W_2(t)+2a\|J^s \mathbf{\tau}_n\|_{\mathrm{L}^2}^2dt.
\end{align*}
\begin{align}\label{Ito2}
d \|J^s \mathbf{\tau}_n\|_{\mathrm{L}^2}^2
 \leq & -2 \left(J^s (\v_n \cdot \nabla)\mathbf{\tau}_n,J^s \mathbf{\tau}_n\right)_{\mathrm{L}^2}dt 
 -2\left( J^s \mathbf{Q}(\mathbf{\tau}_n, \nabla \v_n),J^s \mathbf{\tau}_n\right)_{\mathrm{L}^2}dt\notag\\
 &+2\mu_2 \left(J^s \mathcal{D}(\v_n),J^s \mathbf{\tau}_n\right)_{\mathrm{L}^2} dt
 + \left(\mathcal{J}_{n} J^s\mathcal{S}^2(\mathbf{\tau}_n),J^s \mathbf{\tau}_n\right)_{\mathrm{L}^2} dt  \notag\\
 &+\|h\|_{H^s}^2\|\mathbf{\tau}_n\|_{H^s}^2 dt
 + 2\left(\mathcal{J}_{n} J^s \mathcal{S}(\mathbf{\tau}_n),J^s\mathbf{\tau}_n\right)_{\mathrm{L}^2} d W_2(t)+2a\|J^s \mathbf{\tau}_n\|_{\mathrm{L}^2}^2 dt.
\end{align}

Let $\lambda_j$ be the eigenvalues of $Q$ such that $Qe_j=\lambda_je_j$ for all
$j=1,2,\cdots$, where $\{e_j\}_{j=1}^{\infty}$ are the orthonormal basis in
$\mathrm{L}^2(\mathbb{R}^d).$ Hence we have,
\begin{align}\label{Ito3}
\|J^s \sigma(t,\v_n)\|^2_{\mathcal{L}_Q(\mathrm{L}^2,\mathrm{L}^2)}&\leq
\sum_{j=1}^{\infty}\lambda_j\|J^s\sigma(t,\v_n)e_j\|^2_{\mathrm{L}^2}\nonumber\\&=
\sum_{j=1}^{\infty}\lambda_j\|\sigma(t,\v_n)e_j\|^2_{H^s}=\|\sigma(t,\v_n)\|^2_{\mathcal{L}_Q(\mathrm{L}^2,H^s)}.
\end{align}
Multiplying $\mu_2$ with \eqref{Ito1} and $\mu_1$ with \eqref{Ito2} and then on adding we have
\begin{align} \label{Ito4}
&d \left[\mu_2\|J^s \v_n\|_{\mathrm{L}^2}^2+\mu_1\|J^s \mathbf{\tau}_n\|_{\mathrm{L}^2}^2\right]
 +2 \mu_2\nu \| \nabla J^s \v_n \|_{\mathrm{L}^2}^2dt \notag \\
 &\leq\underbrace{-2\mu_2\left( J^s  (\v_n\cdot \nabla)\v_n,J^s \v_n \right)_{\mathrm{L}^2} } _{I_1} dt \underbrace{ -2 \mu_1 \left(J^s (\v_n \cdot \nabla)\mathbf{\tau}_n,J^s \mathbf{\tau}_n\right)_{\mathrm{L}^2}}_{I_2} dt
 \underbrace{-2\mu_1\left( J^s \mathbf{Q}(\mathbf{\tau}_n, \nabla \v_n),J^s \mathbf{\tau}_n\right)_{\mathrm{L}^2}}_{I_3} dt\notag\\
&\quad +\mu_2 \| J^s \sigma(t,\v_n)\|^2_{\mathcal{L}_\Q(\mathrm{L}^2,\mathrm{L}^2)} dt
+ 2\mu_2( \mathcal{J}_{n} J^s \sigma(t, \v_n)dW_1(t),J^s \v_n)_{\mathrm{L}^2}\notag\\
&\quad +\mu_2\int_Z \| J^s G(\v_n(t-),z)\|_{\mathrm{L}^2}^2 N_1(d t,d z) +2 \mu_2\int_Z ( \mathcal{J}_{n} J^s G(\v_n(t-),z),J^s \v_n(t-))_{\mathrm{L}^2} \tilde{N}_1 (d t,d z)\notag\\
 &\quad + \underbrace{ 2\mu_1\mu_2 \left( J^s \nabla \cdot \mathbf{\tau}_n, J^s \v_n \right)_{\mathrm{L}^2} dt +2\mu_1\mu_2 \left(J^s \mathcal{D}(\v_n),J^s \mathbf{\tau}_n\right)_{\mathrm{L}^2} }_{I_0} dt 
 + \underbrace{\mu_1\left(\mathcal{J}_{n} J^s\mathcal{S}^2(\mathbf{\tau}_n),J^s \mathbf{\tau}_n\right)_{\mathrm{L}^2}}_{I_4} dt \notag\\
 &\quad +\mu_1\|h\|_{H^s}^2\|\mathbf{\tau}_n\|_{H^s}^2 dt
 + 2\mu_1\left(\mathcal{J}_{n} J^s \mathcal{S}(\mathbf{\tau}_n) ,J^s\mathbf{\tau}_n\right)_{\mathrm{L}^2}d W_2(t)+2a\mu_1\|J^s \mathbf{\tau}_n\|_{\mathrm{L}^2}^2 dt.
\end{align}

Using the divergence free condition of $\v_n,$ we directly have $I_0$ is zero.
We now recall the Commutator estimates (mentioned in Subsection \ref{comm.est}) and separately estimate each $I_{i}\,;i=1,\cdots, 4.$
In order to estimate the term ($I_1$), we recall the fact that $H^s$ is an
algebra for $s>d/2$ to get
\begin{align}\label{I_1}
|I_1|=\left|-2\mu_2\left( J^s\left[(\v_n\cdot\nabla)\v_n\right],J^s\v_n\right)_{\mathrm{L}^2}\right|
&\leq
2\mu_2\| J^s\left[(\v_n\cdot\nabla)\v_n\right]\|_{\mathrm{L^2}}\|J^s\v_n\|_{\mathrm{L}^2}\nonumber\\
&\leq
2\mu_2\|\left[(\v_n\cdot\nabla)\v_n\right]\|_{H^s}\|\v_n\|_{H^s} \nonumber\\
&\leq C\mu_2 \|\nabla\v_n\|_{H^s}\|\v_n\|_{H^s}^2.
\end{align}
 Once again we use the fact that $H^s$ is an
algebra for $s>d/2,$ and we recall Theorem \ref{com_thm}. Hence $I_2$ is reduced to:
\begin{align}\label{I_2}
|I_2|=\left|-2\mu_1\left( J^s\left[(\v_n\cdot\nabla)\mathbf{\tau}_n\right],J^s\mathbf{\tau}_n\right)_{\mathrm{L}^2}\right|
&\leq
2\mu_1\| J^s\left[(\v_n\cdot\nabla)\mathbf{\tau}_n\right]\|_{\mathrm{L^2}}\|J^s\mathbf{\tau}_n\|_{\mathrm{L}^2}\nonumber\\
&\leq C\mu_1 \|\nabla\v_n\|_{H^s}\|\mathbf{\tau}_n\|_{H^s}^2.
\end{align}
For $I_3,$ using the classical tame estimate for $\Q$ (Property \ref{propQ}) we get
\begin{align}\label{I_3}
|I_3|=\left|-2\mu_1\left( J^s Q(\mathbf{\tau}_n,\nabla\v_n),J^s\mathbf{\tau}_n\right)_{\mathrm{L}^2}\right|
&\leq
2\mu_1\| J^s Q(\mathbf{\tau}_n,\nabla\v_n)\|_{\mathrm{L^2}}\|J^s\mathbf{\tau}_n\|_{\mathrm{L}^2}\nonumber\\
&\leq
2\mu_1\|Q(\mathbf{\tau}_n,\nabla\v_n)\|_{H^s}\|\mathbf{\tau}_n\|_{H^s} \nonumber\\
&\leq C\mu_1 \|\nabla\v_n\|_{H^s}\|\mathbf{\tau}_n\|_{H^s}^2.
\end{align}
For $I_4$ we recall \eqref{intro1} and \eqref{propS.2} and achieve,
\begin{align}\label{I_4}
 |I_4|=\mu_1 |\left(\mathcal{J}_{n} J^s\mathcal{S}^2(\mathbf{\tau}_n),J^s \mathbf{\tau}_n\right)_{\mathrm{L}^2}|
 &\leq
\mu_1\|\mathcal{J}_{n} J^s \mathcal{S}^2(\mathbf{\tau}_n)\|_{\mathrm{L^2}}\|J^s\mathbf{\tau}_n\|_{\mathrm{L}^2}\nonumber\\
&\leq
\mu_1\| J^s \mathcal{S}^2(\mathbf{\tau}_n)\|_{\mathrm{L^2}}\|J^s\mathbf{\tau}_n\|_{\mathrm{L}^2}\nonumber\\
&\leq C\mu_1 \|\mathcal{S}^2(\mathbf{\tau}_n)\|_{H^s}\|\mathbf{\tau}_n\|_{H^s}\notag\\
&\leq C\mu_1 \|h\|_{H^s}^2\|\mathbf{\tau}_n\|_{H^s}^2.
\end{align}
Using $(\ref{I_1}), (\ref{I_2}), (\ref{I_3})$ and $(\ref{I_4}),$ from $(\ref{Ito4})$ we have
\begin{align}\label{Ener1}
&d \left[\mu_2\| \v_n\|_{H^s}^2+\mu_1\|\mathbf{\tau}_n\|_{H^s}^2\right]
 +2 \mu_2\nu \| \nabla\v_n \|_{H^s}^2dt\notag\\
 &\leq C\left[\mu_2\| \v_n\|_{H^s}^2+\mu_1\|\mathbf{\tau}_n\|_{H^s}^2\right]\| \nabla\v_n \|_{H^s}
 +\left((C+1)\|h\|_{H^s}^2+2a\right)\mu_1\|\mathbf{\tau}_n\|_{H^s}^2 dt\notag\\
&\quad +\mu_2 \| \sigma(t,\v_n)\|^2_{\mathcal{L}_\Q(\mathrm{L}^2,H^s)} dt
+\mu_2\int_Z \|G(\v_n(t-),z)\|_{H^s}^2 N_1(d t,d z)\notag\\
 &\quad + 2\mu_2( \mathcal{J}_{n} J^s \sigma(t, \v_n)dW_1(t),J^s \v_n)_{\mathrm{L}^2}
 +2 \mu_2\int_Z ( \mathcal{J}_{n} J^s G(\v_n(t-),z),J^s \v_n(t-))_{\mathrm{L}^2} \tilde{N}_1 (d t,d z)\notag\\
  &\quad 
 + 2\mu_1\left(\mathcal{J}_{n} J^s \mathcal{S}(\mathbf{\tau}_n),J^s\mathbf{\tau}_n\right)_{\mathrm{L}^2}d W_2(t).
\end{align}
It is to be noted that
\begin{align}\label{m43}
\|\v_n(0)\|^2_{{H}^s}+\|\mathbf{\tau}_n(0)\|^2_{{H}^s}\leq
\|\v_0\|^2_{H^s}+\|\mathbf{\tau}_0\|^2_{H^s}.
\end{align}
For any $T>0$, let us integrate (\ref{Ener1}) from $0$ to $t$, take
supremum from $0$ to $\Ln$ (where $\rho_N^n$ is given by \eqref{stop1}) and then take expectation in (\ref{Ener1})
(use Remark \ref{nlam}) to get
\begin{align}
&\mathbb{E}\left[\sup_{0\leq t\leq
\Ln}\left[\mu_2\|\v_n(t)\|^2_{H^s}+\mu_1\|\mathbf{\tau}_n(t)\|^2_{H^s}\right]\right]
+2\nu\mu_2 \mathbb{E} \left[\int_0^{\Ln} \|\nabla\v_n(r)\|^2_{H^s} dr\right] \label{2.68}\\
&\leq\mathbb{E}\left[\mu_2\|\v_0\|^2_{{H}^s}+\mu_1\|\mathbf{\tau}_0\|^2_{{H}^s}\right]+
2C\mathbb{E}\left[\int_0^{\Ln}\|\nabla\v_n(r)\|_{H^s}\left(\mu_2\|\v_n(r)\|^2_{H^s}+\mu_1\|\mathbf{\tau}_n(r)\|^2_{H^s}\right)dr\right]\nonumber\\
&\quad+\left((C+1)\|h\|_{H^s}^2+2a\right)\mu_1\mathbb{E}\left[\int_0^{\Ln}\|\mathbf{\tau}_n(r)\|^2_{H^s}\d r\right]\nonumber\\
&\quad+\mu_2\mathbb{E}\left[\int_0^{\Ln}\left[\|\sigma(r,\v_n(r))\|^2_{\mathcal{L}_Q(\mathrm{L}^2,H^s)}+
\int_Z\|G(\v_n(r),z)\|^2_{H^s}\lambda(dz)\right]\d r\right]\nonumber\\
&\quad+2\mu_2\mathbb{E}\left[\sup_{0\leq t\leq
\Ln}\left|\int_0^{t}\left(\mathcal{J}_{n} J^s\sigma(r,\v_n)\d
W_1(r),J^s\v_n(r)\right)_{\mathrm{L}^2}\right|\right]\tag{$M_1$}\label{Ib1}\nonumber\\
&\quad+2\mu_1\mathbb{E}\left[\sup_{0\leq t\leq
\Ln}\left|\int_0^{t}\left(\mathcal{J}_{n} J^s \mathcal{S}(\mathbf{\tau}_n(r)),J^s\mathbf{\tau}_n(r)\right)_{\mathrm{L}^2}\d W_2(r)\right|\right]\tag{$M_2$}\label{Ib2}
\nonumber\\
&\quad+2\mu_2\mathbb{E}\left[\sup_{0\leq t\leq
\Ln}\left|\int_0^{t}\int_Z\left(\mathcal{J}_{n} J^s G(\v_n(r-),z),J^s\v_n(r-)\right)_{\mathrm{L}^2}
\tilde{N}_1(\d r,\d z)\right|\right].\tag{$M_3$}\label{Ib3}
\end{align}
Now by applying Burkholder-Davis-Gundy inequality, Young's inequality to the term
(\ref{Ib1}), we get
\begin{align}\label{m1}
M_1 &\leq
2\sqrt{2}\mu_2\mathbb{E}\left(\int_0^{\Ln}\|\mathcal{J}_n J^s\sigma(t,\v_n(t))\|^2_{\mathcal{L}_Q(\mathrm{L}^2,\mathrm{L}^2)}
\|J^s\v_n(t)\|^2_{\mathrm{L}^2}\d t\right)^{1/2}\nonumber\\
&\leq 2\sqrt{2}\mu_2\mathbb{E}\left[\left(\sup_{0\leq t\leq
\Ln}\|\v_n(t)\|^2_{H^s}\right)^{1/2}\left(\int_0^{\Ln}\|J^s\sigma(t,\v_n(t))\|^2_{\mathcal{L}_Q(\mathrm{L}^2,\mathrm{L}^2)}\d
t\right)^{1/2}\right]\nonumber\\&\leq
\frac{\mu_2}{4}\mathbb{E}\left(\sup_{0\leq t\leq
\Ln}\|\v_n(t)\|^2_{{H}^s}\right)+
8\mu_2\mathbb{E}\left(\int_0^{\Ln}\|\sigma(t,\v_n(t))\|^2_{\mathcal{L}_Q(\mathrm{L}^2,H^s)}\d
t\right).
\end{align}
 Again applying Burkholder-Davis-Gundy inequality, Young's inequality  in similar manner, the term \ref{Ib2} is reduced to :
\begin{align}\label{m2}
M_2 &\leq
2\sqrt{2}\mu_1\mathbb{E}\left(\int_0^{\Ln}\|\mathcal{J}_n J^s \mathcal{S}(\mathbf{\tau}_n(t))\|^2_{\mathrm{L}^2}
\|J^s\mathbf{\tau}_n(t)\|^2_{\mathrm{L}^2}\d t\right)^{1/2}\nonumber\\
&\leq\frac{\mu_1}{2}\mathbb{E}\left(\sup_{0\leq t\leq
\Ln}\|\mathbf{\tau}_n(t)\|^2_{{H}^s}\right)+
4\mu_1\|h\|^2_{H^s}\mathbb{E}\left(\int_0^{\Ln}\|\mathbf{\tau}_n(t)\|^2_{H^s}\d
t\right).
\end{align}
Let us now apply Burkholder-Davis-Gundy inequality, Young's inequality and Assumption \ref{hypo} to the term \ref{Ib3} to
obtain
\begin{align}\label{m3}
M_3&\leq
2\sqrt{2}\mu_2\mathbb{E}\left[\int_0^{\Ln}\int_Z\|\mathcal{J}_n J^s G(\v_n(t),z)\|_{\mathrm{L}^2}^2\|J^s\v_n\|_{\mathrm{L}^2}^2
\lambda(\d z)\d t\right]^{1/2}\nonumber\\ 
&\leq
\frac{\mu_2}{4}\mathbb{E}\left(\sup_{0\leq t\leq
\Ln}\|\v_n(t)\|^2_{{H}^s}\right)+8\mu_2
\mathbb{E}\left(\int_{0}^{\Ln}\int_Z\|G(\v_n(t),z)\|_{H^s}^2\lambda(\d
z)\d t\right) \notag\\
&\leq
\frac{\mu_2}{4}\mathbb{E}\left(\sup_{0\leq t\leq
\Ln}\|\v_n(t)\|^2_{{H}^s}\right)+8\mu_2K
\mathbb{E}\left(\int_{0}^{\Ln}(1+\|\v_n(t)\|_{H^s}^2)\d t\right).
\end{align}
By using (\ref{m1}),  (\ref{m2}), and (\ref{m3}) in
(\ref{2.68}), we get
\begin{align}\label{Ener3}
&\mathbb{E}\left[\sup_{0\leq t\leq
\Ln}\left[\mu_2\|\v_n(t)\|^2_{H^s}+\mu_1\|\mathbf{\tau}_n(t)\|^2_{H^s}\right]\right]+
4\nu\mu_2\mathbb{E}\left[\int_0^{\Ln}\|\nabla\v_n(t)\|^2_{H^s}\d
t\right]\nonumber\\
&\leq2\mathbb{E}\left[\mu_2\|\v_0\|^2_{{H}^s}+\mu_1\|\mathbf{\tau}_0\|^2_{{H}^s}\right]+
4C\mathbb{E}\left[\int_0^{\Ln}\|\nabla\v_n\|_{H^s}\left(\mu_2\|\v_n(t)\|^2_{H^s}+\mu_1\|\mathbf{\tau}_n(t)\|^2_{H^s}\right)\d t\right]\nonumber\\
&\quad+2\left((C+5)\|h\|_{H^s}^2+2a\right)\mu_1\mathbb{E}\left[\int_0^{\Ln}\|\mathbf{\tau}_n(t)\|^2_{H^s}\d t\right]\nonumber\\
&\quad+18K\mu_2\mathbb{E}\left[\int_0^{\Ln}\left(1+\|\v_n(t)\|^2_{H^s}+\|\mathbf{\tau}_n(t)\|_{H^s}^2\right)\d
t\right].
\end{align}
Hence, on rearranging the constants of the last two terms of \eqref{Ener3}, the inequality further reduces to 
\begin{align}\label{Ener5}
&\mathbb{E}\left[\sup_{0\leq t\leq
\Ln}\left[\mu_2\|\v_n(t)\|^2_{H^s}+\mu_1\|\mathbf{\tau}_n(t)\|^2_{H^s}\right]\right]+
4\nu\mu_2\mathbb{E}\left[\int_0^{\Ln}\|\nabla\v_n(t)\|^2_{H^s}\d
t\right]\nonumber\\
&\leq2\mathbb{E}\left[\mu_2\|\v_0\|^2_{{H}^s}+\mu_1\|\mathbf{\tau}_0\|^2_{{H}^s}\right]+
4C\mathbb{E}\left[\int_0^{\Ln}\|\nabla\v_n(t)\|_{H^s}\left(\mu_2\|\v_n(t)\|^2_{H^s}+\mu_1\|\mathbf{\tau}_n(t)\|^2_{H^s}\right)dt \right]\nonumber\\
&\quad+C_2\mathbb{E}\left[\int_0^{\Ln}\left(\mu_2\|\v_n(t)\|^2_{H^s}+\mu_1\|\mathbf{\tau}_n(t)\|^2_{H^s}\right)\d t\right]
+18K\mu_2T,
\end{align}
where $C_2=\left(2(C+5)\|h\|_{H^s}^2+4a+\frac{18K\mu_2}{\mu_1}+18K\right).$
By using Young's inequality, we get
\begin{align}\label{young1}
4C\|\nabla\v_n(t)\|_{H^s}\left(\mu_2\|\v_n(t)\|^2_{H^s}+\mu_1\|\mathbf{\tau}_n(t)\|^2_{H^s}\right)\leq
2\nu\mu_2\|\nabla\v_n(t) \|_{H^s}^2+\frac{2C^2}{\nu\mu_2}\left(\mu_2\|\v_n(t)\|^2_{H^s}+\mu_1\|\mathbf{\tau}_n(t)\|^2_{H^s}\right)^2.
\end{align}
By using (\ref{young1}) and making use of the stopping time defined
by (\ref{stop1}), we obtain
\begin{align*}
&\mathbb{E}\left[\sup_{0\leq t\leq
\Ln}\left[\mu_2\|\v_n(t)\|^2_{H^s}+\mu_1\|\mathbf{\tau}_n(t)\|^2_{H^s}\right]\right]+
2\nu\mu_2\mathbb{E}\left[\int_0^{\Ln}\|\nabla\v_n(t)\|^2_{H^s}\d
t\right]\nonumber\\
&\leq2\mathbb{E}\left[\mu_2\|\v_0\|^2_{{H}^s}+\mu_1\|\mathbf{\tau}_0\|^2_{{H}^s}\right]+18K\mu_2T+
\frac{2C^2}{\nu\mu_2}
\mathbb{E}\left[\int_0^{\Ln}\left(\mu_2\|\v_n(t)\|^2_{H^s}+\mu_1\|\mathbf{\tau}_n(t)\|^2_{H^s}\right)^2dt\right]\nonumber\\
&\quad+C_2\mathbb{E}\left[\int_0^{\Ln}\left(\mu_2\|\v_n(t)\|^2_{H^s}+\mu_1\|\mathbf{\tau}_n(t)\|^2_{H^s}\right)\d t\right]\notag\\
&\leq2\mathbb{E}\left[\mu_2\|\v_0\|^2_{{H}^s}+\mu_1\|\mathbf{\tau}_0\|^2_{{H}^s}\right]+18K\mu_2T\nonumber\\
&\quad+\left(\frac{2C^2N}{\nu\mu_2}+C_2\right)
\mathbb{E}\left[\int_0^{\Ln}\left(\mu_2\|\v_n(t)\|^2_{H^s}+\mu_1\|\mathbf{\tau}_n(t)\|^2_{H^s}\right)\d t\right].
\end{align*}
Finally, we have
\begin{align}\label{Ener6}
&\mathbb{E}\left[\sup_{0\leq t\leq
\Ln}\left[\mu_2\|\v_n(t)\|^2_{H^s}+\mu_1\|\mathbf{\tau}_n(t)\|^2_{H^s}\right]\right]+
2\nu\mu_2\mathbb{E}\left[\int_0^{\Ln}\|\nabla\v_n(t)\|^2_{H^s}\d
t\right]\nonumber\\
&\leq2\mathbb{E}\left[\mu_2\|\v_0\|^2_{{H}^s}+\mu_1\|\mathbf{\tau}_0\|^2_{{H}^s}\right]+18K\mu_2T\nonumber\\
&\quad+\left(\frac{2C^2N}{\nu\mu_2}+C_2\right)
\mathbb{E}\left[\int_0^{T}\sup_{0\leq t\leq \Ln}\left(\mu_2\|\v_n(t)\|^2_{H^s}+\mu_1\|\mathbf{\tau}_n(t)\|^2_{H^s}\right)\d t\right].
\end{align}
Now for any $T>0$ and $\rho^n_N$ defined in (\ref{stop1}), an application of Gronwall's inequality yields
\begin{align}\label{Ener.last}
&\mathbb{E}\left[\sup_{0\leq t\leq
\Ln}\left[\mu_2\|\v_n(t)\|^2_{H^s}+\mu_1\|\mathbf{\tau}_n(t)\|^2_{H^s}\right]\right]+
2\nu\mu_2\mathbb{E}\left[\int_0^{\Ln}\|\nabla\v_n(t)\|^2_{H^s}\d
t\right]\nonumber\\
&\leq\left(2\mathbb{E}\left[\mu_2\|\v_0\|^2_{{H}^s}+\mu_1\|\mathbf{\tau}_0\|^2_{{H}^s}\right]+18K\mu_2T\right)
\exp\left\{\left(\frac{2C^2N}{\nu\mu_2}+C_2\right)T\right\}.
\end{align}

\textbf{Step II:} \newline

We further assume that $\mathbb{E}\left[\|\v_0\|_{H^s}^2\right]<\infty$ and
$\mathbb{E}\left[\|\mathbf{\tau}_0\|_{H^s}^2\right]<\infty$.
By using (\ref{Ener.last}), we get
\begin{align}\label{Ener7}
&\mathbb{E}\left[\sup_{0\leq t\leq\delta}\left[\mu_2\|\v_n(t \wedge \rho^n_N)\|^2_{H^s}+\mu_1\|\mathbf{\tau}_n(t \wedge \rho^n_N)\|^2_{H^s}\right]\right]+
2\nu\mu_2\mathbb{E}\left[\int_0^{\delta}\|\nabla\v_n(t)\|^2_{H^s}\d
t\right]\nonumber\\
&\leq\Big(2\mathbb{E}\left[\mu_2\|\v_0\|^2_{{H}^s}+\mu_1\|\mathbf{\tau}_0\|^2_{{H}^s}\right]+18K\mu_2\delta\Big)
\exp\left\{\left(\frac{2C^2 N}{\nu\mu_2}+C_2\right) \delta\right\},
\end{align}
where $C$ is a positive constant independent of $N$ and $\delta.$
Let $0<\delta<1$ be given. Then there exists  a positive integer $N$
such that $$\frac{1}{N+1}\leq \delta<\frac{1}{N}.$$ By the
definition of $\rho^n_N$, one can easily observe that
\begin{align}
\left\{\sup_{0\leq t\leq
\delta}\left[\mu_2\|\v_n(t)\|_{H^s}^2+\mu_1\|\mathbf{\tau}_n(t)\|_{H^s}^2\right]+2\nu\mu_2\int_0^{\delta}\|\nabla\v_n(s)\|_{H^s}^2\d
s\leq N\right\}\subset\left\{\rho^n_N> \delta\right\}.
\end{align}
Now as an application of Markov's inequality for $0<\delta<1,$ and using $1< \frac{N+1}{N}<2,\,\, \frac{1}{N+1} < \delta <\frac{1}{N}$ and \eqref{Ener7}, we get
\begin{align} \label{esti.rho.n}
\mathbb{P}\left(\rho^n_N>\delta\right)&>\mathbb{P}\left(\left\{\sup_{0\leq
t\leq
\delta}\left[\mu_2\|\v_n(t)\|_{H^s}^2+\mu_1\|\mathbf{\tau}_n(t)\|_{H^s}^2\right]+2\nu\mu_2\int_0^{\delta}\|\nabla\v_n(s)\|_{H^s}^2\d
s\leq N\right\}\right)\nonumber\\
&\geq 1-\frac{1}{N}\mathbb{E}\left(\sup_{0\leq t\leq
\delta}\left[\mu_2\|\v_n(t)\|_{H^s}^2+\mu_1\|\mathbf{\tau}_n(t)\|_{H^s}^2\right]+2\nu\mu_2\int_0^{\delta}\|\nabla\v_n(s)\|_{H^s}^2\d
s\right)\nonumber\\
&\geq 1-\frac{1}{N}\left\{\Big(2\mathbb{E}\left(\mu_2\|\v_0\|_{H^s}^2+\mu_1\|\mathbf{\tau}_0\|_{H^s}^2\right)+18K\mu_2\delta\Big) 
\exp \left(\left(\frac{2C^2N}{\nu\mu_2}+C_2\right)\delta\right)\right\}
\nonumber\\
&\geq 1-\frac{e^{\tilde{C}N \delta} e^{C_2 \delta}}{N}\Big(2\mathbb{E}\left(\mu_2\|\v_0\|_{H^s}^2+\mu_1\|\mathbf{\tau}_0\|_{H^s}^2\right)+18K\mu_2\delta\Big)
\nonumber\\
&\geq 1-e^{\tilde{C}} e^{C_2 \delta} \frac{N+1}{N(N+1)}\Big(2\mathbb{E}\left(\mu_2\|\v_0\|_{H^s}^2+\mu_1\|\mathbf{\tau}_0\|_{H^s}^2\right)+18K\mu_2\delta\Big)
\nonumber\\
&\geq 1-2\delta e^{\tilde{C}} e^{C_2 \delta} \Big(2\mathbb{E}\left(\mu_2\|\v_0\|_{H^s}^2+\mu_1\|\mathbf{\tau}_0\|_{H^s}^2\right)+18K\mu_2\delta\Big)\nonumber\\
&\geq 1-2\delta e^{(\tilde{C}+C_2 \delta)} \Big(2\mathbb{E}\left(\mu_2\|\v_0\|_{H^s}^2+\mu_1\|\mathbf{\tau}_0\|_{H^s}^2\right)+18K\mu_2\delta\Big)
\nonumber\\
\end{align}
where
 $\tilde{C}=\frac{2C^2}{\nu\mu_2}$ is a constant independent of $\delta$.
\end{proof}

\section{Strong Convergence of the Truncated Solutions}

We first prove that the solutions $(\v_n,\mathbf{\tau}_n)$ of
\eqref{Trun1}-\eqref{Trun4} converge strongly in $\mathrm L^2(\Omega, \mathrm L^{\infty}(0, T; \mathrm L^2(\R^d)))$.
In the proof we have used certain results, which have been proved in the Appendix.

\begin{thm}\label{cauchy}
Let $\rho^n_N$ be the stopping time defined in \eqref{stop1} and $T>0$. Then the
family of strong solutions $\left\{(\v_n,\mathbf{\tau}_n)\right\}_{n\in\mathbb N}$ of
\eqref{Trun1}-\eqref{Trun4} satisfy the following convergence results:
\begin{itemize}
\item[(i)] \label{cau.v.mn}
$\lim_{n\to\infty}\sup_{m \geq n} \mathbb{E}\left(\sup_{t\in[0,\Lm]}
\|\v_n-\v_m\|_{\mathrm{L}^2}^2+\sup_{t\in[0,\Lm]}\|\mathbf{\tau}_n-\mathbf{\tau}_m\|_{\mathrm{L}^2}^2\right)=0,$
\item[(ii)] $\lim_{n\to\infty}\sup_{m \geq n} \mathbb{E}\left[\int_0^{\Lm}\|\nabla(\v_n-\v_m)\|^2_{\mathrm{L}^2}\d
t\right]=0,$
\end{itemize}
where $ \Rn:= \rho^n_N \wedge \rho^m_N.$

\end{thm}

\begin{proof} We split the proof in three steps.

\textbf{Step I:} \newline
Let $(\v_n,\mathbf{\tau}_n)$ and $(\v_m,\mathbf{\tau}_m)$  be
two strong solutions of (\ref{Trun1})-(\ref{Trun4}) in
$\mathcal{V}_n$ and $\mathcal{V}_{m}$ respectively. 
 Consider the difference between the
equations \eqref{Trun1}-\eqref{Trun4} for $n$ and $m$ to get
\begin{align}
\d(\mathbf{v}_n-\v_m)&=\nu\Delta
(\v_n-\v_m)\d t-\nabla(p_n-p_m)\d t-\left(\mathcal{J}_n [(\mathbf{v}_n \cdot\nabla)\v_n ]-\mathcal{J}_m [(\mathbf{v}_m \cdot\nabla)\v_m ]\right)dt\notag \\
&\quad +\mu_1 \nabla \cdot (\mathbf{\tau}_n-\mathbf{\tau}_m)\d t +\left(\mathcal{J}_n \sigma(t,\v_n)-\mathcal{J}_{m}\sigma(t,\v_{m})\right)\d W_1(t)
\nonumber\\
&\quad+\int_Z \left(\mathcal{J}_n G(\v_n,z)-\mathcal{J}_{m} G(\v_{m},z)\right) \tilde{N}_1(\d t,\d z),\label{cs1}\\
\d(\mathbf{\tau}_n-\mathbf{\tau}_m)&=-\left( \mathcal{J}_n \left[(\mathbf{v}_n \cdot \nabla)\mathbf{\tau}_n \right] 
-\mathcal{J}_m [(\mathbf{v}_m \cdot\nabla)\mathbf{\tau}_m] \right)\d t
-\left(\mathcal{J}_n \Q(\mathbf{\tau}_n, \nabla \v_n)-\mathcal{J}_m \Q(\mathbf{\tau}_m, \nabla \v_m)\right)\d t \notag \\
&\quad+\mu_2 D(\v_n-\v_m)dt -a (\mathbf{\tau}_n-\mathbf{\tau}_m)\d t
+ \frac{1}{2}\left[\mathcal{J}_{n}\mathcal{S}^2(\mathbf{\tau}_n)-\mathcal{J}_{m} \mathcal{S}^2(\mathbf{\tau}_m)\right]\d t\notag \\
&\quad+ \left(\mathcal{J}_{n}\mathcal{S}(\mathbf{\tau}_n)-\mathcal{J}_{m} \mathcal{S}(\mathbf{\tau}_m)\right)\d W_2(t).\label{cs2}
\end{align}

Applying It\^{o}'s Lemma to the function
$\|x\|^2_{\mathrm{L}^2}$ and to the process $\mu_2(\v_n-\v_m)$
in (\ref{cs1}) and to the process $\mu_1(\mathbf{\tau}_n-\mathbf{\tau}_m)$ in (\ref{cs2}),
and  adding these two equations we get
\begin{align}
&\d\left(\mu_2\|\v_n-\v_m\|_{\mathrm{L}^2}^2+\mu_1 \|\mathbf{\tau}_n-\mathbf{\tau}_m\|_{\mathrm{L}^2}^2\right)
+2\mu_2 \nu\|\nabla(\v_n-\v_m)\|^2_{\mathrm{L}^2}\d
t\nonumber\\
&=-2 \mu_2 \left(\mathcal{J}_n [(\mathbf{v}_n \cdot\nabla)\v_n ]-\mathcal{J}_m [(\mathbf{v}_m \cdot\nabla)\v_m ]
,\v_n-\v_m \right)_{\mathrm{L}^2}\d t\notag \\
&\quad + \underbrace{\left[ 2\mu_1 \mu_2 \left\lbrace (\nabla \cdot (\mathbf{\tau}_n-\mathbf{\tau}_m), \v_n-\v_m)_{\mathrm{L}^2}
+(D(\v_n-\v_m),\mathbf{\tau}_n-\mathbf{\tau}_m)_{\mathrm{L}^2}\right\rbrace -2\mu_2\left(\nabla (p_n-p_m),\v_n-\v_m\right)_{\mathrm{L}^2}\right]}_{J_0} \d t\notag\\
&\quad+\mu_2\|\mathcal{J}_n \sigma(t,\v_n)-\mathcal{J}_{m}\sigma(t,\v_{m})\|_{\mathcal{L}_Q(\mathrm{L}^2,\mathrm{L}^2)}^2\d t
+2\mu_2\left(\left(\mathcal{J}_n \sigma(t,\v_n)-\mathcal{J}_{m}\sigma(t,\v_{m})\right)\d W_1(t),\v_n-\v_m \right)\notag\\
&\quad+\mu_2\int_Z\|\mathcal{J}_{n} G(\v_n(t-),z)-\mathcal{J}_{m}G
(\v_m(t-),z)\|^2_{\mathrm{L}^2}N_1(\d t,\d z)\nonumber\\
&\quad+2\mu_2\int_Z\left(\mathcal{J}_{n} G(\v_n(t-),z)-\mathcal{J}_{m} G(\v_m(t-),z),\v_n-\v_m \right)_{\mathrm{L}^2}\tilde{N}_1(\d
t,\d z) \nonumber\\
&\quad-2 \mu_1 \left(\mathcal{J}_n [(\mathbf{v}_n \cdot\nabla)\mathbf{\tau}_n ]-\mathcal{J}_m [(\mathbf{v}_m \cdot\nabla)\mathbf{\tau}_m ]
,\mathbf{\tau}_n-\mathbf{\tau}_m \right)_{\mathrm{L}^2}\d t\notag \\
&\quad-2\mu_1 \left(\mathcal{J}_n \Q(\mathbf{\tau}_n, \nabla \v_n)-
\mathcal{J}_m \Q(\mathbf{\tau}_m, \nabla \v_m), \mathbf{\tau}_n-\mathbf{\tau}_m\right)_{\mathrm{L}^2}\d t-2a\mu_1 \|\mathbf{\tau}_n-\mathbf{\tau}_m\|_{\mathrm{L}^2}^2\d t
\nonumber\\
&\quad+ \mu_1\left(\mathcal{J}_{n}\mathcal{S}^2(\mathbf{\tau}_n)-\mathcal{J}_{m} \mathcal{S}^2(\mathbf{\tau}_m),\mathbf{\tau}_n-\mathbf{\tau}_m\right)_{\mathrm{L}^2}
\d t+ \mu_1\| \mathcal{J}_{n}\mathcal{S}(\mathbf{\tau}_n)-\mathcal{J}_{m} \mathcal{S}(\mathbf{\tau}_m)\|_{\mathrm{L}^2}^2\d t \notag \\
&\quad +2 \mu_1 \left(\mathcal{J}_n \mathcal{S}(\mathbf{\tau}_n)
-\mathcal{J}_m \mathcal{S}(\mathbf{\tau}_m),\mathbf{\tau}_n-\mathbf{\tau}_m\right)_{\mathrm{L}^2}\d W_2(t).
\end{align}
Using the fact that $\nabla \cdot \v_n=\nabla \cdot \v_m=0 $ we have directly $J_0=0.$ 
\begin{align} \label{ito.fin}
&\d\left(\mu_2\|\v_n-\v_m\|_{\mathrm{L}^2}^2+\mu_1 \|\mathbf{\tau}_n-\mathbf{\tau}_m\|_{\mathrm{L}^2}^2\right)
+2\mu_2 \nu\|\nabla(\v_n-\v_m)\|^2_{\mathrm{L}^2}\d
t\nonumber\\
&=\underbrace{-2 \mu_2 \left(\mathcal{J}_n [(\mathbf{v}_n \cdot\nabla)\v_n ]-\mathcal{J}_m [(\mathbf{v}_m \cdot\nabla)\v_m ]
,\v_n-\v_m \right)_{\mathrm{L}^2}}_{J_1}\d t\notag \\
&\quad\underbrace{-2 \mu_1\left(\mathcal{J}_n [(\mathbf{v}_n \cdot\nabla)\mathbf{\tau}_n ]-\mathcal{J}_m [(\mathbf{v}_m \cdot\nabla)\mathbf{\tau}_m ]
,\mathbf{\tau}_n-\mathbf{\tau}_m \right)_{\mathrm{L}^2}}_{J_2}\d t\notag \\
&\quad\underbrace{-2\mu_1\left(\mathcal{J}_n \Q(\mathbf{\tau}_n, \nabla \v_n)-
\mathcal{J}_m \Q(\mathbf{\tau}_m, \nabla \v_m), \mathbf{\tau}_n-\mathbf{\tau}_m\right)_{\mathrm{L}^2}}_{J_3} \d t-2a\mu_1 \|\mathbf{\tau}_n-\mathbf{\tau}_m\|_{\mathrm{L}^2}^2\d t
\nonumber\\
&\quad+ \underbrace{\mu_1\left(\mathcal{J}_{n}\mathcal{S}^2(\mathbf{\tau}_n)-\mathcal{J}_{m} \mathcal{S}^2(\mathbf{\tau}_m),\mathbf{\tau}_n-\mathbf{\tau}_m\right)_{\mathrm{L}^2}
}_{J_4}\d t+ \underbrace{\mu_1\|\mathcal{J}_{n}\mathcal{S}(\mathbf{\tau}_n)-\mathcal{J}_{m} \mathcal{S}(\mathbf{\tau}_m)\|_{\mathrm{L}^2}^2}_{J_5} \d t\notag \\
&\quad+\mu_2\|\mathcal{J}_n \sigma(t,\v_n)-\mathcal{J}_{m}\sigma(t,\v_{m})\|_{\mathcal{L}_Q(\mathrm{L}^2,\mathrm{L}^2)}^2\d t
\nonumber\\
&\quad+\mu_2\int_Z\|\mathcal{J}_{n} G(\v_n(t-),z)-\mathcal{J}_{m}G
(\v_m(t-),z)\|^2_{\mathrm{L}^2}N_1(\d t,\d z)\nonumber\\
&\quad+2\mu_2\left(\left(\mathcal{J}_n \sigma(t,\v_n)-\mathcal{J}_{m}\sigma(t,\v_{m})\right)\d W_1(t),\v_n-\v_m \right)_{\mathrm{L}^2}\notag\\
&\quad +2 \mu_1 \left(\mathcal{J}_n \mathcal{S}(\mathbf{\tau}_n)-\mathcal{J}_m \mathcal{S}(\mathbf{\tau}_m),\mathbf{\tau}_n-\mathbf{\tau}_m\right)_{\mathrm{L}^2}\d W_2(t)\nonumber\\
&\quad+2\mu_2\int_Z\left(\mathcal{J}_{n} G(\v_n(t-),z)-\mathcal{J}_{m} G(\v_m(t-),z),\v_n-\v_m \right)_{\mathrm{L}^2}\tilde{N}_1(\d
t,\d z).
\end{align}
Now  using \eqref{propS.2} and Remark \ref{esti.S}, $J_4$ can be simplified further as: 
\begin{align} \label{J_4}
|\mu_1 \left(\mathcal{J}_{n}\mathcal{S}^2(\mathbf{\tau}_n)-\mathcal{J}_{m} \mathcal{S}^2(\mathbf{\tau}_m),\mathbf{\tau}_n-\mathbf{\tau}_m\right)_{\mathrm{L}^2}|
  \leq \frac{C \mu_1}{n^{\epsilon}}\|h\|_{H^{s}}^2\|\mathbf{\tau}_n\|_{H^{s}}\|\mathbf{\tau}_n-\mathbf{\tau}_{m}\|_{\mathrm{L}^2}
+\mu_1 \|h\|^2_{\mathrm{L}^\infty}\|\mathbf{\tau}_n-\mathbf{\tau}_m\|_{\mathrm{L}^2}^2.
\end{align}

Using Lemma \ref{v}, we obtain
\begin{align} \label{J_1}
|J_1| \leq \frac{2C\sqrt{\mu_2}}{n^\epsilon}\|\v_n\|_{H^s}^2 \sqrt{\mu_2}\|\v_n-\v_m\|_{\mathrm{L}^2}+2 \mu_2 C \|\v_n-\v_m\|_{\mathrm{L}^2}^2 \|\nabla \v_n\|_{H^s}.
\end{align}
Using Lemma \ref{tau}, $J_2$ is reduced to
\begin{align} \label{J_2}
|J_2| \leq \frac{2 \sqrt{\mu_1}C}{n^\epsilon}(\|\v_n\|_{H^s}^2+\|\tau_n\|_{H^s}^2) \sqrt{\mu_1} \|\tau_n-\tau_m\|_{\mathrm{L}^2}+\frac{\nu \mu_2}{2}
\|\v_n-\v_m\|_{H^1}^2+\frac{2C \mu_1^2}{\nu \mu_2} \|\tau_n\|_{H^s}^2 \|\tau_n-\tau_m\|_{\mathrm{L}^2}^2.
\end{align}
Exploiting Lemma \ref{Q}, $J_3$ becomes:
\begin{align} \label{J_3}
|J_3| &\leq \frac{2 \sqrt{\mu_1}C}{n^\epsilon}(\|\v_n\|_{H^s}^2+\|\tau_n\|_{H^s}^2) \sqrt{\mu_1} \|\tau_n-\tau_m\|_{\mathrm{L}^2}+\frac{\nu \mu_2}{2}
\|\v_n-\v_m\|_{H^1}^2+\frac{2C \mu_1^2}{\nu \mu_2} \|\tau_m\|_{H^s}^2 \|\tau_n-\tau_m\|_{\mathrm{L}^2}^2\notag\\
&\quad +2\mu_1C \|\nabla \v_n\|_{H^s}\|\tau_n-\tau_m\|_{\mathrm{L}^2}^2+\frac{\mu_1 C}{n^\epsilon} \| \nabla \v_n\|_{H^s}^2 \|\tau_n-\tau_m\|_{\mathrm{L}^2}.
\end{align}

Exploiting \eqref{S.2} in Lemma \ref{M2}, $J_5,$ becomes
\begin{align} \label{J_5}
|J_5| \leq \frac{2 \mu_1}{n^{\epsilon}}\|h\|_{H^s}^2\|\tau_n\|_{H^s}^2+2\mu_1 \|h\|_{\mathrm{L}^\infty}^2 \|\tau_n-\tau_m\|_{\mathrm{L}^2}^2.
\end{align}

Combining \eqref{J_4},\eqref{J_1},\eqref{J_2},\eqref{J_3} and \eqref{J_5} together and using the fact that $ 4C\|\nabla \v_n\|_{H^s} \leq C^2+ 2\|\nabla \v_n\|_{H^s}^2 $, \eqref{ito.fin} becomes:
\begin{align}\label{estim}
&\d\left(\mu_2\|\v_n-\v_m\|_{\mathrm{L}^2}^2+\mu_1 \|\mathbf{\tau}_n-\mathbf{\tau}_m\|_{\mathrm{L}^2}^2\right)
+\mu_2 \nu\|\nabla(\v_n-\v_m)\|^2_{\mathrm{L}^2}\d
t\nonumber\\
&\leq \frac{C}{n^{\epsilon}}\Big(\left(2\sqrt{\mu_2}+4\sqrt{\mu_1}\right)\|\v_n\|_{H^{s}}^2
+4\sqrt{\mu_1}\|\mathbf{\tau}_n\|_{H^{s}}^2+\sqrt{\mu_1}\|h\|_{H^{s}}^2\|\mathbf{\tau}_n\|_{H^{s}}
+\sqrt{\mu_1}\|\nabla \v_n\|_{H^s}^2\Big)\notag\\
&\qquad \qquad \times \Big(\sqrt{\mu_2}\|\v_n-\v_{m}\|_{\mathrm{L}^2}+\sqrt{\mu_1}\|\mathbf{\tau}_n-\mathbf{\tau}_{m}\|_{\mathrm{L}^2} \Big)\notag \\
&\quad+\left((2C^2+2a)+2\|\nabla \v_n\|_{H^s}^2
+\frac{2C\mu_1}{\nu\mu_2}\left(\|\mathbf{\tau}_n\|_{H^{s}}^2+\|\tau_m\|_{H^s}^2\right)+2\|h\|_{\mathrm{L}^\infty}^2\right)\nonumber\\
&\qquad \qquad\times \Big(\mu_2\|\v_n-\v_{m}\|_{\mathrm{L}^2}^2+\mu_1\|\mathbf{\tau}_n-\mathbf{\tau}_{m}\|_{\mathrm{L}^2}^2 \Big)
+\frac{2\mu_1}{n^{\epsilon}}\|h\|_{H^{s}}^2\|\mathbf{\tau}_n\|_{H^{s}}^2 \notag \\
&\quad+\mu_2\|\mathcal{J}_n \sigma(t,\v_n)-\mathcal{J}_{m}\sigma(t,\v_{m})\|_{\mathcal{L}_Q(\mathrm{L}^2,\mathrm{L}^2)}^2\d t
\nonumber\\
&\quad+\mu_2\int_Z\|\mathcal{J}_{n} G(\v_n(t-),z)-\mathcal{J}_{m}G
(\v_m(t-),z)\|^2_{\mathrm{L}^2}N_1(\d t,\d z)\nonumber\\
&\quad+2\mu_2\left(\left(\mathcal{J}_n \sigma(t,\v_n)-\mathcal{J}_{m}\sigma(t,\v_{m})\right)\d W_1(t),\v_n-\v_m \right)_{\mathrm{L}^2}\notag\\
&\quad +2 \mu_1 \left(\mathcal{J}_n \mathcal{S}(\mathbf{\tau}_n)-\mathcal{J}_m \mathcal{S}(\mathbf{\tau}_m),\mathbf{\tau}_n-\mathbf{\tau}_m\right)_{\mathrm{L}^2}\d W_2(t)\nonumber\\
&\quad+2\mu_2\int_Z\left(\mathcal{J}_{n} G(\v_n(t-),z)-\mathcal{J}_{m} G(\v_m(t-),z),\v_n-\v_m \right)_{\mathrm{L}^2}\tilde{N}_1(\d
t,\d z).
\end{align}

\textbf{Step II:} \newline
Let us define the process
$\eta(t):=\exp\left(-2\mathlarger{\int_{0}^t}\|\nabla\v_n \|^2_{H^s}
\d s\right),$ $t\in[0,\Lm)$ and apply It\^{o} product formula
(see Theorem 4.4.13, Applebaum \cite{Ap}) to the process
$\eta(t)(\mu_2\|\v_n-\v_m\|_{\mathrm{L}^2}^2+\mu_1\|\mathbf{\tau}_n-\mathbf{\tau}_m\|^2_{\mathrm{L}^2})$ in the
interval $[0,t]$ to get
\begin{align} \label{ito.pro}
&\eta(t)\left(\mu_2\|\v_n-\v_m\|_{\mathrm{L}^2}^2+\mu_1 \|\mathbf{\tau}_n-\mathbf{\tau}_m\|_{\mathrm{L}^2}^2\right)
+\mu_2 \nu\int_0^t\eta(s)\|\nabla(\v_n-\v_m)\|^2_{\mathrm{L}^2}\d s\nonumber\\
&\leq\left(\mu_2\|\v_n(0)-\v_m(0)\|_{\mathrm{L}^2}^2+\mu_1 \|\mathbf{\tau}_n(0)-\mathbf{\tau}_m(0)\|_{\mathrm{L}^2}^2\right)\notag\\
&\quad +\frac{C}{n^{\epsilon}}\int_0^t\eta(s)\Big(\left(2\sqrt{\mu_2}+4\sqrt{\mu_1}\right)\|\v_n\|_{H^{s}}^2
+4\sqrt{\mu_1}\|\mathbf{\tau}_n\|_{H^{s}}^2+\sqrt{\mu_1}\|h\|_{H^{s}}^2\|\mathbf{\tau}_n\|_{H^{s}}
+\sqrt{\mu_1}\|\nabla \v_n\|_{H^s}^2\Big)\notag \\
&\qquad \qquad \times \Big(\sqrt{\mu_2}\|\v_n-\v_{m}\|_{\mathrm{L}^2}+\sqrt{\mu_1}\|\mathbf{\tau}_n-\mathbf{\tau}_{m}\|_{\mathrm{L}^2} \Big) \d s \notag \\
&\quad+\int_{0}^{t} \eta(s) \left((2C^2+2a)
+\frac{2C\mu_1}{\nu\mu_2}\left(\|\mathbf{\tau}_n\|_{H^{s}}^2+\|\tau_m\|_{H^s}^2\right)+2\|h\|_{\mathrm{L}^\infty}^2\right)\nonumber\\
&\qquad \qquad \times \Big(\mu_2\|\v_n-\v_{m}\|_{\mathrm{L}^2}^2+\mu_1\|\mathbf{\tau}_n-\mathbf{\tau}_{m}\|_{\mathrm{L}^2}^2 \Big) \d s
+\frac{2\mu_1 \|h\|_{H^{s}}^2}{n^{\epsilon}} \int_{0}^{t} \eta(s) \|\mathbf{\tau}_n\|_{H^{s}}^2 \d s\notag \\
&\quad+\mu_2 \int_{0}^{t} \eta(s) \|\mathcal{J}_n \sigma(s,\v_n)-\mathcal{J}_{m}\sigma(s,\v_{m})\|_{\mathcal{L}_Q(\mathrm{L}^2,\mathrm{L}^2)}^2
 \d s\nonumber\\
&\quad+\mu_2 \int_{0}^{t} \eta(s) \int_Z\|\mathcal{J}_{n} G(\v_n(s-),z)-\mathcal{J}_{m}G
(\v_m(s-),z)\|^2_{\mathrm{L}^2}N_1(\d s,\d z)\nonumber\\
&\quad+2\mu_2\int_{0}^{t} \eta(s)\left(\left(\mathcal{J}_n \sigma(s,\v_n)-\mathcal{J}_{m}\sigma(s,\v_{m})\right)\d W_1(s),\v_n-\v_m \right)_{\mathrm{L}^2}\notag\\
&\quad +2 \mu_1 \int_{0}^{t} \eta(s)\left(\mathcal{J}_n \mathcal{S}(\mathbf{\tau}_n)-\mathcal{J}_m \mathcal{S}(\mathbf{\tau}_m),\mathbf{\tau}_n-\mathbf{\tau}_m\right)_{\mathrm{L}^2}\d W_2(s)\nonumber\\
&\quad+2\mu_2\int_{0}^{t} \eta(s)\int_Z\left(\mathcal{J}_{n} G(\v_n(s-),z)-\mathcal{J}_{m} G(\v_m(s-),z),\v_n-\v_m \right)_{\mathrm{L}^2}\tilde{N}_1(\d
s,\d z).
\end{align}
It is to be noted that the quadratic variation of the product of these two adapted processes $Y_1(t)=\eta(t)$ and $Y_2(t)=\mu_2\|\v_n-\v_m\|_{\mathrm{L}^2}^2+\mu_1\|\mathbf{\tau}_n-\mathbf{\tau}_m\|^2_{\mathrm{L}^2},$ i.e., $[Y_1,Y_2](t)$ is zero (see, Section 4.4.3, page-257 of Applebaum \cite{Ap}).
\newline
 Let us now take the supremum from $0$ to
$T\wedge\Rn$, for any $T>0$ in  \eqref{ito.pro} and then on taking expectation and thereafter using \eqref{mt1} (in Remark \ref{nlam}), we get,
\begin{align} \label{Cau.1}
&\mathbb{E}\Big[\sup_{0\leq t\leq\Lm}\eta(t)\Big(\mu_2\|\v_n-\v_m\|_{\mathrm{L}^2}^2+\mu_1 \|\mathbf{\tau}_n-\mathbf{\tau}_m\|_{\mathrm{L}^2}^2\Big)\Big]
+\mu_2 \nu\mathbb{E}\left[\int_0^{\Lm}\eta(s)\|\nabla(\v_n-\v_m)\|^2_{\mathrm{L}^2}\d s\right]\nonumber\\
&\leq \mathbb{E}\Big(\mu_2\|\v_n(0)-\v_m(0)\|_{\mathrm{L}^2}^2+\mu_1 \|\mathbf{\tau}_n(0)-\mathbf{\tau}_m(0)\|_{\mathrm{L}^2}^2\Big)\notag\\
&\quad +\frac{C}{n^{\epsilon}}\mathbb{E}\Big[\int_0^{\Lm}\eta(s)\Big(\left(2\sqrt{\mu_2}+4\sqrt{\mu_1}\right)\|\v_n\|_{H^{s}}^2
+4\sqrt{\mu_1}\|\mathbf{\tau}_n\|_{H^{s}}^2+\sqrt{\mu_1}\|h\|_{H^{s}}^2\|\mathbf{\tau}_n\|_{H^{s}}
+\sqrt{\mu_1}\|\nabla \v_n\|_{H^s}^2\Big)\notag \\
&\qquad \qquad \times \Big(\sqrt{\mu_2}\|\v_n-\v_{m}\|_{\mathrm{L}^2}+\sqrt{\mu_1}\|\mathbf{\tau}_n-\mathbf{\tau}_{m}\|_{\mathrm{L}^2} \Big) \d s \Big]\notag \\
&\quad+\mathbb{E}\Big[\int_0^{\Lm}\eta(s)\left((2C^2+2a)
+\frac{2C\mu_1}{\nu\mu_2}\left(\|\mathbf{\tau}_n\|_{H^{s}}^2+\|\tau_m\|_{H^s}^2\right)+2\|h\|_{\mathrm{L}^\infty}^2\right)\nonumber\\
&\qquad \qquad \times \Big(\mu_2\|\v_n-\v_{m}\|_{\mathrm{L}^2}^2+\mu_1\|\mathbf{\tau}_n-\mathbf{\tau}_{m}\|_{\mathrm{L}^2}^2 \Big) \d s \Big]\notag\\
&\quad +\frac{2\mu_1}{n^{\epsilon}}\|h\|_{H^{s}}^2\mathbb{E}\Big[\int_0^{\Lm}\eta(s)\|\mathbf{\tau}_n\|_{H^{s}}^2\d s\Big] 
+\mu_2\mathbb{E}\Big[\int_0^{\Lm}\eta(s)\|\mathcal{J}_n \sigma(s,\v_n)-\mathcal{J}_{m}\sigma(s,\v_{m})
\|_{\mathcal{L}_Q(\mathrm{L}^2,\mathrm{L}^2)}^2\d s\Big]
\nonumber\\
&\quad+\mu_2\mathbb{E}\Big[\int_0^{\Lm}\int_Z\eta(s)\|\mathcal{J}_{n} G(\v_n(s),z)-\mathcal{J}_{m}G
(\v_m(s),z)\|^2_{\mathrm{L}^2}\lambda(\d z) \d s\Big]\nonumber\\
&\quad+\underbrace{2\mu_2\mathbb{E}\Big[\sup_{0\leq t\leq\Lm}\int_0^{t}\eta(s)\left(\left(\mathcal{J}_n \sigma(s,\v_n)-\mathcal{J}_{m}
\sigma(s,\v_{m})\right)\d W_1(s),\v_n-\v_m \right)_{\mathrm{L}^2}\Big]}_{J_6}\notag\\
&\quad +\underbrace{2 \mu_1 \mathbb{E}\Big[\sup_{0\leq t\leq\Lm}\int_0^{t}\eta(s)\left(\mathcal{J}_n S(\mathbf{\tau}_n)-\mathcal{J}_m S(\mathbf{\tau}_m),
\mathbf{\tau}_n-\mathbf{\tau}_m\right)_{\mathrm{L}^2}\d W_2(s)\Big]}_{J_7}\nonumber\\
&\quad+\underbrace{2\mu_2\mathbb{E}\Big[\sup_{0\leq t\leq\Lm}\int_0^{t}\int_Z\eta(s)\left(\mathcal{J}_{n} G(\v_n(s-),z)-\mathcal{J}_{m}
G(\v_m(s-),z),\v_n-\v_m \right)_{\mathrm{L}^2}\tilde{N}_1(\d s,\d z)\Big]}_{J_8}
\end{align} Using Lemma \ref{M1}, we have
\begin{align} \label{M1.cau}
|J_6| &\leq \frac{\mu_2}{4}\mathbb{E}\left(\sup_{0\leq t\leq \Lm}\eta(t)\|\v_n(t)-\v_m(t)\|^2_{\mathrm{L}^2}\right) \notag\\
  & \quad +\frac{8CK\mu_2}{n^{\epsilon}}\mathbb{E}\left[\int_{0}^{\Lm}\eta(t)(1+\|\v_n\|_{H^s}^2)\d
t\right]+8L\mu_2 \,\mathbb{E}\left[\int_{0}^{\Lm}\eta(t)\|\v_n-\v_{m}\|_{\mathrm{L}^2}^2\d
t\right].
\end{align} 
Exploiting Lemma \ref{M2}, we achieve
\begin{align}\label{M2.cau}
|J_7| &\leq \frac{\mu_1}{2}\mathbb{E}\left(\sup_{0\leq t\leq \Lm}\eta(t)\|\mathbf{\tau}_n(t)-\mathbf{\tau}_m(t)\|^2_{\mathrm{L}^2}\right)
+\frac{8 \mu_1}{n^{\epsilon}}\mathbb{E}\left(\int_0^{\Lm}\eta(t)\|h\|_{H^{s}}^2\|\mathbf{\tau}_n\|_{H^{s}}^2 \d t\right) \notag\\
&\quad +8 \mu_1\|h\|^2_{\mathrm{L}^\infty}\mathbb{E}\left(\int_0^{\Lm}\eta(t)\|\mathbf{\tau}_n-\mathbf{\tau}_m\|_{\mathrm{L}^2}^2 \d t\right).
\end{align}
Exploiting Lemma \ref{M3}, we have
\begin{align}\label{M3.cau}
|J_8| &\leq \frac{\mu_2}{4}\mathbb{E}\left(\sup_{0\leq t\leq \Lm}\eta(t)\|\v_n(t)-\v_m(t)\|^2_{\mathrm{L}^2}\right)\notag \\
&\leq \frac{8\mu_2CK}{n^{\epsilon}}\mathbb{E}\left[\int_{0}^{\Lm}\eta(t)(1+\|\v_n\|_{H^s}^2)d
t\right]+8\mu_2L\,\mathbb{E}\left[\int_{0}^{\Lm}\eta(t)\|\v_n-\v_{m}\|_{\mathrm{L}^2}^2dt\right].
\end{align}
Combining \eqref{M1.cau}, \eqref{M2.cau} and \eqref{M3.cau}, and using \eqref{s.tau.2}, \eqref{G.lam.2} and further using the fact that $$ \sqrt{\mu_2}\|\v_n-\v_{m}\|_{\mathrm{L}^2}+\sqrt{\mu_1}\|\mathbf{\tau}_n-\mathbf{\tau}_{m}\|_{\mathrm{L}^2} \leq C (1+\mu_2\|\v_n-\v_{m}\|_{\mathrm{L}^2}^2+\mu_1\|\mathbf{\tau}_n-\mathbf{\tau}_{m}\|_{\mathrm{L}^2}^2), $$  equation \eqref{Cau.1} is reduced to

\begin{align}
&\mathbb{E}\Big[\sup_{0\leq t\leq\Lm}\eta(t)\Big(\mu_2\|\v_n-\v_m\|_{\mathrm{L}^2}^2+\mu_1 \|\mathbf{\tau}_n-\mathbf{\tau}_m\|_{\mathrm{L}^2}^2\Big)\Big]
+2\mu_2 \nu\mathbb{E}\left[\int_0^{\Lm}\eta(s)\|\nabla(\v_n-\v_m)\|^2_{\mathrm{L}^2}\d s\right]\nonumber\\
&\leq 2\mathbb{E}\Big(\mu_2\|\v_n(0)-\v_m(0)\|_{\mathrm{L}^2}^2+\mu_1 \|\mathbf{\tau}_n(0)-\mathbf{\tau}_m(0)\|_{\mathrm{L}^2}^2\Big)\notag\\
&\quad +\frac{2\sqrt{2}C}{n^{\epsilon}}\mathbb{E}\Big[\int_0^{\Lm}\eta(s)\Big(\left(2\sqrt{\mu_2}+4\sqrt{\mu_1}\right)\|\v_n\|_{H^{s}}^2
+4\sqrt{\mu_1}\|\mathbf{\tau}_n\|_{H^{s}}^2+\sqrt{\mu_1}\|h\|_{H^{s}}^2\|\mathbf{\tau}_n\|_{H^{s}}\notag\\
&\qquad \qquad +\sqrt{\mu_1}\|\nabla \v_n\|_{H^s}^2\Big)
 \times \Big(1+\mu_2\|\v_n-\v_{m}\|_{\mathrm{L}^2}^2+\mu_1\|\mathbf{\tau}_n-\mathbf{\tau}_{m}\|_{\mathrm{L}^2}^2 \Big) \d s \Big]\notag \\
&\quad+2\mathbb{E}\Big[\int_0^{\Lm}\eta(s)\left((2C^2+2a+18L)
+\frac{2C\mu_1}{\nu\mu_2}\left(\|\mathbf{\tau}_n\|_{H^{s}}^2
+\|\mathbf{\tau}_m\|_{H^{s}}^2\right)+10\|h\|_{\mathrm{L}^\infty}^2\right)\nonumber\\
&\qquad \qquad \times \Big(\mu_2\|\v_n-\v_{m}\|_{\mathrm{L}^2}^2+\mu_1\|\mathbf{\tau}_n-\mathbf{\tau}_{m}\|_{\mathrm{L}^2}^2 \Big) \d s \Big]\notag\\
&\quad +\frac{1}{n^{\epsilon}}\left(20\|h\|_{H^{s}}^2
+36CK\right)\mathbb{E}\Big[\int_0^{\Lm}\eta(s)\left(\mu_2+\mu_2\|\v_n\|_{H^s}^2+\mu_1\|\mathbf{\tau}_n\|_{H^{s}}^2\right)\d s\Big] .
\end{align}
Using definition of stopping time,  H\"{o}lder's inequality and rearranging, the above inequality reduces to
\begin{align} \label{Cau.7}
&\mathbb{E}\Big[\sup_{0\leq t\leq\Lm}\eta(t)\Big(\mu_2\|\v_n-\v_m\|_{\mathrm{L}^2}^2+\mu_1 \|\mathbf{\tau}_n-\mathbf{\tau}_m\|_{\mathrm{L}^2}^2\Big)\Big]
+2\mu_2 \nu\mathbb{E}\left[\int_0^{\Lm}\eta(s)\|\nabla(\v_n-\v_m)\|^2_{\mathrm{L}^2}\d s\right]\nonumber\\
&\leq 2\mathbb{E}\Big(\mu_2\|\v_n(0)-\v_m(0)\|_{\mathrm{L}^2}^2
+\mu_1 \|\mathbf{\tau}_n(0)-\mathbf{\tau}_m(0)\|_{\mathrm{L}^2}^2\Big)\notag\\
&\quad +\frac{1}{n^{\epsilon}}\frac{\sqrt{2\mu_1}CN}{\mu_2\nu}\mathbb{E}\Big[\sup_{0\leq t\leq\Lm}\eta(t)
\Big(\mu_2\|\v_n-\v_{m}\|_{\mathrm{L}^2}^2
+\mu_1\|\mathbf{\tau}_n-\mathbf{\tau}_{m}\|_{\mathrm{L}^2}^2 \Big) \Big]\notag \\
&\quad+\left(4\left((C^2+a+9L)
+\frac{2CN}{\nu\mu_2}+5\|h\|_{\mathrm{L}^\infty}^2\right)
+\frac{2\sqrt{2}C}{n^{\epsilon}}\left(\frac{(2\sqrt{\mu_2}
+4\sqrt{\mu_1})N}{\mu_2}+\frac{4N}{\sqrt{\mu_1}}+\|h\|_{H^{s}}^2N\right)\right)\nonumber\\
&\qquad \qquad \times \mathbb{E}\Big[\int_0^{\Lm}\eta(s)\Big(\mu_2\|\v_n-\v_{m}\|_{\mathrm{L}^2}^2
+\mu_1\|\mathbf{\tau}_n-\mathbf{\tau}_{m}\|_{\mathrm{L}^2}^2 \Big) \d s \Big]\notag\\
&\quad +\frac{1}{n^{\epsilon}}\left(\left(20\|h\|_{H^{s}}^2
+36CK\right)\left(\mu_2+N\right)+\frac{2\sqrt{2}C}{n^\epsilon}\left(\frac{(2\sqrt{\mu_2}
+4\sqrt{\mu_1})N}{\mu_2}+\frac{4N}{\sqrt{\mu_1}}+\|h\|_{H^{s}}^2N\right)\right)\nonumber\\
&\qquad \qquad \times\mathbb{E}\Big[\int_0^{\Lm}\eta(s)\d s\Big]
+\frac{1}{n^{\epsilon}}\frac{\sqrt{2\mu_1}CN}{\mu_2\nu}.
\end{align}
Note that the second term on the right hand side of \eqref{Cau.7} can be balanced with the first term of the left hand side of \eqref{Cau.7} for sufficiently large $n$, so that 
$ \frac{2\sqrt{2}CN\sqrt{\mu_1}}{n^{\epsilon}} <<1.$
 Therefore, using
$$\mathbb{E}\left[\int_{0}^{\Lm}\eta(t)\d
t\right]\leq \mathbb{E}\left[\int_{0}^{T}\eta(t)\d
t\right]=\mathbb{E}\left[\int_{0}^{T}\exp\left(-2C\int_{0}^t\|\nabla\v_n\|_{H^s}^2\d
s\right)\d t\right]\leq T,$$
finally \eqref{Cau.7} becomes
\begin{align}\label{last.1}
&\mathbb{E}\Big[\sup_{0\leq t\leq\Lm}\eta(t)\Big(\mu_2\|\v_n-\v_m\|_{\mathrm{L}^2}^2
+\mu_1 \|\mathbf{\tau}_n-\mathbf{\tau}_m\|_{\mathrm{L}^2}^2\Big)\Big]
+4\mu_2 \nu\mathbb{E}\left[\int_0^{\Lm}\eta(s)\|\nabla(\v_n-\v_m)\|^2_{\mathrm{L}^2}\d s\right]\nonumber\\
&\leq 4\mathbb{E}\Big(\mu_2\|\v_n(0)-\v_m(0)\|_{\mathrm{L}^2}^2
+\mu_1 \|\mathbf{\tau}_n(0)-\mathbf{\tau}_m(0)\|_{\mathrm{L}^2}^2\Big) +\frac{1}{n^{\epsilon}}\frac{\sqrt{2\mu_1}CN}{\mu_2\nu} \notag\\
&\quad +\frac{2T}{n^{\epsilon}}\left(\left(20\|h\|_{H^{s}}^2
+36CK\right)\left(\mu_2+N\right)+2\sqrt{2}C\left(\frac{(2\sqrt{\mu_2}
+4\sqrt{\mu_1})N}{\mu_2}+\frac{4N}{\sqrt{\mu_1}}+\|h\|_{H^{s}}^2N\right)\right)\notag\\
&\quad+2\left(4\left((C^2+a+9L)
+\frac{2CN}{\nu\mu_2}+5\|h\|_{\mathrm{L}^\infty}^2\right)
+\frac{2\sqrt{2}C}{n^{\epsilon}}\left(\frac{(2\sqrt{\mu_2}
+4\sqrt{\mu_1})N}{\mu_2}+\frac{4N}{\sqrt{\mu_1}}+\|h\|_{H^{s}}^2N\right)\right)\nonumber\\
&\qquad \qquad \times \mathbb{E}\Big[\int_0^{\Lm}\sup_{0\leq s\leq t}\eta(s)\Big(\mu_2\|\v_n-\v_{m}\|_{\mathrm{L}^2}^2
+\mu_1\|\mathbf{\tau}_n-\mathbf{\tau}_{m}\|_{\mathrm{L}^2}^2 \Big) \d t \Big].
\end{align}
An application of standard
Gronwall's inequality yields
\begin{align}\label{gron}
&\mathbb{E}\Big[\sup_{0\leq t\leq\Lm}\eta(t)\Big(\mu_2\|\v_n-\v_m\|_{\mathrm{L}^2}^2+\mu_1 \|\mathbf{\tau}_n-\mathbf{\tau}_m\|_{\mathrm{L}^2}^2\Big)\Big]
+4\mu_2 \nu\mathbb{E}\left[\int_0^{\Lm}\eta(s)\|\nabla(\v_n-\v_m)\|^2_{\mathrm{L}^2}\d s\right]\nonumber\\
&\leq \Big(4\mathbb{E}\left[\mu_2\|\v_n(0)-\v_m(0)\|_{\mathrm{L}^2}^2
+\mu_1 \|\mathbf{\tau}_n(0)-\mathbf{\tau}_m(0)\|_{\mathrm{L}^2}^2\right] + C_1\Big)e^{C_2 T},
\end{align}
where 
\begin{align*}
C_1
&=\frac{2T}{n^{\epsilon}}\left(\left(20\|h\|_{H^{s}}^2
+36CK\right)\left(\mu_2+N\right)+2\sqrt{2}C\left(\frac{(2\sqrt{\mu_2}
+4\sqrt{\mu_1})N}{\mu_2}+\frac{4N}{\sqrt{\mu_1}}+\|h\|_{H^{s}}^2N\right)\right)\\ &\quad+\frac{1}{n^{\epsilon}}\frac{\sqrt{2\mu_1}CN}{\mu_2\nu}
\end{align*}
and $$ C_2=2\left(4\left((C^2+a+9L)
+\frac{2CN}{\nu\mu_2}+5\|h\|_{\mathrm{L}^\infty}^2\right)
+\frac{2\sqrt{2}C}{n^{\epsilon}}\left(\frac{(2\sqrt{\mu_2}
+4\sqrt{\mu_1})N}{\mu_2}+\frac{4N}{\sqrt{\mu_1}}+\|h\|_{H^{s}}^2N\right)\right).
$$ The
right hand side of (\ref{gron}) tends to zero, since
$\v_n(0)=\mathcal{J}_n\v_0$, $\v_m(0)=\mathcal{J}_m\v_0$,
$\mathbf{\tau}_n(0)=\mathcal{J}_n\mathbf{\tau}_0$ and $\mathbf{\tau}_m(0)=\mathcal{J}_m\mathbf{\tau}_0$,
as $n,m\to\infty$. Also from (\ref{last.1}) and (\ref{gron}), we
get
\begin{align}\label{last.2}
\mathbb{E}\left[\int_0^{\Lm}\eta(t)\|\nabla(\v_n-\v_m)\|^2_{\mathrm{L}^2}\d
t\right]\to 0\text{ as }n,m\to\infty.
\end{align}
Since $\eta(\cdot)$ is a bounded measurable $\mathscr{F}_t$ adapted process, it directly yields the required results $(i)$ and $(ii).$ 
\end{proof}

\begin{rem}\label{stopping} As a cosequence of Theorems \ref{positive1} and \ref{cauchy}, we conclude that there exists a stopping time $\xi_N$ and processes $(\v,\\tau)$ such that $(\v_n,\tau_n) \rightarrow (\v,\tau)$ in $\mathrm{L}^2(\Omega;\mathrm{L}^{\infty}(0,\xi_N \wedge T;\mathrm{L}^2(\mathbb{R}^d))).$ 
We later (in Theorem \ref{existence}) identify $\xi_N$  as $\rho_N$ (as defined in the Main Result \ref{MR}), which is the pointwise limit of $\rho_N^n$.
\end{rem}


\section{Existence and Uniqueness of Local Strong Solutions}

\begin{prop} \label{conv.v}
For any $s'<s$ with $s'>d/2$ and $T>0$, the following convergences hold:
\begin{itemize}
\item[(i)] the family of solutions $(\v_n,\mathbf{\tau}_n)\to(\v,\mathbf{\tau})$ strongly in the space 
$\mathrm{L}^2(\Omega;\mathrm{L}^{\infty}(0,\xi_N \wedge T;H^{s'}(\mathbb{R}^d)))$ as $n\to \infty$
\item[(ii)]$\nabla\v_n \to\nabla\v$ strongly in the space
$\mathrm{L}^2(\Omega;\mathrm{L}^2(0,\xi_N \wedge T;H^{s'}(\mathbb{R}^d)))$ as $n\to \infty$
 \item[(iii)]$\Delta\v_n\to\Delta\v$ strongly in the space
$\mathrm{L}^2(\Omega;\mathrm{L}^2(0,\xi_N \wedge T;H^{s'-1}(\mathbb{R}^d)))$ as $n\to \infty$
\item[(iv)] $\nabla \cdot \tau_n \to \nabla \cdot \tau$ strongly in the space $\mathrm{L}^2(\Omega;\mathrm{L}^\infty(0,\xi_N \wedge T;H^{s'-1}(\mathbb{R}^d)))$ as $n\to \infty$
 \item[(v)] $\mathcal{D}(\v_n)\to \mathcal{D}(\v)$ strongly in the space $\mathrm{L}^2(\Omega;\mathrm{L}^2(0,\xi_N \wedge T;H^{s'}(\mathbb{R}^d)))$ as $n\to \infty$.
\end{itemize} 

\end{prop}

\begin{proof}
We first prove (i). It follows from (\ref{gron}) that the sequence of solutions $(\v_n,\tau_n)\to(\v,\tau)$ strongly in
$\mathrm{L}^2(\Omega;\mathrm{L}^{\infty}(0,\xi_N\wedge T,\mathrm{L}^2(\mathbb{R}^d)))$.
Also from the estimate (\ref{last.2}), we have
$\nabla\v_n\to\nabla\v$ strongly in
$\mathrm{L}^2(\Omega;\mathrm{L}^2(0,\xi_N\wedge T;\mathrm{L}^2(\mathbb{R}^d)))$.
 Exploiting the interpolation inequality
$\Big($Lemma \autoref{iss}, with exponents
$\mathlarger{\frac{s}{s-s'}}$ and $\mathlarger{\frac{s}{s'}}\Big)$ and H\"{o}lder's inequality for $0<s'<s$, we
obtain
\begin{align}
&\mathbb{E}\left[\sup_{0\leq t\leq\xi_N \wedge T}\|\v_n-\v\|_{H^{s'}}^2\right]
\leq C \left\{\mathbb{E}\left[\sup_{0\leq
t\leq\xi_N \wedge T}\|\v_n-\v\|_{\mathrm{L}^2}^2\right]\right\}^{1-s'/s}
\left\{\mathbb{E}\left[\sup_{0\leq
t\leq\xi_N \wedge T}\|\v_n-\v\|_{H^s}^2\right]\right\}^{s'/s}
\nonumber\\ 
&\leq C\left\{\mathbb{E}\left[\sup_{0\leq
t\leq\xi_N \wedge T}\|\v_n-\v\|_{\mathrm{L}^2}^2\right]\right\}^{1-s'/s}
\left\{\mathbb{E}\left[\sup_{0\leq t\leq\xi_N \wedge T}\|\v_n\|_{H^s}^2+\sup_{0\leq
t\leq\xi_N \wedge T}\|\v\|_{H^s}^2\right]\right\}^{s'/s}\nonumber\\
&\leq (2N)^{s'/s}C \left\{\mathbb{E}\left[\sup_{0\leq
t\leq\xi_N \wedge T}\|\v_n-\v\|_{\mathrm{L}^2}^2\right]\right\}^{1-s'/s}
\to 0, \quad \textit{as}\quad n\to\infty.
\end{align}
\noindent
Combining  Remark \ref{ener.estim} and Theorem \autoref{cauchy} and
using Sobolev interpolation for any $s'<s,$ we infer $(\v_n,\tau_n)\to(\v,\tau)$ strongly in
$\mathrm{L}^2(\Omega;\mathrm{L}^{\infty}(0,\xi_N \wedge T;H^{s'}(\mathbb{R}^d))).$ This proves (i). 
\par
\noindent
We also have $\nabla\v_n \to\nabla\v$ strongly in
$\mathrm{L}^2(\Omega;\mathrm{L}^2(0,\xi_N \wedge T;\mathrm{L}^2(\mathbb{R}^d)))$, hence for any $s'<s,$ using Sobolev interpolation and similar arguments as above, we have  $\nabla\v_n\to\nabla\v$
strongly in
$\mathrm{L}^2(\Omega;\mathrm{L}^2(0,\xi_N \wedge T;H^{s'}(\mathbb{R}^d))).$
This directly implies (ii). 
\par
\noindent
Proceeding in similar manner as above we directly have (iii), (iv) and (v).
\end{proof}

\begin{prop} \label{conv.non}
For any $s'>d/2$ and $T>0$,
\begin{itemize}
\item[(i)] the quadratic form $\mathcal{J}_n \Q(\mathbf{\tau}_n ,\nabla\v_n)\to \Q(\mathbf{\tau},\nabla\v)$
 strongly in $\mathrm{L}^1(\Omega;\mathrm{L}^{\infty}(0,\xi_N \wedge T;H^{s'-1}(\mathbb{R}^d)))$
as $n\to \infty$
and
\item[(ii)] the non-linear term
$\mathcal{J}_n[(\v_n \cdot\nabla)\mathbf{\tau}_n]\to(\v\cdot\nabla)\mathbf{\tau}$ strongly
in
$\mathrm{L}^1(\Omega;\mathrm{L}^{2}(0,\xi_N \wedge T;H^{s'-1}(\mathbb{R}^d)))$
as $n\to \infty.$
\end{itemize}

\end{prop}

\begin{proof}
For $s'>d/2$, by using (\ref{intro1}), (\ref{intro2a}), bilinear property of $\Q$, Remark \ref{div} and H\"{o}lder's inequality, for $0<\epsilon<1$, we have
\begin{align}
&\mathbb{E}\left[\sup_{0\leq t\leq
\xi_N \wedge T}\|\mathcal{J}_n \Q(\mathbf{\tau}_n,\nabla\v_n)-\Q(\mathbf{\tau},\nabla\v) \|_{H^{s'-1}}\right]
\nonumber\\&\leq  \mathbb{E}\left[\sup_{0\leq t\leq
\xi_N \wedge T}\|\mathcal{J}_n \Q((\mathbf{\tau}_n-\mathbf{\tau}),\nabla\v_n)\|_{H^{s'-1}}\right]+\mathbb{E}\left[\sup_{0\leq
t\leq\xi_N \wedge T}\|\mathcal{J}_n \Q(\mathbf{\tau},\nabla(\v_n-\v))\|_{H^{s'-1}}\right]\nonumber\\&\quad+\mathbb{E}\left[\sup_{0\leq
t\leq \xi_N \wedge T}\|\mathcal{J}_n \Q(\mathbf{\tau},\nabla\v)-\Q(\mathbf{\tau},\nabla\v)\|_{H^{s'-1}}\right] \nonumber\\&\leq
C\mathbb{E}\left[\sup_{0\leq t\leq
\xi_N \wedge T}\| \Q((\mathbf{\tau}_n-\mathbf{\tau}),\nabla\v_n)\|_{H^{s'-1}}\right]+C\mathbb{E}\left[\sup_{0\leq
t\leq
\xi_N \wedge T}\|\Q(\mathbf{\tau},\nabla(\v_n-\v))\|_{H^{s'-1}}\right]\nonumber\\&\quad+\frac{C}{n^{\epsilon}}\mathbb{E}\left[\sup_{0\leq
t\leq \xi_N \wedge T}\|\Q(\mathbf{\tau},\nabla\v)\|_{H^{s'-1+\epsilon}}\right]\nonumber\\
&\leq C\mathbb{E}\left[\sup_{0\leq t\leq\xi_N \wedge T}\left(\|\mathbf{\tau}_n-\mathbf{\tau}\|_{\mathrm{L}^\infty}\|\nabla\v_n\|_{H^{s'}}\right)
+\sup_{0\leq t\leq\xi_N \wedge T}\left(\|\mathbf{\tau}_n-\mathbf{\tau}\|_{H^{s'}}\|\nabla\v_n\|_{\mathrm{L}^\infty}\right)\right]\nonumber\\
& \quad +C\mathbb{E}\left[\sup_{0\leq t\leq \xi_N \wedge T}\left(\|\nabla(\v_n-\v)\|_{H^{s'}}\|\mathbf{\tau}\|_{\mathrm{L}^\infty}\right)+
\sup_{0\leq t\leq \xi_N \wedge T}\left(\|\nabla(\v_n-\v)\|_{\mathrm{L}^\infty}\|\mathbf{\tau}\|_{H^{s'}}\right)\right]\nonumber\\
&\quad +\frac{C}{n^{\epsilon}}\mathbb{E}\left[\sup_{0\leq t\leq \xi_N \wedge T}\|\v\|_{H^{s'+\epsilon-1}}\|\mathbf{\tau}\|_{\mathrm{L}^\infty}+
\sup_{0\leq t\leq \xi_N \wedge T}\|\nabla\v\|_{\mathrm{L}^\infty}\|\mathbf{\tau}\|_{H^{s'+\epsilon-1}}\right]
\nonumber\\
&\leq 2C\mathbb{E}\left[\sup_{0\leq t\leq\xi_N \wedge T}\left(\|\mathbf{\tau}_n-\mathbf{\tau}\|_{H^{s'}}\|\nabla\v_n\|_{H^{s'}}\right)\right]
+2C\mathbb{E}\left[\sup_{0\leq t\leq \xi_N \wedge T}\left(\|\nabla(\v_n-\v)\|_{H^{s'}}\|\mathbf{\tau}\|_{H^{s'}}\right)\right]\nonumber\\
& \quad +\frac{C}{n^{\epsilon}}\mathbb{E}\left[\sup_{0\leq t\leq \xi_N \wedge T}\|\v\|_{H^{s'+\epsilon}}\|\mathbf{\tau}\|_{H^{s'}}+
\sup_{0\leq t\leq \xi_N \wedge T}\|\nabla\v\|_{H^{s'}}\|\mathbf{\tau}\|_{H^{s'+\epsilon-1}}\right]
\nonumber\\
&\leq 2C\left[\mathbb{E}\left(\sup_{0\leq t\leq \xi_N \wedge T}\|\mathbf{\tau}_n-\mathbf{\tau}\|_{H^{s'}}^2\right)\right]^{1/2}\left[\mathbb{E}\left(\sup_{0\leq
t\leq \xi_N \wedge T}\|\nabla\v_n\|_{H^{s'}}^2\right)\right]^{1/2}\nonumber\\&\quad+
2C\left[\mathbb{E}\left(\sup_{0\leq t\leq \xi_N \wedge T}\|\nabla(\v_n-\v)\|_{H^{s'}}^2\right)\right]^{1/2}\left[\mathbb{E}\left(\sup_{0\leq
t\leq \xi_N \wedge T}\|\mathbf{\tau}\|_{H^{s'}}^2\right)\right]^{1/2}\nonumber\\
&\quad+\frac{C}{2n^{\epsilon}}\mathbb{E}\left[\sup_{0\leq t\leq \xi_N \wedge T}\|\v\|_{H^{s'+\epsilon}}^2\right]+\frac{C}{2n^{\epsilon}}
\mathbb{E}\left[\sup_{0\leq t\leq \xi_N \wedge T}\|\mathbf{\tau}\|_{H^{s'+\epsilon}}^2\right]\nonumber\\
&\quad+\frac{C}{2n^{\epsilon}}\mathbb{E}\left[\sup_{0\leq t\leq \xi_N \wedge T}\|\nabla\v\|_{H^{s'+\epsilon}}^2\right]+\frac{C}{2n^{\epsilon}}
\mathbb{E}\left[\sup_{0\leq t\leq \xi_N \wedge T}\|\mathbf{\tau}\|_{H^{s'+\epsilon}-1}^2\right]\nonumber\\
&\to 0\text{ as }n\to\infty.
\end{align}
This completes the proof for (i).
\par\noindent
For $s'>d/2$, since $(\v_n,\mathbf{\tau}_n)\to(\v,\mathbf{\tau})$ strongly in
$\mathrm{L}^2(\Omega;\mathrm{L}^{\infty}(0,\xi_N \wedge T;H^{s'}(\mathbb{R}^d)))$, we proceed in the similar way as in the proof of (i) to infer  
$\mathcal{J}_n[(\v_n \cdot\nabla)\mathbf{\tau}_n]\to(\v\cdot\nabla)\mathbf{\tau}$ strongly 
in $\mathrm{L}^1(\Omega;\mathrm{L}^{\infty}(0,\xi_N \wedge T;H^{s'-1}(\mathbb{R}^d))).$
Since 
\begin{align} \label{emb}
\mathrm{L}^{1}(\Omega;\mathrm{L}^{\infty}(0,\xi_N \wedge T;H^{s'-1}(\mathbb{R}^d)))
\subset\mathrm{L}^{1}(\Omega;\mathrm{L}^{2}(0,\xi_N \wedge T;H^{s'-1}(\mathbb{R}^d))),
\end{align}
we have the strong convergence in $\mathrm{L}^{1}(\Omega;\mathrm{L}^{2}(0,\xi_N \wedge T;H^{s'-1}(\mathbb{R}^d))).$ This completes the proof for (ii).
\end{proof}

\begin{prop} \label{all.nois}
For any $s'>d/2$ and $T>0$,
\begin{itemize}
\item[1.] $\mathcal{S}^2(\mathbf{\tau}_n)\to \mathcal{S}^2(\mathbf{\tau})\,\,\text{strongly in}\,\, 
\mathrm{L}^2(\Omega;\mathrm{L}^{\infty}(0,\xi_N \wedge T;H^{s'-1}(\mathbb{R}^d)))$ as $n\to \infty$.
\item[2.] $\mathcal{S}(\mathbf{\tau}_n)\to \mathcal{S}(\mathbf{\tau})\,\,\text{strongly in}\,\, 
\mathrm{L}^2(\Omega;\mathrm{L}^{\infty}(0,\xi_N \wedge T;H^{s'-1}(\mathbb{R}^d)))$ as $n\to \infty$.
\item[3.] $\sigma(t,\v_n)\to\sigma(t,\v)\,\,\text{strongly in}\,\, \mathrm{L}^2(\Omega;\mathrm{L}^2(0,\xi_N \wedge T;\mathcal{L}_Q(\mathrm{L}^2,H^{s'-1})))$ as $n\to \infty$.
\item[4.]$G(\v_n,z)\to
G(\v,z)$ strongly in $\mathbb{H}^2_{\lambda}([0,\xi_N \wedge T]\times
Z;H^{s'-1}(\mathbb{R}^d))$, as $n\to \infty$.
\end{itemize}
\end{prop}

\begin{proof} Since for any $d/2<s'<s$ and $T>0$, $(\v_n,\mathbf{\tau}_n)\to(\v,\mathbf{\tau})$ strongly in
$\mathrm{L}^2(\Omega;\mathrm{L}^2(0,\xi_N \wedge T;H^{s'}(\mathbb{R}^d)))$, by Remark \ref{esti.S}, we infer
 \begin{align} \label{conv.noiseS2}
&\mathbb{E}\left[\sup_{0\leq t\leq \xi_N \wedge T}\|\mathcal{S}^2(\mathbf{\tau}_n)-\mathcal{S}^2(\mathbf{\tau})\|^{2}_{H^{s'-1}}\right]
=\mathbb{E}\left[\sup_{0\leq t\leq \xi_N \wedge T}\|\mathcal{S}^2(\mathbf{\tau}_n-\mathbf{\tau})\|^{2}_{H^{s'-1}}\right] \nonumber\\ &\leq\|h\|_{H^{s'-1}}^4\mathbb{E}\left[\sup_{0\leq t\leq \xi_N \wedge T}\|\mathbf{\tau}_n-\mathbf{\tau}\|^{2}_{H^{s'-1}}\right]
\to 0\text{ as }n\to\infty,
\end{align}
 and \begin{align} \label{conv.noiseS}
&\mathbb{E}\left[\sup_{0\leq t\leq \xi_N \wedge T}\|\mathcal{S}(\mathbf{\tau}_n)-\mathcal{S}(\mathbf{\tau})\|^{2}_{H^{s'-1}}\right]
\leq\|h\|^{2}_{H^{s'-1}}\mathbb{E}\left[\sup_{0\leq t\leq \xi_N \wedge T}\|\mathbf{\tau}_n-\mathbf{\tau}\|^{2}_{H^{s'-1}}\right]
\to 0\text{ as }n\to\infty.
\end{align}
Hence, convergences in $1$ and $2$ are established.\newline
Again implementing Assumption \ref{hypo} and H\"{o}lder's inequality, we have
\begin{align} \label{conv.noise1}
&\mathbb{E}\left[\int_0^{\xi_N \wedge T}\|\sigma(t,\v_n)-\sigma(t,\v)\|_{\mathcal{L}_Q(\mathrm{L}^2,H^{s'-1})}^2\d
t\right]\leq
L \,\mathbb{E}\left[\int_0^{\xi_N \wedge T} \|\v_n-\v\|_{H^{s'-1}}^2\d
t\right]\notag \\
&\leq L\, \mathbb{E}\left[(\sup_{0\leq t\leq
\xi_N \wedge T}\|\v_n-\v\|_{H^{s'}}^2) (\xi_N \wedge T)\right]
\leq L\, T \, \mathbb{E}\left[\sup_{0\leq t\leq
\xi_N \wedge T}\|\v_n-\v\|_{H^{s'}}^2 \right] \to 0\text{ as }n\to\infty
\end{align}
and
 \begin{align} \label{conv.G.1}
\mathbb{E}\left[\int_0^{\xi_N \wedge T}\int_Z\|G(\v_n,z)-G(\v,z)\|^2_{H^{s'-1}}\lambda(\d
z)\d t\right]
\leq LT\,\mathbb{E}\left[\sup_{0\leq t\leq
\xi_N \wedge T}\|\v_n-\v\|_{H^{s'}}^2 \right] \to 0\text{ as }n\to\infty.
\end{align}
Hence we have the convergences in 3 and $4.$
\end{proof}




Next we prove the main result of this section on the existence of local strong solution of the original problem \eqref{se1}-\eqref{se4}.

\begin{thm}\label{existence}
Let $\v_0,\mathbf{\tau}_0$ be $\mathscr{F}_0$-measurable and $\nabla\cdot \v_0 =0$. Let $\v_0,\mathbf{\tau}_0\in\mathrm{L}^2(\Omega;H^s(\mathbb{R}^d))$ for $s>d/2$. Then there exists a local in
time strong solution of the problem (\ref{se1})-(\ref{se4}) such that
\begin{enumerate}
\item [(i)] $\v\in\mathrm{L}^{2}(\Omega;\mathrm{L}^{\infty}(0,\rho_N \wedge T;H^s(\mathbb{R}^d))\cap\mathrm{L}^2(0,\rho_N \wedge T;H^{s+1}(\mathbb{R}^d))),
\mathbf{\tau}\in\mathrm{L}^{2}(\Omega;\mathrm{L}^{\infty}(0,\rho_N \wedge T;H^s(\mathbb{R}^d))),$
where \begin{align} \label{stop.lim} 
\rho_N=\inf_{t\geq
0}\left\{t:\mu_2\|\v(t)\|_{H^s}^2+\mu_1\|\mathbf{\tau}(t)\|_{H^s}^2
+2 \mu_2\nu \int_0^t \| \nabla\v (r)\|_{H^s}^2dr>N\right\},
\end{align} 
\item [(ii)] the $\mathscr{F}_t$-adapted paths of $(\v,\rho_N)$ and $(\mathbf{\tau},\rho_N)$ are c\`{a}dl\`{a}g and continuous respectively,
\item[(iii)] $\rho_N$ is a predictable strictly positive stopping time
	satisfying 
	\begin{align*}
	\mathbb{P}\left(\rho_N>\delta\right)
	\geq 1-2\delta e^{(\tilde{C}+C_2 \delta)} \Big(2\mathbb{E}\left(\mu_2\|\v_0\|_{H^s}^2+\mu_1\|\mathbf{\tau}_0\|_{H^s}^2\right)+18K\mu_2\delta\Big)	
	\end{align*}
	for any $\delta\in (0,1)$, and for some positive constant $\tilde{C}$ independent of $\delta$.
\end{enumerate}

\end{thm}

\begin{proof} 
Using the above Propositions, we can see from $(\ref{Trun1})-(\ref{Trun4})$ that with probability 1
\begin{align} \label{exist.lim1}
\left(J^s\v_n(t),J^s\phi_1\right)_{\mathrm{L}^2}&=\left(J^s\v_0,J^s\phi_1\right)_{\mathrm{L}^2}+\int_0^t\left(\nu\Delta
J^s\v_n-J^s[(\v_n\cdot\nabla)\v_n]+\mu_1\nabla\cdot J^s\tau_n,J^s \phi_1\right)_{\mathrm{L}^2}\d s\nonumber\\
& \quad +\int_0^t\left(J^s\sigma(s,\v_n(s))\d W_1(s),J^s\phi_1\right)_{\mathrm{L}^2}
+\int_0^t\int_Z\left(J^sG(\v_n(s-),z),J^s\phi_1\right)_{\mathrm{L}^2}\tilde{N}_1(\d s,\d z),\\
\left(J^s\tau_n(t),J^s\phi_2\right)_{\mathrm{L}^2}&=\left(J^s\tau_0,J^s\phi_2\right)_{\mathrm{L}^2}
+\mu_2 \int_0^t \left( J^s\mathcal{D}(\v_n),J^s\phi_2\right)_{\mathrm{L}^2}\d s\nonumber\\
&\quad -\int_0^t\left(J^s[(\v_n\cdot\nabla)\tau_n]+J^sQ(\tau_n,\nabla\v_n)+aJ^s\tau_n,J^s\phi_2\right)_{\mathrm{L}^2}\d s\nonumber\\
& \quad +\frac{1}{2}\int_0^t\left(J^sS^2(\tau_n),J^s\phi_2\right)_{\mathrm{L}^2}\d s+\int_0^t\left(J^sS(\tau_n(s)),
J^s\phi_2\right)_{\mathrm{L}^2}\d W_2(s),\label{exist.lim2} \\
\quad\quad \nabla\cdot\v_n&=0,\nonumber
\end{align}
are satisfied for any $t\in[0,\rho_N^n \wedge T)$ and $\phi_i \in H^s(\mathbb{R}^d)\,;\,i=1,2$ with
$\nabla\cdot \phi_1=0.$ \newline

Note, $\mathrm{L}^2(\Omega;\mathrm{L}^{\infty}(0,\xi_N \wedge T;H^{s}(\mathbb{R}^d)))$
is the dual of
$\mathrm{L}^2(\Omega;\mathrm{L}^{1}(0,\xi_N \wedge T;H^{-s}(\mathbb{R}^d)))$ and 
$\mathrm{L}^2(\Omega;\mathrm{L}^{1}(0,\xi_N \wedge T;H^{-s}(\mathbb{R}^d)))$ is separable Hilbert space
(see Remark 10.1.10 and Theorem 10.1.13 of Papageorgiou and
Kyritsi-Yiallourou \cite{PNKS}). Therefore, due to uniform boundedness of the sequences $\v_n$ and $\tau_n$ from Remark $\ref{ener.estim}$, we can apply Banach-Alaoglu Theorem (see Theorem 4.18 of
Robinson \cite{RJC}), to extract subsequences $\v_{n_k}$  and $\tau_{n_k}$ such that 
\begin{align} \label{conv.weak}
\v_{n_k}\xrightarrow{w^{*}}\v,\quad&\tau_{n_k}\xrightarrow{w^{*}}\tau\,\,\text{
in
}\,\,\mathrm{L}^{2}(\Omega;\mathrm{L}^{\infty}(0,\xi_N \wedge T;H^s(\mathbb{R}^d)))\notag\\
&\mbox{and}\quad \nabla\v_{n_k}\xrightarrow{w}\nabla\v \,\,\text{ in
}\,\,\mathrm{L}^{2}(\Omega;\mathrm{L}^{2}(0,\xi_N \wedge T;H^s(\mathbb{R}^d))).
\end{align}
This assures the limit satisfies
$$\v\in\mathrm{L}^{2}(\Omega;\mathrm{L}^{\infty}(0,\xi_N \wedge T;H^s(\mathbb{R}^d))\cap\mathrm{L}^2(0,\xi_N \wedge T;H^{s+1}(\mathbb{R}^d))),
\;\;\tau \in\mathrm{L}^{2}(\Omega;\mathrm{L}^{\infty}(0,\xi_N \wedge T;H^s(\mathbb{R}^d))).$$
Passing to the limit to \eqref{exist.lim1}-\eqref{exist.lim2} as $n\to \infty,$ it yields with probability 1
\begin{align} \label{exist.1}
\left(J^s\v(t),J^s\phi_1\right)_{\mathrm{L}^2}&=\left(J^s\v_0,J^s\phi_1\right)_{\mathrm{L}^2}+\int_0^t\left(\nu\Delta
J^s\v-J^s[(\v\cdot\nabla)\v]+\mu_1\nabla\cdot J^s\tau,J^s \phi_1\right)_{\mathrm{L}^2}\d s\nonumber\\
&\quad +\int_0^t\left(J^s\sigma(s,\v(s))\d W_1(s),J^s\phi_1\right)_{\mathrm{L}^2}
+\int_0^t\int_Z\left(J^sG(\v(s-),z),J^s\phi_1\right)_{\mathrm{L}^2}\tilde{N}_1(\d s,\d z),\\
\left(J^s\tau(t),J^s\phi_2\right)_{\mathrm{L}^2}&=\left(J^s\tau_0,J^s\phi_2\right)_{\mathrm{L}^2}
+\mu_2 \int_0^t \left( J^s\mathcal{D}(\v),J^s\phi_2\right)_{\mathrm{L}^2}\d s\nonumber\\
&\quad -\int_0^t\left(J^s[(\v\cdot\nabla)\tau]+J^sQ(\tau,\nabla\v)+aJ^s\tau,J^s\phi_2\right)_{\mathrm{L}^2}\d s\nonumber\\
& \quad +\frac{1}{2}\int_0^t\left(J^sS^2(\tau),J^s\phi_2\right)_{\mathrm{L}^2}\d s+\int_0^t\left(J^sS(\tau(s)),
J^s\phi_2\right)_{\mathrm{L}^2}\d W_2(s),\notag \\
\quad \textit{with}\quad \nabla\cdot\v&=0,\label{exist.2}
\end{align}
for any $t\in[0,\xi_N \wedge T)$.
Hence, $(\v,\tau)$ solves (\ref{exist.1})-(\ref{exist.2}) for $s>d/2.$
\par\noindent
We now define
 \begin{align} \label{stop.lim.1} 
\rho_N:=\inf_{t\geq
0}\left\{t:\mu_2\|\v(t)\|_{H^s}^2+\mu_1\|\mathbf{\tau}(t)\|_{H^s}^2
+2 \mu_2\nu \int_0^t \| \nabla\v (r)\|_{H^s}^2dr>N\right\}.
\end{align}
\noindent
\textbf{Claim:} For fixed $N \geq 1,\,\,T>0,\,\,\, \lim_{n \rightarrow \infty}\rho^n_N \wedge T=\rho_N \wedge T = \xi_N\wedge T.$ 
\par\noindent
\textit{Proof.}

 Recalling the arguments as used for \eqref{conv.weak}, we have (the subsequences still denoted by the same)
 \begin{align} \label{conv.mod1}
\v_{n}\xrightarrow{w^{*}}\v,\quad&\tau_{n}\xrightarrow{w^{*}}\tau\,\,\text{
in
}\,\,\mathrm{L}^{2}(\Omega;\mathrm{L}^{\infty}(0,\xi_N \wedge T;H^s(\mathbb{R}^d)))\notag\\
&\mbox{and}\quad \nabla\v_{n}\xrightarrow{w}\nabla\v \,\,\text{ in
}\,\,\mathrm{L}^{2}(\Omega;\mathrm{L}^{2}(0,\xi_N \wedge T;H^s(\mathbb{R}^d))).
\end{align} 
Therefore by the lower semicontinuity property of weak and weak-star convergences (see Chapter 10 of Lax \cite{PL}), and using the inequality $\liminf_{n \rightarrow \infty}(f_n+g_n) \geq \liminf_{n \rightarrow \infty} f_n + \liminf_{n \rightarrow \infty} g_n$ for bounded sequence of functions $f_n$ and $g_n$, we obtain 
\begin{align}\label{E1111}
&\liminf_{n \rightarrow \infty} \mathbb{E}\Big[ \sup_{0 \leq t \leq \xi_N \wedge T} \Big(\mu_2\|\v_n(t)\|_{H^s}^2+\mu_1\|\mathbf{\tau}_n(t)\|_{H^s}^2\Big)
+2 \mu_2\nu \int_0^{\xi_N \wedge T} \| \nabla\v_n (t)\|_{H^s}^2dt \Big] \notag\\
&\geq \mathbb{E}\Big[ \sup_{0 \leq t \leq \xi_N \wedge T} \Big(\mu_2\|\v(t)\|_{H^s}^2+\mu_1\|\mathbf{\tau}(t)\|_{H^s}^2\Big)
+2 \mu_2\nu \int_0^{\xi_N \wedge T} \| \nabla\v (t)\|_{H^s}^2dt \Big].
\end{align}
\noindent
Define 
\begin{align}
\mathcal{E}_n(t)&:=\mu_2\|\v_n(t)\|_{H^s}^2+\mu_1\|\mathbf{\tau}_n(t)\|_{H^s}^2
+2 \mu_2\nu \int_0^t \| \nabla\v_n (r)\|_{H^s}^2dr,\notag\\ \mbox{and}\quad \mathcal{E}(t)&:=\mu_2\|\v(t)\|_{H^s}^2+\mu_1\|\mathbf{\tau}(t)\|_{H^s}^2
+2 \mu_2\nu \int_0^t \| \nabla\v(r)\|_{H^s}^2dr.
\end{align}
Therefore we may consider the case when  
\begin{align}\label{xxxx}
\liminf_{n \rightarrow \infty} \mathbb{E}\Big[ \mathcal{E}_n(t)\Big] \geq \mathbb{E}\Big[ \mathcal{E}(t)\Big], \,\,\, \mbox{for Lebesgue-almost all}\,\, t \in [0,\xi_N \wedge T],
\end{align}
as this would imply \eqref{E1111} due to Fubini's theorem.
\par\noindent 
By the definition of $\rho_N$ in \eqref{stop.lim.1}, for each $\epsilon>0$, there exists a $t_0 \in [0,T]$
such that $\rho_N \leq t_0 < \rho_N+\epsilon$ with
\begin{align} \label{ener.t0}
\mathcal{E}(t_0)>N.
\end{align}
If $[\rho_N, \rho_N+\epsilon) \subset [0, \xi_N \wedge T]$, by \eqref{xxxx}, along a subsequence of $\mathbb{E}\Big[\mathcal{E}_n(t_0)\Big]$ (still denoted by the same) it converges to $\liminf_{n \rightarrow \infty} \mathbb{E}\Big[\mathcal{E}_n(t_0)\Big]$, and this yields
\begin{align} \label{conv.ener}
\lim_{n \rightarrow \infty}\mathbb{E}\Big[\mathcal{E}_{n}(t_0)\Big] \geq \mathbb{E}\Big[\mathcal{E}(t_0)\Big] >N.
\end{align}
\noindent
Hence there exists $\tilde{n} \in \mathbb{N}$ such that $\mathbb{E} \Big[ \mathcal{E}_{n}(t_0)\Big] >N, \, \forall \, n \geq \tilde{n}
$, from which we claim that $t_0 \geq \rho_N^n, \,\forall \, n \geq \tilde{n}$. If not, then there exists a natural number $n_1>\tilde{n}$ such that $t_0 < \rho_N^{n_1}$. Hence $\mathcal{E}_{n_1}(t_0)\leq N$
and thus
$\mathbb{E} \Big[ \mathcal{E}_{n_1}(t_0)\Big] \leq N 
$, a contradiction. Hence the claim is true.
\par\noindent
Thus we have 
\begin{align} \label{ineq1.1}
\rho_N^n \wedge T \leq t_0 < (\rho_N+ \epsilon)\wedge T, \quad \forall \,\, n \geq \tilde{n}.
\end{align}
Now, since $(\v,\tau)$ is a local strong solution of \eqref{se1}-\eqref{se4} and the approximate equations \eqref{Trun1}-\eqref{Trun4} have unique strong solutions, we can identify $\mathcal{J}_n \v$ as $\v_n$ and $\mathcal{J}_n \tau$ as $\tau_n.$ Since by \eqref{intro1},  for every $s>0$, $\|\v_n\|_{H^s} \leq \|\v\|_{H^s}$ and $\|\tau_n\|_{H^s} \leq \|\tau\|_{H^s}$, we have
\begin{align}
& N < \mu_2\|\v_n(t)\|_{H^s}^2+\mu_1\|\mathbf{\tau}_n(t)\|_{H^s}^2
+2 \mu_2\nu \int_0^t \| \nabla\v_n (r)\|_{H^s}^2dr \notag \\ & \leq \mu_2\|\v(t)\|_{H^s}^2+\mu_1\|\mathbf{\tau}(t)\|_{H^s}^2
+2 \mu_2\nu \int_0^t \| \nabla\v (r)\|_{H^s}^2dr, \quad \forall \, n \geq 1.
\end{align}
Hence,
\begin{align} \label{ineq1}
\rho_N \wedge T \leq \rho_N^n \wedge T, \quad \forall \,\, n \geq 1.
\end{align}
Therefore, for each $\epsilon >0$,
\begin{align} \label{ineq2}
(\rho_N-\epsilon)\wedge T \leq \rho_N^n\wedge T, \quad \forall \,\, n \geq 1.
\end{align}
Combining \eqref{ineq1.1} and \eqref{ineq2} we have for all $n \geq \tilde{n}$ and for each $\epsilon >0$,
\begin{align*}
(\rho_N-\epsilon)\wedge T \leq \rho_N^n\wedge T < (\rho_N+\epsilon)\wedge T.
\end{align*}
Taking limit as $n\to\infty$, and using the definition of $\xi_N$, we have
\begin{align*}
(\rho_N-\epsilon)\wedge T \leq \xi_N\wedge T < (\rho_N+\epsilon)\wedge T.
\end{align*}
Since $\epsilon>0$ is arbitrary, we finally infer $\lim_{n\to\infty} \rho_N^n\wedge T = \xi_N\wedge T = \rho_N\wedge T.$ This proves the claim and (i).


From Remark \ref{ener.estim}, it is assured that $(\v_n,\tau_n)$ is almost surely
uniformly convergent to
$(\v,\tau)$ on finite interval $[0, \rho_N \wedge T)$, from which it follows that $\v$ is
adapted and c\`{a}dl\`{a}g (Theorem 6.2.3, Applebaum \cite{Ap}) and $\tau$ is continuous. Hence (ii) follows.

Now using the continuity argument and using \eqref{esti.rho.n} we achieve, \begin{align}
	&\mathbb{P}\left(\rho_N>\delta\right)=\lim_{n \rightarrow \infty}\mathbb{P}\left(\rho^n_N>\delta\right) 
	\geq 1-2\delta e^{(\tilde{C}+C_2 \delta)} \Big(2\mathbb{E}\left(\mu_2\|\v_0\|_{H^s}^2+\mu_1\|\mathbf{\tau}_0\|_{H^s}^2\right)+18K\mu_2\delta\Big)	
	\end{align}
	for any $\delta\in (0,1)$, and for some positive constant $\tilde{C}$ independent of $\delta$.
\end{proof}

Next we proceed to prove uniqueness of the local strong solution of \eqref{se1}-\eqref{se4}.

\begin{thm}\label{uniqueness}
Let $\v_0,\mathbf{\tau}_0$ be $\mathscr{F}_0$-measurable and $\nabla\cdot \v_0 =0$. Let $\v_0,\mathbf{\tau}_0\in\mathrm{L}^2(\Omega;H^s(\mathbb{R}^d))$ for $s>d/2$.. Let $\v_i$ and $\tau_i$ $i=1,2$ be $\mathscr{F}_t$-adapted c\`adl\`ag and continuous processes respectively such that $(\v_i,\tau_i,\rho^{i}_N)$ are local strong solutions of (\ref{se1})-(\ref{se4})
with the same initial conditions $\v_i(0)=\v_0,\\tau_i(0)=\\tau_0$, and
$$\v_i\in\mathrm{L}^{2}(\Omega;\mathrm{L}^{\infty}(0,\rho^i_N \wedge T; H^s(\mathbb{R}^d))\cap\mathrm{L}^2(0,\rho^i_N \wedge T;H^{s+1}(\mathbb{R}^d))),
\tau_i\in\mathrm{L}^{2}(\Omega;\mathrm{L}^{\infty}(0,\rho^i_N \wedge T;H^s(\mathbb{R}^d))),$$
for $s>d/2$. Then $$ \v_1(t)=\v_2(t), \tau_1(t)=\tau_2(t)\quad \mbox{a.s.} \quad \forall\, t \in [0,\rho^1_N \wedge \rho^2_N \wedge T]$$ 
 as functions in $\mathrm{L}^2(\Omega;\mathrm{L}^{\infty}(0,T;\mathrm{L}^2(\mathbb{R}^d)))$. Moreover, $\rho^1_N= \rho^2_N \,\,\mathbb{P}-\mbox{a.s.}$
\end{thm}
 \quad

\begin{proof}
\textbf{Step I} \newline
 Let $(\v_1,\tau_1,\rho^1_N)$ and $(\v_2,\tau_2,\rho^2_N)$ be two local strong solutions of the system
of equations (\ref{se1})-(\ref{se4}) having common initial data
$\v_1(0)=\v_2(0)=\v_0$ and $\tau_1(0)=\tau_2(0)=\tau_0$ such that
$\mathbb{E}\left(\|\v_0\|_{H^s}^2\right)<\infty\textrm{ and
}\mathbb{E}\left(\|\tau_0\|_{H^s}^2\right)<\infty$.

Considering the difference between the two equations satisfied by
$(\v_1,\tau_1)$ and $(\v_2,\tau_2)$ we obtain
\begin{align}
\d(\mathbf{v}_1-\v_2)&=\nu\Delta
(\v_1-\v_2)\d t-\nabla(p_1-p_2)\d t-[(\mathbf{v}_1 \cdot\nabla)\v_1 -(\mathbf{v}_2 \cdot\nabla)\v_2] dt\notag \\
&\quad +\mu_1 \nabla \cdot (\mathbf{\tau}_1-\mathbf{\tau}_2)\d t +\left( \sigma(t,\v_1)-\sigma(t,\v_{2})\right)\d W_1(t)
\nonumber\\
&\quad+\int_Z \left( G(\v_1,z)- G(\v_{2},z)\right) \tilde{N}_1(\d t,\d z),\label{uni1}\\
\d(\mathbf{\tau}_1-\mathbf{\tau}_2)&=-\left( \left[(\mathbf{v}_1 \cdot \nabla)\mathbf{\tau}_1 \right] 
-[(\mathbf{v}_2 \cdot\nabla)\mathbf{\tau}_2] \right)\d t
-\left( \Q(\mathbf{\tau}_1, \nabla \v_1)- \Q(\mathbf{\tau}_2, \nabla \v_2)\right)\d t \notag \\
&\quad+\mu_2 \mathcal{D}(\v_1-\v_2)dt -a (\mathbf{\tau}_1-\mathbf{\tau}_2)\d t
+ \frac{1}{2}\left[\mathcal{S}^2(\mathbf{\tau}_1)- \mathcal{S}^2(\mathbf{\tau}_2)\right]\d t\notag \\
&\quad+ \left(\mathcal{S}(\mathbf{\tau}_1)- \mathcal{S}(\mathbf{\tau}_2)\right)\d W_2(t).\label{uni2}
\end{align}
Let us apply It\^{o}'s Lemma to the function
$\|x\|^2_{\mathrm{L}^2}$ and to the process $\mu_2(\v_1-\v_2)$
in (\ref{uni1}) and to the process $\mu_1(\mathbf{\tau}_1-\mathbf{\tau}_2)$ in (\ref{uni2}),
and  adding these two equations and further exploiting the fact that $2\mu_1 \mu_2 \left\lbrace (\nabla \cdot (\mathbf{\tau}_1-\mathbf{\tau}_2), \v_1-\v_2)_{\mathrm{L}^2}
+(\mathcal{D}(\v_1-\v_2),\mathbf{\tau}_1-\mathbf{\tau}_2)_{\mathrm{L}^2}\right\rbrace -2\mu_2\left(\nabla (p_1-p_2),\v_1-\v_2\right)_{\mathrm{L}^2}=0$ as $\nabla \cdot \v_1=\nabla \cdot \v_2=0,$ we achieve 

\begin{align} \label{uni1.1}
&\d\left(\mu_2\|\v_1-\v_2\|_{\mathrm{L}^2}^2+\mu_1 \|\tau_1-\tau_2\|_{\mathrm{L}^2}^2\right)
+2\mu_2 \nu\|\nabla(\v_1-\v_2)\|^2_{\mathrm{L}^2}\d
t\nonumber\\
&=\underbrace{-2 \mu_2 \left([(\v_1 \cdot\nabla)\v_1 ]- [(\v_2 \cdot\nabla)\v_2 ]
,\v_1-\v_2 \right)_{\mathrm{L}^2}}_{I_5}\d t\notag \\
&\quad\underbrace{-2 \mu_1\left( [(\v_1 \cdot\nabla)\tau_1 ]- [(\v_2 \cdot\nabla)\tau_2 ]
,\tau_1-\tau_2 \right)_{\mathrm{L}^2}}_{I_6}\d t\notag \\
&\quad\underbrace{-2\mu_1\left( \Q(\tau_1, \nabla \v_1)-
\Q(\tau_2, \nabla \v_2), \tau_1-\tau_2\right)_{\mathrm{L}^2}}_{I_7} \d t-2a\mu_1 \|\tau_1-\tau_2\|_{\mathrm{L}^2}^2\d t
\nonumber\\
&\quad+ \underbrace{\mu_1\left(\mathcal{S}^2(\tau_1)-\mathcal{S}^2(\tau_2),\tau_1-\tau_2\right)_{\mathrm{L}^2}
}_{I_8}\d t+ \underbrace{\mu_1\|\mathcal{S}(\tau_1)-\mathcal{S}(\tau_2)\|_{\mathrm{L}^2}^2}_{I_9} \d t\notag \\
&\quad+\mu_2\|\sigma(t,\v_1)-\sigma(t,\v_{m})\|_{\mathcal{L}_Q(\mathrm{L}^2,\mathrm{L}^2)}^2\d t
\nonumber\\
&\quad+\mu_2\int_Z\| G(\v_1(t-),z)-G
(\v_2(t-),z)\|^2_{\mathrm{L}^2}N_1(\d t,\d z)\nonumber\\
&\quad+2\mu_2\left(\left( \sigma(t,\v_1)-\sigma(t,\v_{m})\right)\d W_1(t),\v_1-\v_2 \right)_{\mathrm{L}^2}\notag\\
&\quad +2 \mu_1 \left(\mathcal{S}(\tau_1)-\mathcal{S}(\tau_2),\tau_1-\tau_2\right)_{\mathrm{L}^2}\d W_2(t)\nonumber\\
&\quad+2\mu_2\int_Z\left( G(\v_1(t-),z)- G(\v_2(t-),z),\v_1-\v_2 \right)_{\mathrm{L}^2}\tilde{N}_1(\d
t,\d z).
\end{align}
Using Lemmas (\ref{v},\ref{tau},\ref{Q}), \eqref{S.2} in Lemma \ref{M2} and \eqref{propS.2} for $\v_i,\tau_i;i=1,2,$ estimating some of these integrals separately we have
\begin{align} 
&|I_5| \leq 2 C \mu_2 \|\v_1-\v_2\|_{\mathrm{L}^2}^2 \|\nabla \v_1\|_{H^s},\label{I_5}\\
&|I_6| \leq \frac{\nu \mu_2}{2}
\|\v_1-\v_2\|_{H^1}^2+\frac{2C \mu_1^2}{\nu \mu_2} \|\tau_1\|_{H^s}^2 \|\tau_1-\tau_2\|_{\mathrm{L}^2}^2, \label{I_6}\\
& |I_7| \leq \frac{\nu \mu_2}{2}
\|\v_1-\v_2\|_{H^1}^2+\frac{2C \mu_1^2}{\nu \mu_2} \|\tau_2\|_{H^s}^2 \|\tau_2-\tau_2\|_{\mathrm{L}^2}^2
+2\mu_1C \|\nabla \v_1\|_{H^s}\|\tau_1-\tau_2\|_{\mathrm{L}^2}^2,\label{I_7}\\
&|I_8| \leq \mu_1 \|h\|_{\mathrm{L}^\infty}^2 \|\tau_1-\tau_2\|_{\mathrm{L}^2} \label{I_8}\\ 
&|I_9| \leq 2\mu_1 \|h\|_{\mathrm{L}^\infty}^2 \|\tau_1-\tau_2\|_{\mathrm{L}^2}^2.\label{I_9}
\end{align}
Combining \eqref{I_5}-\eqref{I_9} and implementing the fact that $ 4C\|\nabla \v_1\|_{H^s} \leq C^2+ 2\|\nabla \v_1\|_{H^s}^2,$ \eqref{uni1.1} reduces to:
\begin{align} \label{esti.u}
&\d\left(\mu_2\|\v_1-\v_2\|_{\mathrm{L}^2}^2+\mu_1 \|\mathbf{\tau}_1-\mathbf{\tau}_2\|_{\mathrm{L}^2}^2\right)
+\mu_2 \nu\|\nabla(\v_1-\v_2)\|^2_{\mathrm{L}^2}\d t\nonumber\\
&\leq \left((2C^2+2a)+2\|\nabla \v_1\|_{H^s}^2
+\frac{2C\mu_1}{\nu\mu_2}\Big(\|\mathbf{\tau}_1\|_{H^{s}}^2 +\|\mathbf{\tau}_2\|_{H^{s}}^2\Big)
+2\|h\|_{\mathrm{L}^\infty}^2\right)\nonumber\\
&\qquad \qquad \times\Big(\mu_2\|\v_1-\v_2\|_{\mathrm{L}^2}^2+\mu_1\|\tau_1-\tau_2\|_{\mathrm{L}^2}^2 \Big) \notag \\
&\quad+\mu_2\|\sigma(t,\v_1)-\sigma(t,\v_2)\|_{\mathcal{L}_Q(\mathrm{L}^2,\mathrm{L}^2)}^2\d t
\nonumber\\
&\quad+\mu_2\int_Z\|G(\v_1(t-),z)-G(\v_2(t-),z)\|^2_{\mathrm{L}^2}N_1(\d t,\d z)\nonumber\\
&\quad+2\mu_2\left(\left(\sigma(t,\v_1)-\sigma(t,\v_2)\right)\d W_1(t),\v_1-\v_2 \right)_{\mathrm{L}^2}\notag\\
&\quad +2 \mu_1 \left(\mathcal{S}(\mathbf{\tau}_1)- \mathcal{S}(\tau_2),
\tau_1-\tau_2\right)_{\mathrm{L}^2}\d W_2(t)\nonumber\\
&\quad+2\mu_2\int_Z\left(G(\v_1(t-),z)-G(\v_2(t-),z),\v_1-\v_2 \right)_{\mathrm{L}^2}\tilde{N}_1(\d
t,\d z).
\end{align}

\textbf{Step II:} \newline
For the stopping time $\rho^1_N$, let us take the process
$\eta(t)=\exp\left(-2\mathlarger{\int_{0}^t}\|\nabla\v_1 \|^2_{H^s}
\d s\right),$ $t\in[0,T \wedge \rho^1_N)$ and apply It\^{o} product formula
(see Theorem 4.4.13, Applebaum \cite{Ap}) to the process
$\eta(t)(\mu_2\|\v_1-\v_2\|_{\mathrm{L}^2}^2+\mu_1\|\tau_1-\tau_2\|^2_{\mathrm{L}^2})$ in the
interval $[0,t]$ to get
\begin{align} \label{ito.pro.u}
&\eta(t)\left(\mu_2\|\v_1-\v_2\|_{\mathrm{L}^2}^2+\mu_1 \|\mathbf{\tau}_1-\mathbf{\tau}_2\|_{\mathrm{L}^2}^2\right)
+\mu_2 \nu\int_0^t\eta(s)\|\nabla(\v_1-\v_2)\|^2_{\mathrm{L}^2}\d s\nonumber\\
&\leq\int_{0}^{t} \eta(s) \left((2C^2+2a)+\frac{2C\mu_1}{\nu\mu_2}\left(\|\mathbf{\tau}_1\|_{H^{s}}^2
+\|\tau_2\|_{H^s}^2\right)+2\|h\|_{\mathrm{L}^\infty}^2\right)\nonumber\\
&\qquad \qquad \times \Big(\mu_2\|\v_1-\v_2\|_{\mathrm{L}^2}^2+\mu_1\|\mathbf{\tau}_1-\mathbf{\tau}_2\|_{\mathrm{L}^2}^2 \Big) \d s
\notag \\
&\quad+\mu_2 \int_{0}^{t} \eta(s) \|\sigma(s,\v_1)-\sigma(s,\v_2\|_{\mathcal{L}_Q(\mathrm{L}^2,\mathrm{L}^2)}^2
 \d s\nonumber\\
&\quad+\mu_2 \int_{0}^{t} \eta(s) \int_Z\|G(\v_1(s-),z)-G(\v_2(s-),z)\|^2_{\mathrm{L}^2}N_1(\d s,\d z)\nonumber\\
&\quad+2\mu_2\int_{0}^{t} \eta(s)\left(\left(\sigma(s,\v_1)-\sigma(s,\v_2)\right)\d W_1(s),\v_1-\v_2 \right)_{\mathrm{L}^2}\notag\\
&\quad +2 \mu_1 \int_{0}^{t} \eta(s)\left(\mathcal{S}
(\mathbf{\tau}_1)-\mathcal{S}(\mathbf{\tau}_2),\mathbf{\tau}_1-\mathbf{\tau}_2\right)_{\mathrm{L}^2}\d W_2(s)\nonumber\\
&\quad+2\mu_2\int_{0}^{t} \eta(s)\int_Z\left(G(\v_1(s-),z)-G(\v_2(s-),z),\v_1-\v_2 \right)_{\mathrm{L}^2}\tilde{N}_1(\d
s,\d z).
\end{align}
It is to be noted that the quadratic variation of the product of these two adapted processes
$Z_1(t)=\eta(t)$ and $Z_2(t)=\mu_2\|\v_1-\v_2\|_{\mathrm{L}^2}^2
+\mu_1\|\mathbf{\tau}_1-\mathbf{\tau}_2\|^2_{\mathrm{L}^2},$ i.e., $[Z_1,Z_2](t)$ is zero (see, Section 4.4.3, page-257 of Applebaum \cite{Ap}).
\newline
 Let us now take the supremum from $0$ to
$T\wedge\rho^1_N$, for any $T>0$ in  \eqref{ito.pro.u} and then on taking expectation and thereafter using \eqref{mt1} (in Remark \ref{nlam}),  we get
\begin{align} \label{Cau.1.u}
&\mathbb{E}\Big[\sup_{0\leq t\leq T\wedge \rho^1_N}\eta(t)\Big(\mu_2\|\v_1-\v_2\|_{\mathrm{L}^2}^2+\mu_1 \|\mathbf{\tau}_1-\mathbf{\tau}_2\|_{\mathrm{L}^2}^2\Big)\Big]
+\mu_2 \nu\mathbb{E}\left[\int_0^{T \wedge \rho^1_N} \eta(s)\|\nabla(\v_1-\v_2)\|^2_{\mathrm{L}^2}\d s\right]\nonumber\\
&\leq \mathbb{E}\Big[\int_0^{T \wedge \rho^1_N} \eta(s)\left((2C^2+2a)
+\frac{2C\mu_1}{\nu\mu_2}\left(\|\mathbf{\tau}_1\|_{H^{s}}^2+\|\tau_2\|_{H^s}^2\right)+2\|h\|_{\mathrm{L}^\infty}^2\right)\nonumber\\
&\qquad \qquad\times \Big(\mu_2\|\v_1-\v_2\|_{\mathrm{L}^2}^2+\mu_1\|\mathbf{\tau}_1-\mathbf{\tau}_2\|_{\mathrm{L}^2}^2 \Big) \d s \Big]\notag\\
&\quad +\mu_2\mathbb{E}\Big[\int_0^{T \wedge \rho^1_N}\eta(s)\| \sigma(s,\v_1)-\sigma(s,\v_2)
\|_{\mathcal{L}_Q(\mathrm{L}^2,\mathrm{L}^2)}^2\d s\Big]
\nonumber\\
&\quad+\mu_2\mathbb{E}\Big[\int_0^{T \wedge \rho^1_N}\int_Z\eta(s)\| G(\v_1(s),z)-G
(\v_2(s),z)\|^2_{\mathrm{L}^2}\lambda(\d z) \d s\Big]\nonumber\\
&\quad+\underbrace{2\mu_2\mathbb{E}\Big[\sup_{0\leq t\leq T \wedge \rho^1_N}\int_0^{t}\eta(s)\left(\left(\mathcal{J}_1 \sigma(s,\v_1)-\mathcal{J}_2
\sigma(s,\v_2)\right)\d W_1(s),\v_1-\v_2 \right)_{\mathrm{L}^2}\Big]}_{I_{10}}\notag\\
&\quad +\underbrace{2 \mu_1 \mathbb{E}\Big[\sup_{0\leq t\leq T \wedge \rho^1_N}\int_0^{t}\eta(s)\left(S(\mathbf{\tau}_1)-S(\mathbf{\tau}_2),
\mathbf{\tau}_1-\mathbf{\tau}_2\right)_{\mathrm{L}^2}\d W_2(s)\Big]}_{I_{11}}\nonumber\\
&\quad+\underbrace{2\mu_2\mathbb{E}\Big[\sup_{0\leq t\leq T \wedge \rho^1_N}\int_0^{t}\int_Z\eta(s)\left(G(\v_1(s-),z)-
G(\v_2(s-),z),\v_1-\v_2 \right)_{\mathrm{L}^2}\tilde{N}_1(\d s,\d z)\Big]}_{I_{12}}.
\end{align}
We note that direct application of Burkholder-Davis-Gundy inequality, Young's inequality and  Assumption \ref{hypo} produces the following estimates:
\begin{align} \label{M1.cau.u}
|I_{10}| \leq \frac{\mu_2}{4}\mathbb{E}\left(\sup_{0\leq t\leq T \wedge \rho^1_N}\eta(t)\|\v_1(t)-\v_2(t)\|^2_{\mathrm{L}^2}\right) 
 +8\mu_2L\, \mathbb{E}\left[\int_{0}^{T \wedge \rho^1_N}\eta(t)\|\v_1-\v_{2}\|_{\mathrm{L}^2}^2 \d
t\right],
\end{align} 
\begin{align}\label{M2.cau.u}
|I_{11}| &\leq \frac{\mu_1}{2}\mathbb{E}\left(\sup_{0\leq t\leq T \wedge \rho^1_N}\eta(t)
\|\mathbf{\tau}_1(t)-\mathbf{\tau}_2(t)\|^2_{\mathrm{L}^2}\right)
  +8\mu_1\|h\|^2_{\mathrm{L}^\infty} \mathbb{E}\left(\int_0^{T \wedge \rho^1_N}\eta(t)
  \|\mathbf{\tau}_1-\mathbf{\tau}_2\|_{\mathrm{L}^2}^2\d t \right),
\end{align}
and
\begin{align}\label{M3.cau.u}
|I_{12}| &\leq \frac{\mu_2}{4}\mathbb{E}\left(\sup_{0\leq t\leq T \wedge \rho^1_N}\eta(t)
\|\v_1(t)-\v_2(t)\|^2_{\mathrm{L}^2}\right)+8\mu_2L\, \mathbb{E}\left[\int_{0}^{T \wedge \rho^1_N}\eta(t)\|\v_1-\v_{2}\|_{\mathrm{L}^2}^2 \d
t\right].
\end{align}
Combining \eqref{M1.cau.u}-\eqref{M3.cau.u} and Assumption \ref{hypo}, equation \eqref{Cau.1.u} is reduced to
\begin{align} \label{Cau.4.u}
&\frac{1}{2}\mathbb{E}\Big[\sup_{0\leq t\leq T \wedge \rho^1_N}\eta(t)
\Big(\mu_2\|\v_1-\v_2\|_{\mathrm{L}^2}^2+\mu_1 \|\mathbf{\tau}_1-\mathbf{\tau}_2\|_{\mathrm{L}^2}^2\Big)\Big]
+\mu_2 \nu\mathbb{E}\left[\int_0^{T \wedge \rho^1_N} \eta(s) \|\nabla(\v_1-\v_2)\|^2_{\mathrm{L}^2} \d s \right]\nonumber\\
&\leq \mathbb{E}\Big[\int_0^{T \wedge \rho^1_N}\eta(s)\left((2C^2+2a+18L)
+\frac{2C\mu_1}{\nu\mu_2}\left(\|\mathbf{\tau}_1\|_{H^{s}}^2
+\|\mathbf{\tau}_2\|_{H^{s}}^2\right)+10\|h\|_{\mathrm{L}^\infty}^2\right)\nonumber\\
&\qquad\times \Big(\mu_2\|\v_1-\v_2\|_{\mathrm{L}^2}^2
+\mu_1\|\mathbf{\tau}_1-\mathbf{\tau}_2\|_{\mathrm{L}^2}^2 \Big) \d s \Big].
\end{align}
Using the definition of stopping time we have
\begin{align} \label{Cau.5.u}
&\mathbb{E}\Big[\sup_{0\leq t\leq T \wedge \rho^1_N}\eta(t)
\Big(\mu_2\|\v_1-\v_2\|_{\mathrm{L}^2}^2+\mu_1 \|\mathbf{\tau}_1-\mathbf{\tau}_2\|_{\mathrm{L}^2}^2\Big)\Big]
+2\mu_2 \nu\mathbb{E}\left[\int_0^{T \wedge \rho^1_N} \eta(s) \|\nabla(\v_1-\v_2)\|^2_{\mathrm{L}^2} \d s \right]\nonumber\\
&\leq 4\left((C^2+a+9L)+\frac{2CN\mu_1}{\nu\mu_2}+5\|h\|_{\mathrm{L}^\infty}^2\right)\mathbb{E}\Big[\int_0^{T \wedge \rho^1_N}
\sup_{0\leq s\leq t}\eta(s)
\Big(\mu_2\|\v_1-\v_2\|_{\mathrm{L}^2}^2+\mu_1\|\mathbf{\tau}_1-\mathbf{\tau}_2\|_{\mathrm{L}^2}^2 \Big) \d t \Big].
\end{align}

\textbf{Step III:} \newline
Define $A_N=\left\{(\v,\tau):\mu_2\|\v(t)\|_{H^s}^2+\mu_1\|\mathbf{\tau}(t)\|_{H^s}^2
+2 \mu_2\nu \int_0^t \| \nabla\v (r)\|_{H^s}^2dr \leq N \right\}.$
For $i=1,2,$ let us define a process $\alpha^i=\{\alpha^i(t)\},\,t \geq 0,$ by
\begin{align}
\alpha^i(t,\omega):=\left\{\begin{aligned}
                  &1, \quad \mbox{if} \quad (\v_i,\tau_i)(s, \omega) \in A_N,\, \forall s\in [0, t)\\
                   &0,\quad \mbox{otherwise}.
                   \end{aligned}
   \right.
\end{align}
Note that for all $t\geq 0$, $\alpha^i(t) (\v_1(0) - \v_2(0)) = 0 ~ \mbox{a.s., \ and} ~ \alpha^i(t) (\tau_1(0) - \tau_2(0)) = 0 ~ \mbox{a.s.}$ for $i=1, 2$. 
\par\noindent
For any fixed $t\geq 0$, and for each $i=1,2$, $(\alpha^i)^{-1}(\{1\}) = \{\omega : (\v_i, \tau_i)(s, \omega)\in A_N, ~\forall s \in [0,t) \} = \bigcap_{s \in [0,t)}\{\omega : (\v_i, \tau_i)(s, \omega)\in A_N \} \in \mathscr{F}$ as $\{\omega : (\v_i, \tau_i)(s, \omega)\in A_N  \} \in \mathscr{F}$. Similarly we can conclude  $(\alpha^i)^{-1}(\{0\}) \in \mathscr{F}$ for $i=1,2$. Hence for each $i=1,2$, $\alpha^i(t)$ is a random variable for each $t \geq 0$. Hence for each $i=1,2$ $\alpha^i$ is a well-defined stochastic process.
\par\noindent
Hence from \eqref{Cau.5.u} we have
\begin{align} \label{Cau.new}
&\mathbb{E}\Big[\sup_{0\leq t\leq T \wedge \rho^1_N} \alpha^1(t) \eta(t)
\Big(\mu_2\|\v_1-\v_2\|_{\mathrm{L}^2}^2+\mu_1 \|\mathbf{\tau}_1-\mathbf{\tau}_2\|_{\mathrm{L}^2}^2\Big)\Big]\nonumber\\
&\leq 4\left((C^2+a+9L)+\frac{2CN\mu_1}{\nu\mu_2}+5\|h\|_{\mathrm{L}^\infty}^2\right)\mathbb{E}\Big[\int_0^{T \wedge \rho^1_N}
\sup_{0\leq s\leq t}\alpha^1(s) \eta(s)
\Big(\mu_2\|\v_1-\v_2\|_{\mathrm{L}^2}^2+\mu_1\|\mathbf{\tau}_1-\mathbf{\tau}_2\|_{\mathrm{L}^2}^2 \Big) \d t \Big]\nonumber\\
&\leq 4\left((C^2+a+9L)+\frac{2CN\mu_1}{\nu\mu_2}+5\|h\|_{\mathrm{L}^\infty}^2\right)\mathbb{E}\Big[\int_0^{T}
\sup_{0\leq s\leq t \wedge \rho^1_N} \alpha^1(s) \eta(s)
\Big(\mu_2\|\v_1-\v_2\|_{\mathrm{L}^2}^2+\mu_1\|\mathbf{\tau}_1-\mathbf{\tau}_2\|_{\mathrm{L}^2}^2 \Big) \d t \Big].
\end{align}
\noindent
Applying Gronwall's inequality we obtain
\begin{align} \label{Cau65.u}
&\mathbb{E}\Big[\sup_{0\leq t\leq\L} \alpha^1(t) \eta(t)
\Big(\mu_2\|\v_1-\v_2\|_{\mathrm{L}^2}^2+\mu_1 \|\mathbf{\tau}_1-\mathbf{\tau}_2\|_{\mathrm{L}^2}^2\Big)\Big] \leq 0.
\end{align}
Hence we infer that 
\begin{align} \label{cau.alpha1}
\v_1(t)=\v_2(t) \quad \mbox{and} \quad \tau_1(t)=\tau_2(t)\quad \mbox{a.s. for}\quad 0 \leq t \leq T \wedge \rho^1_N.
\end{align}
Thus we conclude $\mathbb{P}(\omega \in \Omega: \rho^1_N(\omega) \leq \rho^2_N(\omega))=1$, because for $t<\rho^1_N$, $(\v_1, \tau_1)(t, \omega)\in A_N$ and hence $(\v_2, \tau_2)(t, \omega)\in A_N$.
\par\noindent
Proceeding in the similar way, replacing $\rho_N^1$ and $\alpha^1$ in \eqref{Cau.new} and \eqref{Cau65.u} by $\rho_N^2$ and $\alpha^2$ respectively, we have
\begin{align} \label{cauchy.alpha1.1}
\v_1(t)=\v_2(t) \quad \mbox{and} \quad \tau_1(t)=\tau_2(t)\quad \mbox{a.s. for}\quad 0 \leq t \leq T \wedge \rho^2_N
\end{align}
 and this implies
$\mathbb{P}(\omega \in \Omega: \rho^2_N(\omega) \leq \rho^1_N(\omega))=1.$
\par\noindent
Finally combining \eqref{cau.alpha1} and \eqref{cauchy.alpha1.1} we have,
$$\v_1(t)=\v_2(t) \quad \mbox{and} \quad \tau_1(t)=\tau_2(t)\quad \mbox{a.s. for}\quad 0 \leq t \leq T \wedge \rho^1_N \wedge \rho^2_N$$ and $\mathbb{P}(\omega \in \Omega: \rho^1_N(\omega)= \rho^2_N(\omega))=1.$
This provides uniqueness of the local strong solution of the system \eqref{se1}-\eqref{se4}
in
$\mathrm{L}^2(\Omega;\mathrm{L}^{\infty}(0,T;\mathrm{L}^2(\mathbb{R}^d)))$.
\end{proof}

\begin{thm}\label{remunique}
Let the initial data $\v_0,\tau_0\in\mathrm{L}^2(\Omega;H^s(\mathbb{R}^d))$  with $\nabla\cdot \v_0 =0$ be
$\mathscr{F}_0$-measurable for $s>d/2$. Let $\v_j$ and $\tau_j$
;\,$j=1,2$ be two $\mathscr{F}_t$-adapted processes with c\`{a}dl\`{a}g and continuous
paths respectively such that $(\v_j, \tau_j, \rho^{j}_N), j=1,2$ are local strong solutions of (\ref{se1})-(\ref{se4})
having same initial conditions $\v_1(0)=\v_2(0)=\v_0,\tau_1(0)=\tau_2(0)=\tau_0$, and for $s>d/2$
\begin{align} \label{esti.vt}
&\v_j\in\mathrm{L}^{2}(\Omega;\mathrm{L}^{\infty}(0,\rho^j_N \wedge T;H^s(\mathbb{R}^d))\cap\mathrm{L}^2(0,\rho^j_N \wedge T;H^{s+1}(\mathbb{R}^d))),\\
&\tau_j\in\mathrm{L}^{2}(\Omega;\mathrm{L}^{\infty}(0,\rho^j_N \wedge T;H^s(\mathbb{R}^d))),\label{esti.vt1}.
\end{align}
Then $$ \v_1(t)=\v_2(t), \tau_1(t)=\tau_2(t)\quad \mbox{a.s.} \quad \forall\, t \in [0,\rho^1_N \wedge \rho^2_N \wedge T],$$ 
as functions in $\left(\mathrm{L}^2(\Omega;\mathrm{L}^{\infty}(0,T;H^{s'}(\mathbb{R}^d))\cap\mathrm{L}^2(0,T;H^{s'+1}(\mathbb{R}^d))),
\mathrm{L}^2(\Omega;\mathrm{L}^{\infty}(0,T;H^{s'}(\mathbb{R}^d)))\right)$,
for any $0<s'<s$. Moreover, $\rho^1_N= \rho^2_N \,\,\mathbb{P}-\mbox{a.s.}$
\end{thm}

\begin{proof}
Now by using the Sobolev interpolation Lemma \autoref{iss} and H\"{o}lder's inequality
for $0<s'<s$, 
\begin{align} \label{esti.rho1.2}
&\mathbb{E}\left[\sup_{0\leq
t\leq\rho^1_N \wedge \rho^2_N \wedge T}\|\v_1-\v_2\|_{H^{s'}}^2\right] \nonumber\\&\leq
C \left\{\mathbb{E}\left[\sup_{0\leq
t\leq\rho^1_N \wedge \rho^2_N \wedge T}\|\v_1-\v_2\|_{\mathrm{L}^2}^2\right]\right\}^{1-s'/s}
\left\{\mathbb{E}\left[\sup_{0\leq
t\leq\rho^1_N \wedge \rho^2_N \wedge T}\|\v_1-\v_2\|_{H^s}^2\right]\right\}^{s'/s}
\nonumber\\&\leq C\left\{\mathbb{E}\left[\sup_{0\leq
t\leq\rho^1_N \wedge \rho^2_N \wedge T}\|\v_1-\v_2\|_{\mathrm{L}^2}^2\right]\right\}^{1-s'/s}
\left\{\mathbb{E}\left[\sup_{0\leq
t\leq\rho^1_N \wedge \rho^2_N \wedge T}\|\v_1\|_{H^s}^2+\sup_{0\leq
t\leq\rho^1_N \wedge \rho^2_N \wedge T}\|\v_2\|_{H^s}^2\right]\right\}^{s'/s}\nonumber\\
&\leq (2N)^{s'/s}C \left\{\mathbb{E}\left[\sup_{0\leq
t\leq \rho^1_N \wedge \rho^2_N \wedge T}\|\v_1-\v_2\|_{\mathrm{L}^2}^2\right]\right\}^{1-s'/s}.
\end{align}
Since by Theorem \ref{uniqueness}, $$ \v_1(t)=\v_2(t), \tau_1(t)=\tau_2(t)\quad \mbox{a.s.} \quad \forall\, t \in [0,\rho^1_N \wedge \rho^2_N  \wedge T] \quad \mbox{and} \quad \rho^1_N= \rho^2_N \,\,\mathbb{P}-\mbox{a.s.}$$ 
 as functions in $\mathrm{L}^2(\Omega;\mathrm{L}^{\infty}(0,\rho^1_N \wedge T;\mathrm{L}^2(\mathbb{R}^d)))$, we infer from 
\eqref{esti.rho1.2} that $$\mathbb{E}\left[\sup_{0\leq
t\leq\rho^1_N \wedge \rho^2_N \wedge T}\|\v_1-\v_2\|_{H^{s'}}^2\right] =0.$$
Hence we achieve  $\v_1(\cdot)=\v_2(\cdot)$ in
$\mathrm{L}^2(\Omega;\mathrm{L}^{\infty}(0,\rho^1_N \wedge \rho^2_N \wedge T;H^{s'}(\mathbb{R}^d)))$
for any $0<s'<s$. 

By using (\ref{esti.vt}) and Sobolev interpolation, we
also get $\v_1(\cdot)=\v_2(\cdot)$ in
$\mathrm{L}^2(\Omega;\mathrm{L}^{2}(0,\rho^1_N \wedge \rho^2_N \wedge T;H^{s'+1}(\mathbb{R}^d))).$
Similar calculation reveals that
$\tau_1(\cdot)=\tau_2(\cdot)$ in
$\mathrm{L}^2(\Omega;\mathrm{L}^{\infty}(0,\rho^1_N \wedge \rho^2_N \wedge T;H^{s'}(\mathbb{R}^d)))$
for any $0<s'<s$. This implies $\mathbb{P}(\omega \in \Omega: \rho^1_N(\omega)= \rho^2_N(\omega))=1,$ i.e., $\rho^1_N= \rho^2_N \,\,\mathbb{P}-\mbox{a.s.}$

\end{proof}

\section{Existence and Uniqueness of Local Maximal Solutions}
\begin{thm}\label{max}
Under Assumption \autoref{hypo}, Theorem \autoref{existence} and Theorem \autoref{uniqueness}, there exists a unique triplet
$(\v,\tau,\rho_{\infty})$, which is a maximal strong solution of
(\ref{se1})-(\ref{se4}) such that
\begin{align}
\sup_{0\leq s\leq
\rho_{\infty}}\|\v(s)\|^2_{H^s}+\sup_{0\leq s\leq
\rho_{\infty}}\|\tau(s)\|^2_{H^s}+\int_0^{\rho_{\infty}}\|\nabla\v(s)\|^2_{H^s}\d
s=\infty,\,\,\mathbb{P}-\text{ a. s. on
}\{\omega:\rho_{\infty}(\omega)<\infty\}.
\end{align}
\end{thm}

\begin{proof}
Let us construct a sequence of stopping times $\{\rho_{k},k \in \mathbb{N}\}$ as follows:
\begin{align}
\rho_k=\inf_{t\geq 0}\left\{t:\mu_2\|\v_k(t)\|_{H^s}^2+\mu_1\|\tau_k(t)\|_{H^s}^2+2\mu_2\nu\int_0^t\|\nabla\v_k(s)\|_{H^s}^2\d
s>k\right\}.
\end{align} 
For $n>k$, let us define a sequence of stopping times $\rho^n_{k}$
such that
\begin{align}
\rho^n_{k}=\inf_{t\geq 0}\left\{t:\mu_2\|\v_n(t)\|_{H^s}^2+\mu_1\|\tau_n(t)\|_{H^s}^2+2\mu_2\nu\int_0^t\|\nabla\v_n(s)\|_{H^s}^2\d
s>k\right\},\,\,\,k, n \in\mathbb{N}. 
\end{align}
It is evident from the definition of $\rho_{n}$ that $\rho^n_{k}\leq \rho_{n}$ a.s. for $n>k$.
Hence $(\v_n,\tau_n,\rho^n_{k})$ is a local strong solution to (\ref{Trun1})-(\ref{Trun4}) and $(\v_k,\tau_k,\rho_{k})$ is also a local
strong solution to (\ref{Trun1})-(\ref{Trun4}). Hence by the uniqueness (see Theorems 
\autoref{uniqueness},\autoref{remunique}), we conclude that $(\v_k(t),\tau_k(t))=(\v_n(t),\tau_n(t))$ a.s. for all
$t\in[0,\rho_{k} \wedge \rho^n_{k} \wedge T)$. 
This proves that
$(\v_k(t),\tau_k(t))=(\v_n(t),\tau_n(t))$ a.s. for all $t\in[0,\rho_{k} \wedge T)$
and hence $\rho_{k}<\rho_{n}$ a.s. for all $k<n$. Thus
$\{\rho_{k}:k\in\mathbb{N}\}$ is an increasing sequence and has a limit $\rho_{\infty}:=\mathlarger{\lim_{k\to\infty}}\rho_{k}$ a.s.
By letting $k\to\infty$, let $\{(\v(t),\tau(t)),0\leq t<\rho_{\infty}\}$ be the stochastic processes defined by
\begin{align}
\v(t)=\v_k(t),\tau(t)=\tau_k(t),\;\;t\in[\rho_{k-1}\wedge T,\rho_{k}\wedge T),\;\;k\geq 1,
\end{align}
where $\rho_{0}=0$. Hence, $(\v,\tau,\rho_{k})$ is local strong solution to (\ref{se1})-(\ref{se4}).
We now have a triplet $(\v,\tau,\rho_{\infty})$ such that $(\v,\tau,\rho_{\infty})$ is local strong solution to (\ref{se1})-(\ref{se4}). On the set
$\{\omega:\rho_{\infty}(\omega)<T\}$ we have
\begin{align}
&\lim_{t\uparrow\rho_{\infty}}\left[\sup_{0\leq s\leq
t}\mu_2\|\v(s)\|^2_{H^s}+\sup_{0\leq s\leq
t}\mu_1\|\tau(s)\|^2_{H^s}+2\mu_2 \nu\int_0^{t}\|\nabla\v(s)\|^2_{H^s}\d
s\right]\nonumber\\&\quad\geq  \lim_{k\uparrow\infty}
\left[\sup_{0\leq s\leq \rho_{k}\wedge T}\mu_2\|\v(s)\|^2_{H^s}+\sup_{0\leq s\leq
\rho_{k}\wedge T}\mu_1\|\tau(s)\|^2_{H^s}+2\mu_2 \nu\int_0^{t}\|\nabla\v(s)\|^2_{H^s}\d
s\right]&\nonumber\\&\quad\geq \lim_{k\uparrow\infty}
\left[\sup_{0\leq s\leq \rho_{k}\wedge T}\mu_2\|\v_k(s)\|^2_{H^s}+\sup_{0\leq s\leq
\rho_{k}\wedge T}\mu_1\|\tau_k(s)\|^2_{H^s}+2\mu_2 \nu\int_0^{t}\|\nabla\v_k(s)\|^2_{H^s}\d
s\right]=\infty.
\end{align} Thus $(\v,\tau,\rho_{\infty})$ is a maximal local strong solution to (\ref{se1})-(\ref{se4}).
Now, in order to prove that this maximal strong solution is unique we let the triplet $(\tilde{\v},\tilde{\tau},\sigma_{\infty})$ be
another maximal solution and $\{\sigma_k,k\geq 0\}$ is an
increasing sequence of stopping times converging to
$\sigma_{\infty}$ and is defined by
\begin{align}\sigma_k=\inf_{t\geq 0}\left\{t:\mu_2\|\tilde{\v}(t)\|_{H^s}^2+\mu_1\|\tilde{\tau}(s)\|^2_{H^s}+2\mu_2 \nu\int_0^{t}\|\nabla\tilde{\v}(s)\|^2_{H^s}\d
s> k\right\},k\in\mathbb{N}.\end{align}   
Exploiting the same arguments as above and by the uniqueness Theorems (\autoref{uniqueness}
and \autoref{remunique}) one can prove that $\v(t)=\tilde{\v}(t)$,
$\tau(t)=\tilde{\tau}(t)$ for all $t\in[0,\rho_{k}\wedge\sigma_{k}\wedge T]$ a. s. for
$k\geq 0$. Hence,
\begin{align}\label{mx10}
\v(t)=\tilde{\v}(t), \tau(t)=\tilde{\tau}(t)\text{ for all
}t\in[0,\rho_{\infty}\wedge\sigma_{\infty}\wedge T]\text{ a. s. }
\end{align}
 on letting $k\uparrow\infty.$
From (\ref{mx10}), one can easily verify that
$\rho_{\infty}=\sigma_{\infty}$ a.s. However, if not, then either $\rho_{\infty}>\sigma_{\infty}$ or $\rho_{\infty}<\sigma_{\infty}$.
Now for the first case we have
\begin{align}\label{mx11}
&\lim_{t\uparrow\rho_{\infty}} \left[\sup_{0\leq s\leq
t}\mu_2\|\mathbf{1}_{\{\sigma_{\infty}<\rho_{\infty}\}}\v\|^2_{H^s}+\sup_{0\leq
s\leq
t}\mu_1 \|\mathbf{1}_{\{\sigma_{\infty}<\rho_{\infty}\}}\tau\|^2_{H^s}+2\mu_2\nu\int_0^{t}\|\nabla\mathbf{1}_{\{\sigma_{\infty}<\rho_{\infty}\}}\v\|^2_{H^s}\d
s\right]\nonumber\\
&=\lim_{k\uparrow\infty} \left[\sup_{0\leq s\leq
\rho_{k} \wedge T}\mu_2\|\mathbf{1}_{\{\sigma_{\infty}<\rho_{\infty}\}}
\v\|^2_{H^s}+\sup_{0\leq s\leq
\rho_{k} \wedge T}\mu_1\|\mathbf{1}_{\{\sigma_{\infty}<\rho_{\infty}\}}
\tau\|^2_{H^s}+2\mu_2\nu\int_0^{\rho_{k} \wedge T}\|\nabla\mathbf{1}_{\{\sigma_{\infty}<\rho_{\infty}\}}
\v\|^2_{H^s}\d s\right]
\nonumber\\
&=\lim_{k\uparrow\infty} \left[\sup_{0\leq s\leq
\sigma_k \wedge T}\mu_2\|\mathbf{1}_{_{\{\sigma_{\infty}<\rho_{\infty}\}}}
\tilde{\v}\|^2_{H^s}+\sup_{0\leq s\leq
\sigma_k \wedge T}\mu_1\|\mathbf{1}_{_{\{\sigma_{\infty}<\rho_{\infty}\}}}
\tilde{\tau}\|^2_{H^s}+2\mu_2\nu\int_0^{\sigma_k \wedge T}\|\nabla\mathbf{1}_{_{\{\sigma_{\infty}<\rho_{\infty}\}}}
\tilde{\v}\|^2_{H^s}\d s\right]\nonumber\\&=\infty.
\end{align}
and for the second case,
\begin{align}\label{mx12}
&\lim_{t\uparrow\sigma_{\infty}} \left[\sup_{0\leq s\leq
t}\mu_2\|\mathbf{1}_{\{\sigma_{\infty}>\rho_{\infty}\}}\tilde{\v}\|^2_{H^s}+\sup_{0\leq
s\leq
t}\mu_1\|\mathbf{1}_{\{\sigma_{\infty}>\rho_{\infty}\}}\tilde{\tau}\|^2_{H^s}+2\mu_2\nu\int_0^{t}\|\nabla\mathbf{1}_{\{\sigma_{\infty}>\rho_{\infty}\}}\tilde{\v}\|^2_{H^s}\d
s\right]\nonumber\\
&=\lim_{k\uparrow\infty} \left[\sup_{0\leq s\leq
\sigma_k \wedge T}\mu_2\|\mathbf{1}_{\{\sigma_{\infty}>\rho_{\infty}\}}\tilde{\v}\|^2_{H^s}+\sup_{0\leq s\leq
\sigma_k \wedge T}\mu_1\|\mathbf{1}_{\{\sigma_{\infty}>\rho_{\infty}\}}\tilde{\tau}\|^2_{H^s}+2\mu_2\nu\int_0^{t}\|\nabla\mathbf{1}_{\{\sigma_{\infty}>\rho_{\infty}\}}\tilde{\v}\|^2_{H^s}\d
s\right]
\nonumber\\
&=\lim_{k\uparrow\infty} \left[\sup_{0\leq s\leq
\rho_{k} \wedge T}\mu_2\|\mathbf{1}_{\{\sigma_{\infty}>\rho_{\infty}\}}
\v\|^2_{H^s}+\sup_{0\leq s\leq
\rho_{k} \wedge T}\mu_1\|\mathbf{1}_{\{\sigma_{\infty}>\rho_{\infty}\}}\tau\|^2_{H^s}+2\mu_2\nu\int_0^{t}\|\nabla\mathbf{1}_{\{\sigma_{\infty}>\rho_{\infty}\}}\v\|^2_{H^s}\d
s\right]\nonumber\\&=\infty.
\end{align}
 Identity (\ref{mx11}) contradicts the fact that $(\v,\tau)$
does not explode before the time $\rho_{\infty}$ and the next
identity  \eqref{mx12} contradicts the fact that $(\tilde{\v},\tilde{\tau})$ does not explode
before the time $\sigma_{\infty}$. Hence, the only possibility is
$\rho_{\infty}=\sigma_{\infty}$ a. s. and this proves the uniqueness
of  the maximal local strong solution $(\v,\tau,\rho_{\infty})$ of the Stochastic equations (\ref{se1})-(\ref{se4}).
\end{proof} 

Similar ideas of proving maximal local solutions in Proposition 3.11 of
Brze\'{z}niak et al. \cite{BHR}, Theorem 3.5 of Bessaih et al. \cite{BHHR}, and Theorem 5.4 of Manna et al. \cite{MaMo}.
 
Finally observing that since $\rho_{k} \uparrow \rho_{\infty}$, for any fixed $0 < \delta <1$ and for the choice of $k$ with $\dfrac{1}{k+1}\leq \delta<\dfrac{1}{k}$, we infer from \eqref{rhoN} in Theorem \ref{positive1},
	\begin{align*}
	\mathbb{P}\left(\rho_{\infty}>\delta\right)\geq 1-2\delta e^{(\tilde{C}+C_2 \delta)} \Big(2\mathbb{E}\left(\mu_2\|\v_0\|_{H^s}^2+\mu_1\|\mathbf{\tau}_0\|_{H^s}^2\right)+18K\mu_2\delta\Big).	
	\end{align*} 
\begin{appendix}
\section{}
The Appendix is devoted to prove some small Lemmas which were useful in proving Theorem \ref{cauchy} and Theorem \ref{uniqueness}.

\begin{lem} \label{v}
Let $\nabla\cdot\v=0.$ Then for $m>n$ there exists $C>0$ (independent of $n,m,\v, \epsilon$) such that
\begin{align}\label{cau.l1}
|\left(\left(\mathcal{J}_n [(\mathbf{v}_n \cdot\nabla)\v_n ]-\mathcal{J}_m [(\mathbf{v}_m \cdot\nabla)\v_m ]\right),
\v_n-\v_m \right)_{\mathrm{L}^2}|
&\leq \frac{C}{n^{\epsilon}}\|\v_n\|_{H^{s}}^2\|\v_n-\v_{m}\|_{\mathrm{L}^2} \notag \\
&\quad+ C\|\v_n-\v_{m}\|_{\mathrm{L}^2}^2\|\nabla\v_n\|_{H^s}.
\end{align}
\end{lem}
\begin{proof}
In order to estimate the term $\left(\left(\mathcal{J}_n [(\mathbf{v}_n \cdot\nabla)\v_n ]-\mathcal{J}_m [(\mathbf{v}_m \cdot\nabla)\v_m ]\right),\v_n-\v_m \right)_{\mathrm{L}^2}$ we split it into three parts 
\begin{align} \label{split}
&\left((\mathcal{J}_n-\mathcal{J}_{m})[(\v_n \cdot\nabla)\v_n],\v_n-\v_{m}\right)_{\mathrm{L}^2}+\left(\mathcal{J}_{m}[((\v_n-\v_{m})\cdot\nabla)\v_n],\v_n-\v_{m}\right)_{\mathrm{L}^2} \notag
\\&\qquad+\left(\mathcal{J}_{m}[(\v_{m} \cdot\nabla)(\v_n-\v_{m})],\v_n-\v_{m}\right)_{\mathrm{L}^2}.
\end{align}
For $m>n,$ we exploit H\"{o}lder's inequality, the cut off property
[see (\ref{sr})] and $H^s$ is an algebra for $0< \epsilon <s-1,s>d/2$ to the first term of \eqref{split} to achieve,
\begin{align}\label{split.1}
 \left|\left((\mathcal{J}_n-\mathcal{J}_{m})[(\v_n \cdot\nabla)\v_n],\v_n-\v_{m}\right)_{\mathrm{L}^2}\right|
&\leq \|(\mathcal{J}_n-\mathcal{J}_{m})(\v_n \cdot\nabla)\v_n \|_{\mathrm{L}^2}\|\v_n-\v_{m}\|_{\mathrm{L}^2}\nonumber\\
&\leq
\frac{C}{n^{\epsilon}}\|\v_n\|_{H^{s}}^2\|\v_n-\v_{m}\|_{\mathrm{L}^2}.
\end{align}
Direct application of H\"{o}lder's inequality to the second term of \eqref{split} yields
\begin{align}\label{split.2}
&\left|\left(\mathcal{J}_{m}[((\v_n-\v_{m})\cdot\nabla)\v_n],\v_n-\v_{m}\right)_{\mathrm{L}^2}\right|
=\left|\left([((\v_n-\v_{m})\cdot\nabla)\v_n],\mathcal{J}_{m}(\v_n-\v_{m})\right)_{\mathrm{L}^2}\right|\nonumber\\&\leq
\|\v_n-\v_{m}\|_{\mathrm{L}^2}\|\nabla\v_n\|_{\mathrm{L}^{\infty}}\|\v_n-\v_{m}\|_{\mathrm{L}^2}\leq
C\|\v_n-\v_{m}\|_{\mathrm{L}^2}^2\|\nabla\v_n\|_{H^s}.
 \end{align}
Now using the Parseval's identity, integration by parts and $\nabla \cdot \v_n= \nabla \cdot \v_m=0,$
we directly have the third term of \eqref{split} is zero.
Hence, we have \eqref{cau.l1}.
\end{proof}

\begin{lem} \label{tau}
Let $\nabla\cdot\v=0.$ Then for $m>n$ there exists $C>0$(independent of $n,m,\v, \epsilon$) such that
\begin{align}\label{cau.l2}
&|\left(\left(\mathcal{J}_n [(\mathbf{v}_n \cdot\nabla)\mathbf{\tau}_n ]-\mathcal{J}_m [(\mathbf{v}_m \cdot\nabla)\mathbf{\tau}_m ]\right),
\mathbf{\tau}_n-\mathbf{\tau}_m \right)_{\mathrm{L}^2}| \notag\\
&\leq \frac{C}{n^{\epsilon}}\left(\|\v_n\|_{H^{s}}^2+\|\mathbf{\tau}_n\|_{H^{s}}^2\right)\|\mathbf{\tau}_n-\mathbf{\tau}_{m}\|_{\mathrm{L}^2} 
+\frac{\nu\mu_2}{4\mu_1}\|\v_n-\v_{m}\|_{H^1}^2+\frac{C\mu_1}{\nu\mu_2}\|\mathbf{\tau}_n\|_{H^{s}}^2\|\mathbf{\tau}_n-\mathbf{\tau}_{m}\|_{\mathrm{L}^2}^2.
\end{align}
\end{lem}
\begin{proof}
Let us consider the term $\left(\mathcal{J}_n \left[(\v_n \cdot \nabla)\mathbf{\tau}_n \right]-\mathcal{J}_{m}[(\v_{m}\cdot\nabla)\mathbf{\tau}_{m}],\mathbf{\tau}_n-\mathbf{\tau}_{m}\right)_{\mathrm{L}^2}$ and split it into three parts
\begin{align} \label{split2}
&\left((\mathcal{J}_n-\mathcal{J}_{m})[(\v_n \cdot\nabla)\mathbf{\tau}_n],\mathbf{\tau}_n-\mathbf{\tau}_m \right)_{\mathrm{L}^2}
+\left(\mathcal{J}_{m}[((\v_n-\v_{m})\cdot\nabla)\mathbf{\tau}_n],\mathbf{\tau}_n-\mathbf{\tau}_m\right)_{\mathrm{L}^2}\notag \\
&\quad+\left(\mathcal{J}_{m}[(\v_{m}\cdot\nabla)(\mathbf{\tau}_n-\mathbf{\tau}_{m})],\mathbf{\tau}_n-\mathbf{\tau}_{m}\right)_{\mathrm{L}^2}.
\end{align}
We use here same arguments as used in the previous  Lemma \ref{v}. Hence the first term in \eqref{split2} is reduced to
\begin{align}
|\left((\mathcal{J}_n-\mathcal{J}_{m})[(\v_n \cdot\nabla)\mathbf{\tau}_n],\mathbf{\tau}_n-\mathbf{\tau}_m \right)_{\mathrm{L}^2}|
&\leq
\frac{C}{n^{\epsilon}}\|\v_n \|_{H^{s}}\|\mathbf{\tau}_n\|_{H^{s}}\|\mathbf{\tau}_n-\mathbf{\tau}_{m}\|_{\mathrm{L}^2}\nonumber\\&
\leq
\frac{C}{n^{\epsilon}}\left(\|\v_n\|_{H^{s}}^2+\|\mathbf{\tau}_n\|_{H^{s}}^2\right)\|\mathbf{\tau}_n-\mathbf{\tau}_{m}\|_{\mathrm{L}^2}.
\end{align}
Consider the term $\left(\mathcal{J}_{m}[((\v_n-\v_{m})\cdot\nabla)\mathbf{\tau}_n],\mathbf{\tau}_n-\mathbf{\tau}_m\right)_{\mathrm{L}^2}$ and using and use H\"{o}lder's inequality, Remark \ref{r23} to get
\begin{align} \label{split2.1}
&\left|\left(\mathcal{J}_{m}[((\v_n-\v_{m})\cdot\nabla)\mathbf{\tau}_n],\mathbf{\tau}_n-\mathbf{\tau}_m\right)_{\mathrm{L}^2}\right|
\quad=\left|\left([((\v_n-\v_{m})\cdot \nabla)\mathbf{\tau}_n],\mathcal{J}_{m}(\mathbf{\tau}_n-\mathbf{\tau}_{m})\right)_{\mathrm{L}^2}\right|
\nonumber\\&\quad\leq
\|((\v_n-\v_{m})\cdot\nabla)\mathbf{\tau}_n\|_{\mathrm{L}^2}\|\mathcal{J}_{m}(\mathbf{\tau}_n-\mathbf{\tau}_{m})\|_{\mathrm{L}^2}
\leq
C\|\v_n-\v_{m}\|_{H^1}\|\nabla\mathbf{\tau}_n\|_{H^{s-1}}\|\mathbf{\tau}_n-\mathbf{\tau}_{m}\|_{\mathrm{L}^2}\nonumber\\
&\quad\leq
\frac{\nu\mu_2}{4\mu_1}\|\v_n-\v_{m}\|_{H^1}^2+\frac{C\mu_1}{\nu\mu_2}\|\mathbf{\tau}_n\|_{H^{s}}^2\|\mathbf{\tau}_n-\mathbf{\tau}_{m}\|_{\mathrm{L}^2}^2.
\end{align} 
Once again on applying Parseval's identity, integration by parts and divergence free condition on $\v_m$ we get the third term of
\eqref{split2} to be zero. Hence, we have \eqref{cau.l2}.
\end{proof}

\begin{lem} \label{Q}
For $m>n$ there exists $C>0$ (independent of $n,m,\v, \epsilon$) such that
\begin{align}\label{cau.l3}
&|\left(\mathcal{J}_n \Q(\mathbf{\tau}_n , \nabla \mathbf{v}_n)-\mathcal{J}_m \Q(\mathbf{\tau}_m ,\nabla \mathbf{\v}_m ),
\mathbf{\tau}_n-\mathbf{\tau}_m \right)_{\mathrm{L}^2}|\notag\\
&\leq \frac{C}{n^{\epsilon}}\left(\|\v_n\|_{H^{s}}^2
+\|\mathbf{\tau}_n\|_{H^{s}}^2\right)\|\mathbf{\tau}_n-\mathbf{\tau}_{m}\|_{\mathrm{L}^2}
+\frac{C}{2n^\epsilon} \|\nabla\v_n\|^2_{H^{s}}\|\mathbf{\tau}_n-\mathbf{\tau}_m\|_{\mathrm{L}^2}\notag \\
&\quad +\frac{\nu\mu_2}{4\mu_1}\|\v_n-\v_{m}\|_{H^1}^2
+\frac{C\mu_1}{\nu\mu_2}\|\mathbf{\tau}_m\|_{H^{s}}^2\|\mathbf{\tau}_n-\mathbf{\tau}_{m}\|_{\mathrm{L}^2}^2
+C\|\nabla\v_n \|_{H^s}\|\mathbf{\tau}_n-\mathbf{\tau}_m\|_{\mathrm{L}^2}^2.
\end{align}
\end{lem}
\begin{proof}
We recall from the definition of $\Q$ that $\mathcal{J}_n \Q(\mathbf{\tau}_n, \nabla \v_n)=\mathcal{J}_n \Big(\mathbf{\tau}_n \mathcal{W}(\v_n)\Big)-\mathcal{J}_n \Big(\mathcal{W}(\v_n) \mathbf{\tau}_n\Big) -b \mathcal{J}_n\Big(\mathcal{D}(\v_n)\mathbf{\tau}_n-\mathbf{\tau}_n \mathcal{D}(\v_n)\Big).$
Let us estimate the following term for $m>n,$
$$\mathcal{J}_n \Big(\mathbf{\tau}_n \mathcal{W}(\v_n)\Big)-\mathcal{J}_m \Big(\mathbf{\tau}_m \mathcal{W}(\v_m)\Big)=\Big(\mathcal{J}_n - \mathcal{J}_m \Big) \mathbf{\tau}_n \mathcal{W}(\v_n)
+\mathcal{J}_m \Big((\mathbf{\tau}_n-\mathbf{\tau}_m) \mathcal{W}(\v_n)\Big)+\mathcal{J}_m \Big(\mathbf{\tau}_m (\mathcal{W}(\v_n)-\mathcal{W}(\v_m))\Big).$$
Hence,
\begin{align*}
&\Big|\Big(\mathcal{J}_n \Big(\mathbf{\tau}_n \mathcal{W}(\v_n)\Big)-\mathcal{J}_m \Big(\mathbf{\tau}_m \mathcal{W}(\v_m)\Big),\tau_n-\tau_m \Big)_{\mathrm{L}^2}\Big|\notag \\
&\leq \Big|\underbrace{\Big((\mathcal{J}_n - \mathcal{J}_m) \mathbf{\tau}_n \mathcal{W}(\v_n),\tau_n-\tau_m\Big)_{\mathrm{L}^2}}_{I_{13}}\Big|
+\Big|\underbrace{\Big( \mathcal{J}_m \Big((\mathbf{\tau}_n-\mathbf{\tau}_m) \mathcal{W}(\v_n)\Big),\tau_n-\tau_m \Big)_{\mathrm{L}^2}}_{I_{14}}\Big| \notag\\
&\quad +\Big|\underbrace{\mathcal{J}_m \Big(\mathbf{\tau}_m (\mathcal{W}(\v_n)-\mathcal{W}(\v_m))\Big),\tau_n -\tau_m \Big)_{\mathrm{L}^2}}_{I_{15}}\Big|.
\end{align*}
Now, for $0<\epsilon<s-1,$ using Remark \ref{prohs}, properties of $\mathcal{J}_n$ ( \eqref{intro1}, \eqref{intro2a}, \eqref{sr} in Subsection \ref{FTO}) and for $s>0,$ using the embedding $H^s \subset H^{s-1}$ and using Young's inequality we get
\begin{align*}
|I_{13}| &\leq \|(\mathcal{J}_n - \mathcal{J}_m) \mathbf{\tau}_n \mathcal{W}(\v_n)\|_{\mathrm{L}^2} \|\tau_n-\tau_m\|_{\mathrm{L}^2}
\leq \frac{C}{n^\epsilon} \|\mathbf{\tau}_n \mathcal{W}(\v_n)\|_{H^\epsilon}\|\tau_n-\tau_m\|_{\mathrm{L}^2}\\
&\leq  \frac{C}{n^\epsilon} \Big[\|\mathbf{\tau}_n\|_{H^s} \|\nabla \v_n\|_{H^s}+\|\tau_n\|_{H^s} \|\v_n\|_{H^s}\Big] \|\tau_n-\tau_m\|_{\mathrm{L}^2}\\
&\leq  \frac{C}{n^\epsilon} \Big( \|\v_n\|_{H^s}^2+\|\tau_n\|_{H^s}^2\Big) \|\tau_n-\tau_m\|_{\mathrm{L}^2}+ \frac{C}{2n^\epsilon} \|\nabla \v_n\|_{H^s}^2 \|\tau_n-\tau_m\|_{\mathrm{L}^2}.
\end{align*}
Again applying similar arguments we have, 
\begin{align}
|I_{14}| \leq \|(\mathbf{\tau}_n-\mathbf{\tau}_m) \mathcal{W}(\v_n)\|_{\mathrm{L}^2}\|\mathcal{J}_m (\tau_n-\tau_m)\|_{\mathrm{L}^2} 
\leq C \|\nabla \v_n\|_{H^s} \|\tau_n-\tau_m\|_{\mathrm{L}^2}^2.
\end{align}

Similarly the term $I_{15}$ is reduced to
\begin{align}
|I_{15}| &\leq C \|\mathbf{\tau}_m (\mathcal{W}(\v_n)-\mathcal{W}(\v_m))\|_{\mathrm{L}^2} \|\tau_n-\tau_m\|_{\mathrm{L}^2} 
 \leq C \|\mathbf{\tau}_m \|_{\mathrm{L}^\infty} \|\mathcal{W}(\v_n)-\mathcal{W}(\v_m)\|_{\mathrm{L}^2} \|\tau_n-\tau_m\|_{\mathrm{L}^2} \notag\\
 &\leq \frac{\nu \mu_2}{16 \mu_1} \|\nabla (\v_n-\v_m)\|_{\mathrm{L}^2}^2+\frac{C \mu_1}{\nu \mu_2} \|\tau_m\|_{H^s}^2 \|\tau_n-\tau_m\|_{\mathrm{L}^2}^2.
\end{align}
Therefore after combining all the similar estimates for other terms of $\mathcal{J}_n \Q(\tau_n, \nabla \v_n),$ finally we have \eqref{cau.l3}.

\end{proof}


 

\begin{lem}\label{M1}
 Let $\eta(t)$ be a stochastic process. Let $\rho_N^n$ be the stopping time given by \eqref{stop1}. Denote $ \Rn= \rho^n_N \wedge \rho^m_N.$ Then
 \begin{align}\label{est.m1}
  &\mathbb{E}\left[\sup_{0\leq t\leq \Lm}\left|\int_0^{t}\eta(s)
  \Big(\Big(\mathcal{J}_{n}\sigma(s,\v_n)-\mathcal{J}_{m}\sigma(s,\v_m)\Big)
  \d W_1(s),\v_n(s)-\v_m(s)\Big)_{\mathrm{L}^2}\right|\right] \notag\\
  &\leq \frac{1}{8}\mathbb{E}\left(\sup_{0\leq t\leq \Lm}\eta(t)\|\v_n(t)-\v_m(t)\|^2_{\mathrm{L}^2}\right)\notag\\
  &\quad+\frac{4CK}{n^{\epsilon}}\mathbb{E}\left[\int_{0}^{\Lm}\eta(t)(1+\|\v_n\|_{H^s}^2)\d
t\right]+4L\,\mathbb{E}\left[\int_{0}^{\Lm}\eta(t)\|\v_n-\v_{m}\|_{\mathrm{L}^2}^2\d
t\right].
 \end{align}
\end{lem}
\begin{proof}
Let us take the term on the left hand side of the estimate \eqref{est.m1} and apply Burkholder-Davis-Gundy
inequality and Young's inequality to get
\begin{align}\label{mhde44}
&\mathbb{E}\left[\sup_{0\leq t\leq \Lm}\left|\int_0^{t}\eta(s)
  \Big(\Big(\mathcal{J}_{n}\sigma(s,\v_n)-\mathcal{J}_{m}\sigma(s,\v_m)\Big)
  \d W_1(s),\v_n(s)-\v_m(s)\Big)_{\mathrm{L}^2}\right|\right]\notag\\
  &\leq
2\sqrt{2}\mathbb{E}\left[\int_{0}^{\Lm}(\eta(t))^2\|\mathcal{J}_{n}\sigma(t,\v_n)
-\mathcal{J}_{m}\sigma(t,\v_{m})\|_{\mathcal{L}_Q(\mathrm{L}^2,\mathrm{L}^2)}^2\|\v_n(t)-\v_m(t)\|_{\mathrm{L}^2}^2\d t\right]^{1/2}\nonumber\\
&\leq 2\sqrt{2}\mathbb{E}\left[\left(\sup_{0\leq t\leq
\Lm}\eta(t)\|\v_n(t)-\v_m(t)\|^2_{\mathrm{L}^2}\right)^{1/2}\times\right.\nonumber\\
&\qquad\left.
\left(\int_{0}^{\Lm}\eta(t)\|\mathcal{J}_{n}\sigma(t,\v_n)
-\mathcal{J}_{m}\sigma(t,\v_{m})\|_{\mathcal{L}_Q(\mathrm{L}^2,\mathrm{L}^2)}^2\d t\right)^{1/2}\right]\nonumber\\
&\leq \frac{1}{8}\mathbb{E}\left(\sup_{0\leq t\leq
\Lm} \eta(t)\|\v_n(t)-\v_{m}(t)\|^2_{\mathrm{L}^2}\right)
\nonumber\\&\quad+4\mathbb{E}\left[\int_{0}^{\Lm}\eta(t)\|\mathcal{J}_{m}\sigma(t,\v_n)
-\mathcal{J}_{n}\sigma(t,\v_{m})\|_{\mathcal{L}_Q(\mathrm{L}^2,\mathrm{L}^2)}^2 dt\right].
\end{align} 
Now, exploiting Assumption \ref{hypo} and cut off property (for the noise term) \eqref{sr2} we have for 
$0<\epsilon<s-1,$
\begin{align} \label{s.tau.2}
&\mathbb{E}\left[\int_{0}^{\Lm} \eta(t)\|\mathcal{J}_{n}\sigma(t,\v_n)
-\mathcal{J}_{m}\sigma(t,\v_m)\|_{\mathcal{L}_Q(\mathrm{L}^2,\mathrm{L}^2)}^2\d
t\right] \notag \\
&\leq
\mathbb{E}\left[\int_{0}^{\Lm} \eta(t)\|(\mathcal{J}_{n}-\mathcal{J}_{m})
\sigma(t,\v_n)\|_{\mathcal{L}_Q(\mathrm{L}^2,\mathrm{L}^2)}^2\d
t\right]\nonumber\\&\quad
+\mathbb{E}\left[\int_{0}^{\Lm}\eta(t)
\|\mathcal{J}_{m}(\sigma(t,\v_n)-\sigma(t,\v_{m}))\|_{\mathcal{L}_Q(\mathrm{L}^2,\mathrm{L}^2)}^2\d
t\right]\nonumber\\&\leq
\frac{C}{n^{\epsilon}}\mathbb{E}\left[\int_{0}^{\Lm}\eta(t)\|\sigma(t,\v_n)\|_{\mathcal{L}_Q(\mathrm{L}^2,H^{\epsilon})}^2\d
t\right]
+\mathbb{E}\left[\int_{0}^{\Lm}\eta(t)\|\sigma(t,\v_n)-\sigma(t,\v_m)\|_{\mathcal{L}_Q(\mathrm{L}^2,\mathrm{L}^2)}^2\d
t\right]\nonumber\\
&\leq
\frac{CK}{n^{\epsilon}}\mathbb{E}\left[\int_{0}^{\Lm}\eta(t)(1+\|\v_n\|_{H^s}^2)\d
t\right]+L\,\mathbb{E}\left[\int_{0}^{\Lm}\eta(t)\|\v_n-\v_{m}\|_{\mathrm{L}^2}^2\d
t\right].
\end{align}

\end{proof}


  \begin{lem}\label{M2}
 Let $\eta(t)$ be a stochastic process. Let $\rho_N^n$ be the stopping time given by \eqref{stop1}. Denote $ \Rn= \rho^n_N \wedge \rho^m_N.$ Then
 \begin{align}\label{est.m2}
  &\mathbb{E}\left[\sup_{0\leq t\leq \Lm}\left|\int_0^{t}\eta(s)
  \Big(\mathcal{J}_{n}\mathcal{S}(\mathbf{\tau}_n)-\mathcal{J}_{m}\mathcal{S}(\mathbf{\tau}_m),\mathbf{\tau}_n(s)-\mathbf{\tau}_m(s)\Big)_{\mathrm{L}^2}\d W_2(s)\right|\right] \notag\\
  & \leq \frac{1}{8}\mathbb{E}\left(\sup_{0\leq t\leq \Lm}\eta(t)\|\mathbf{\tau}_n(t)-\mathbf{\tau}_m(t)\|^2_{\mathrm{L}^2}\right)\notag\\
  &\quad+\frac{4}{n^\epsilon}\|h\|_{H^{s}}^2 \mathbb{E}\left(\int_0^{\Lm}\eta(t)\|\mathbf{\tau}_n\|_{H^{s}}^2 \d t \right)
+4\|h\|^2_{\mathrm{L}^\infty} \mathbb{E}\left(\int_0^{\Lm}\eta(t)\|\mathbf{\tau}_n-\mathbf{\tau}_m\|_{\mathrm{L}^2}^2\d t \right).
 \end{align}
\end{lem}
\begin{proof}
 
Using \eqref{intro1}, \eqref{sr}, \eqref{propS.2}, we have the following inequality 
 \begin{align}\label{S.2}
  \|\mathcal{J}_{n}\mathcal{S}(\mathbf{\tau}_n)-\mathcal{J}_{m} \mathcal{S}(\mathbf{\tau}_m)\|_{\mathrm{L}^2}^2
  \leq \frac{2}{n^{\epsilon}}\|h\|_{H^{s}}^2\|\mathbf{\tau}_n\|_{H^{s}}^2
+2\|h\|^2_{\mathrm{L}^\infty}\|\mathbf{\tau}_n-\mathbf{\tau}_m\|_{\mathrm{L}^2}^2.
 \end{align}
Again application of Burkholder-Davis-Gundy inequality, Young's inequality and \eqref{S.2} produces the required estimate.
\end{proof}

\begin{lem}\label{M3}
Let $\eta(t)$ be a stochastic process. Let $\rho_N^n$ be the stopping time given by \eqref{stop1}. Denote $ \Rn= \rho^n_N \wedge \rho^m_N.$ Then
 \begin{align}\label{est.m3}
  &\mathbb{E}\left[\sup_{0\leq t\leq \Lm}\left|\int_0^{t}\int_{Z}\eta(s)
  \Big(\mathcal{J}_{n}G(\v_n(s-),z)-\mathcal{J}_{m}G(\v_m(s-),z),\v_n(s-)-\v_m(s-)\Big)_{\mathrm{L}^2}
  \tilde{N}(\d s,\d z)\right|\right] \notag\\
  & \leq \frac{1}{8}\mathbb{E}\left(\sup_{0\leq t\leq \Lm}\eta(t)\|\v_n(t)-\v_m(t)\|^2_{\mathrm{L}^2}\right)+\frac{4CK}{n^{\epsilon}}\mathbb{E}\left[\int_{0}^{\Lm}\eta(t)(1+\|\v_n\|_{H^s}^2)d
t\right]\notag\\& \quad +4L\,\mathbb{E}\left[\int_{0}^{\Lm}\eta(t)\|\v_n-\v_{m}\|_{\mathrm{L}^2}^2dt\right].
 \end{align}
\end{lem}
\begin{proof}
An application of Burkholder-Davis-Gundy inequality, Young's inequality produces the required estimate
\begin{align}
 &\mathbb{E}\left[\sup_{0\leq t\leq \Lm}\left|\int_0^{t}\int_{Z}\eta(s)
  \Big(\mathcal{J}_{n}G(\v_n(s-),z)-\mathcal{J}_{m}G(\v_m(s-),z),\v_n(s-)-\v_m(s-)\Big)_{\mathrm{L}^2}
  \tilde{N}(\d s,\d z)\right|\right] \notag\\
  &\leq \frac{1}{8}\mathbb{E}\left(\sup_{0\leq t\leq \Lm}\eta(t)\|\v_n(t)-\v_m(t)\|^2_{\mathrm{L}^2}\right)\notag\\
  &\quad+4\mathbb{E}\left(\int_0^{\Lm}\int_Z\eta(t)
  \|\mathcal{J}_{n}G(\v_n(s-),z)-\mathcal{J}_{m}G(\v_m(s-),z)\|^2_{\mathrm{L}^2}\lambda(\d z)\d t\right).
\end{align}
Further exploiting Assumption \ref{hypo} and \eqref{sr}, for $0<\epsilon< s-1$ we have 
\begin{align} \label{G.lam.2}
&\mathbb{E}\left[\int_{0}^{\Lm}\int_Z \eta(t)\|\mathcal{J}_{n} G (\v_n,z)
-\mathcal{J}_{m} G(\v_{m},z)\|^2_{\mathrm{L}^2}\lambda(\d
z) \d t\right] \notag \\
& \leq 
\mathbb{E}\left[\int_{0}^{\Lm} \int_Z \eta(t)\|(\mathcal{J}_{n}-\mathcal{J}_{m})G(\v_n,z)\|^2_{\mathrm{L}^2}\lambda(\d
z) \d t\right]\nonumber\\&\quad+
\mathbb{E}\left[\int_{0}^{\Lm}\int_Z \eta(t)\|
\mathcal{J}_{m}(G(\v_n,z)-G(\v_m,z))\|^2_{\mathrm{L}^2}\lambda(\d
z) \d t\right]\nonumber\\&\leq
\frac{C}{n^{\epsilon}}\mathbb{E}\left[\int_{0}^{\Lm}\int_Z \eta(t)\|G(\v_n,z)\|_{H^{\epsilon}}^2\lambda(\d
z)\d t\right]+
\mathbb{E}\left[\int_{0}^{\Lm}\int_Z\eta(t)\|
G(\v_n,z)-G(\v_{m},z)\|^2_{\mathrm{L}^2}\lambda(\d
z) \d t\right]\nonumber\\&\leq
\frac{CK}{n^{\epsilon}}\mathbb{E}\left[\int_{0}^{\Lm}\eta(t)(1+\|\v_n\|_{H^s}^2)\d
t\right]+L\,\mathbb{E}\left[\int_{0}^{\Lm}\eta(t)\|\v_n-\v_{m}\|_{\mathrm{L}^2}^2\d t\right].
\end{align}
\end{proof}

\begin{lem} \label{lambda}
Let $\eta(t)$ be a stochastic process. Let $\rho_N^n$ be the stopping time given by \eqref{stop1}. Denote $ \Rn= \rho^n_N \wedge \rho^m_N.$ Then
\begin{align}
&\mathbb{E}\left[\int_{0}^{\Lm}\int_Z \eta(t)\|\mathcal{J}_{n} G (\v_n,z)
-\mathcal{J}_{m} G(\v_{m},z)\|^2_{\mathrm{L}^2}\lambda(\d
z) \d t\right]\notag\\ &\leq
\frac{CK}{n^{\epsilon}}\mathbb{E}\left[\int_{0}^{\Lm}\eta(t)(1+\|\v_n\|_{H^s}^2)\d
t\right]+L\,\mathbb{E}\left[\int_{0}^{\Lm}\eta(t)\|\v_n-\v_{m}\|_{\mathrm{L}^2}^2\d t\right].
\end{align}
\end{lem}

\begin{proof}
Exploiting Assumption \ref{hypo} and \eqref{sr}, for $0<\epsilon< s-1,$ we have 
\begin{align}
&\mathbb{E}\left[\int_{0}^{\Lm}\int_Z \eta(t)\|\mathcal{J}_{n} G (\v_n,z)
-\mathcal{J}_{m} G(\v_{m},z)\|^2_{\mathrm{L}^2}\lambda(d
z) \d t\right] \notag \\
& \leq 
\mathbb{E}\left[\int_{0}^{\Lm} \int_Z \eta(t)\|(\mathcal{J}_{n}-\mathcal{J}_{m})G(\v_n,z)\|^2_{\mathrm{L}^2}\lambda(\d
z) \d t\right]\nonumber\\&\quad+
\mathbb{E}\left[\int_{0}^{\Lm}\int_Z \eta(t)\|
\mathcal{J}_{m}(G(\v_n,z)-G(\v_m,z))\|^2_{\mathrm{L}^2}\lambda(\d
z) \d t\right]\nonumber\\&\leq
\frac{C}{n^{\epsilon}}\mathbb{E}\left[\int_{0}^{\Lm}\int_Z \eta(t)\|G(\v_n,z)\|_{H^{\epsilon}}^2\lambda(\d
z)\d t\right]+
\mathbb{E}\left[\int_{0}^{\Lm}\int_Z\eta(t)\|
G(\v_n,z)-G(\v_{m},z)\|^2_{\mathrm{L}^2}\lambda(\d
z) \d t\right]\nonumber\\&\leq
\frac{CK}{n^{\epsilon}}\mathbb{E}\left[\int_{0}^{\Lm}\eta(t)(1+\|\v_n\|_{H^s}^2)d
t\right]+L\,\mathbb{E}\left[\int_{0}^{\Lm}\eta(t)\|\v_n-\v_{m}\|_{\mathrm{L}^2}^2dt\right].
\end{align}
\end{proof}

\end{appendix}

\section*{Acknowledgement}
The authors would like to thank Professor Zdzislaw Brze\'{z}niak of University of York, UK for his valuable comments and pointing our attention to certain references.


\begin{thebibliography}{00}
\bibitem{AF} {\sc Adams, R. A., and Fournier, J. J. F.:} 
{\em Sobolev Spaces,}
Pure and Applied Mathematics, (Amsterdam) \textbf{140}, 
Academic press, 1975.


\bibitem{Ap}  {\sc Applebaum, D.:}
{\em L\'{e}vy Processes and Stochastic Calculus}, Cambridge Studies
in Advanced Mathematics, \textbf{93}, Cambridge University press, 2004.

\bibitem{Ba} {\sc Barbu, V., Bonaccorsi, S. and Tubaro, L.:} Existence and Asymptotic Behavior for Hereditary
Stochastic Evolution Equations, {\em Appl. Math. Optim.}, {\bf 69}, 273--314, 2014.


\bibitem{BKM} {\sc Beale, J., Kato, T., and Majda, A.:} 
Remarks on the breakdown of smoothness for the 3-D
Euler equations, {\em Comm. Math. Phys.}, {\bf 94}, 61--66, 1984. 

\bibitem{BF} {\sc Bessaih, H., Ferrario, B.}
The regularized 3D Boussinesq equations with fractional Laplacian and no diffusion,
arXiv:1504.05067v1 math.AP.

\bibitem{BHHR} {\sc Bessaih, H., Hausenblas, E., and Razafimandimby, P.:}
Strong Solutions to Stochastic Hydrodynamical Systems with
Multiplicative Noise of Jump Type, {\em Nonlinear Differential
Euqations and Appl.}, {\bf 22}(6), 1661--1697, 2015.




\bibitem{BiP}  {\sc Billingsley, P.:}
{\em Convergence of Probability Measures}, Wiley, New York, 1969.


\bibitem{BHR} {\sc Brze\'{z}niak, Z., Hausenblas, E., and Razafimandimby, P.:}
Stochastic Nonparabolic Dissipative Systems Modeling the Flow of
Liquid Crystals: Strong Solution, {\em Mathematical Analysis of
Incompressible Flow}, {\bf 1875}, 41--72, 2014.

\bibitem{BHZ} {\sc Brze\'{z}niak, Z., Hausenblas, E., and Zhu, J.:}
2D Stochastic Navier-Stokes equations driven by jump noise, {\em Nonlinear Analysis}, {\bf 79}, 123--139, 2013.

\bibitem{BS} {\sc Brze\'{z}niak, Z., Szymon, P.:}
Strong local and global solutions for stochastic Navier-Stokes equations.  
{\em Infinite dimensional stochastic analysis} (Amsterdam, 1999), 85--98, 
Verh. Afd. Natuurkd. 1. Reeks. K. Ned. Akad. Wet., 52, R. Neth. Acad. Arts Sci., Amsterdam, 2000.


\bibitem{CKS} {\sc Caflisch, R. E., Klapper, I. and Steele, G.} Remarks on Singularities,
Dimension and Energy Dissipation for Ideal Hydrodynamics and MHD,
{\em Communications in  Mathematical Physics}, {\bf 184} (2)
443-455, 1997.

\bibitem{C} {\sc Chemin, J.-Y.:} {\em Perfect Incompressible Fluids}, Oxford University Press, New York, 1998.


\bibitem{CM} {\sc Chemin,J-Y.,and Masmoudi,N.:} About Lifespan Of Regular Solutions Of Equations related to Viscoelastic Fluids, {\em SIAM J. Math. Anal.}, {\bf 33}(1), 84–-112, 2001.

\bibitem{DaZ}  {\sc Da Prato, G. and Zabczyk, J.:}
{\em Stochastic Equations in Infinite Dimensions}, Cambridge
University Press, 1992.

\bibitem{ER} {\sc Elgindi, T.M. and Rousset, F.:}  Global Regularity for Some Oldroyd-B Type Models, {\em Communications on Pure and Applied Mathematics,} {\bf LXVIII}, 2005--2021, 2015.

\bibitem{FMRR} {\sc Fefferman, C. L., McCormick, D. S., Robinson, J. C. and Rodrigo, J. L.:} Higher
Order Commutator Estimates and Local Existence for the Non-resistive
MHD Equations and Related Models, {\em Journal of Functional
Analysis}, {\bf 267}, 1035--1056, 2014.

\bibitem{FGO} {\sc Fernandez-Cara, E., Guill\'en, F. and Ortega, R.R:} Existence et unicit\'e de solution forte
locale en temps pour des fluides non newtoniens de type Oldroyd (version $L^s-L^r$); {\em C. R.
Acad. Sci. Paris S\'er. I Math.}, \textbf{319}, 411--416, 1994.

\bibitem{GaMa}  {\sc Gawarecki, L. and Mandrekar, V.:}
{\em Stochastic Differential Equations in Infinite Dimensions with
Applications to Stochastic Partial Differential Equations},
Springer-Verlag, New York, 2011.

\bibitem{NV} {\sc Glatt-Holtz, N. E. and Vicol, V. C.}
Local and global existence of smooth solutions for the stochastic Euler equations with multiplicative noise;
{\em Ann. Probab.}, {\bf 42}, no. 1, 80--145, 2014.

\bibitem{GS1} {\sc Guillop\'e, C and Saut, J.-C.:} 
Existence results for the flow of viscoelastic fluids with a differential
constitutive law; {\em Nonlinear Anal.}, \textbf{15}, 849--869, 1990. 

\bibitem{GS2} {\sc Guillop\'e, C and Saut, J.-C.:} 
Global existence and one-dimensional nonlinear stability of
shearing motions of viscoelastic fluids of Oldroyd type; {\em RAIRO Mod\'el. Math. Anal. Num\'er.}, \textbf{24}, 369--401, 1990.

\bibitem{H1} {\sc Hmidi, T., Keraani, S. and Rousset, F.:} Global well-posedness for a Boussinesq-Navier-Stokes
system with critical dissipation, {\em J. Differential Equations},  {\bf 249}(9), 2147--2174, 2010.

\bibitem{H2} {\sc Hmidi, T., Keraani, S. and Rousset, F.:} Global well-posedness for Euler-Boussinesq system with critical dissipation, {\em  Comm. Partial Differential Equations}, {\bf 36}(3), 420--445, 2011.

\bibitem{H3} {\sc Hmidi, T. and Rousset, F.:} Global well-posedness for the Navier-Stokes-Boussinesq system with
axisymmetric data, {\em Ann. Inst. H. Poincar\'e Anal. Non Lin\'eaire}, {\bf 27}(5), 1227--1246, 2010.

\bibitem{H4} {\sc Hmidi, T. and Rousset, F.:} Global well-posedness for the Euler-Boussinesq system with axisymmetric
data, {\em J. Funct. Anal.}, {\bf 260}(3), 745--796, 2011.

\bibitem{IW}  {\sc Ikeda, N. and Watanabe, S.:} {\em Stochastic
Differential Equations and Diffusion Processes (second edition)};
North-Holland-Kodansha, Tokyo, 1989.

\bibitem{JLL1} {\sc Jourdain, B., Leli\`evre, T. and Le Bris, C.:}
Numerical analysis of micro-macro simulations of polymeric
fluid flows: a simple case, {\em Math. Models Methods Appl. Sci.}, {\bf 12} (9), 1205--1243, 2002.

\bibitem{JLL2} {\sc Jourdain, B., Leli\`evre, T. and Le Bris, C.:}
Existence of solution for a micro-macro model
of polymeric fluid: the FENE model; {\em J. of Funct. Anal.}, {\bf 209}, 162--193, 2004.

\bibitem{JLLO} {\sc Jourdain, B.,  Le Bris, C., Leli\`evre, T. and Otto, F.:}
Long-Time Asymptotics of a Multiscale Model for Polymeric Fluid Flows, {\em Arch. Rational Mech. Anal.}, {\bf 181}, 97--148, 2006.


\bibitem{KaPo} {\sc Kato, T. and Ponce, G.:} Commutator Estimats and the Euler and
Navier-Stokes Equations, {\em Communucations in Pure and  Applied
Mathematics}, {\bf 41}(7), 891--907, 1988.

\bibitem{Ks} {\sc Kesavan, S.,}
{\em Topics in Functional Analysis and Applications}, Second Ed., 2015.

\bibitem{Kim} {\sc Kim, J. U.:}
Existence of a local smooth solution in probability to the stochastic Euler equations in $\mathbb{R}^3$;
{\em Journal of Functional Analysis}, {\bf 256}, 3660-3687, 2009.

\bibitem{Kuo} {\sc Kuo, H.-H.} {\em Introduction to Stochastic Integration}, Springer, New York, 2006.

\bibitem{PL} {\sc Lax, P. D.} {\em Functional Analysis}, Wiley Interscience, 2002.  
 
\bibitem{Lei1} {\sc Lei, Z., Liu, C. and Zhou, Y.:} 
Global existence for a 2D incompressible viscoelastic model with
small strain; {\em Commun. Math. Sci.}, \textbf{5}(3), 595--616, 2007. 

\bibitem{Lei2} {\sc Lei, Z., Liu, C. and Zhou, Y.:}
 Global solutions for incompressible viscoelastic fluids; {\em Arch. Ration.
 Mech. Anal.}, \textbf{188} (3), 371--398, 2008.

\bibitem{LMZ} {\sc Lei, Z., Masmoudi, N. and Zhou, Y.:} 
Remarks on the blowup criteria for Oldroyd models; {\em J. Differential
Equations}, \textbf{248}(2), 328--341, 2010.
 
\bibitem{Li} {\sc Li, T., Vanden-Eijnden, E., Zhang, P. and E, W.:}
Stochastic models of polymeric fluids at small Deborah number; {\em J. Non-Newtonian Fluid Mech.}, {\bf 121}, 117--125, 2004.
 
\bibitem{LLZ} {\sc Lin, F.-H., Liu, C. and Zhang, P.:} 
On hydrodynamics of viscoelastic fluids; {\em Comm. Pure Appl. Math.},
\textbf{58}(11), 1437--1471, 2005.
 
\bibitem{LL} {\sc Lin, F.-H. and Liu, C.:} 
Nonparabolic dissipative systems modeling the flow of Liquid Crystals; {\em Comm. Pure Appl. Math.}, {\bf XLVIII}, 501--537, 1995.

\bibitem{LM} {\sc Lions, and J.-L., Magenes, E.} 
{\em Non-homogeneous boundary Value problems and Applications,}
\textbf{1}, Springer-Verlag, 
New York, 1972.

\bibitem{LMa} {\sc Lions, P.-L. and Masmoudi, N.} 
Global solutions for some Oldroyd models of non-Newtonian
flows; {\em Chinese Ann. Math. Ser. B}, \textbf{21}, 131–-146, 2000.

\bibitem{MaRu}  {\sc Mandrekar, V. and R\"{u}diger, B.:} Existence and Uniqueness of
path-wise solutions for stochastic integral equations driven by
L\'{e}vy noise on separable Banach space, {\em Stochastics}, {\bf
 78} (4), 189--212, 2006.

\bibitem{MR}  {\sc Mandrekar, V. and R\"{u}diger, B.:} L\'{e}vy Noises and Stochastic Integrals on Banach
Spaces. Stochastic Partial Differential Equations and
Applications-VII, {\em Lect. Notes Pure Appl. Math., Chapman $\&$
Hall/CRC, Boca Raton}, {\bf
 7}, 193--213, 2006.


\bibitem{MaRu1}  {\sc Mandrekar, V. and R\"{u}diger, B.:} {\em
 Stochastic Integration in Banach Spaces: Theory and Applications},
 Springer Cham Heidelberg New York, 2015.

\bibitem{MaMo} {\sc Manna, U., Mohan, M. T. and Sritharan. S. S.:}  Stochastic non-resistive Magnetohydrodynamic system with L\'{e}vy noise, {\em preprint}.

\bibitem{MP} {\sc Manna, U., and Panda, A. A.:}
Higher Order Regularity and Blow-up Criterion for Semi-dissipative and Ideal Boussinesq Equations,
{\em preprint}.

\bibitem{Me1}   {\sc M\'{e}tivier, M.:}
{\em Semimartingales: A Course on Stochastic Processes}, {Walter de
Gruyter, Berlin}, 1982.

 \bibitem{Me}   {\sc M\'{e}tivier, M.:}
{\em Stochastic Partial Differential Equations in Infinite
Dimensional Spaces}, {Quaderni, Scuola Normale Superiore}, Pisa,
1988.

\bibitem{MCRM} {\sc Marinelli, C. and R\"{o}ckner M.:} On the Maximal
Inequalities of Burkholder, Davis and Gundy, {\em Expositiones
Mathematicae}, Article in Press, 2015.

\bibitem{Mot} {\sc Motyl, E.:}
Stochastic Navier-Stokes equations driven by L\'evy noise in unbounded 3D domains,
{\em Potential Anal.}, {\bf 38}, 863--912, 2013. 

\bibitem{Old} {\sc Oldroyd, J.G.:} 
Non-Newtonian effects in steady motion of some idealized elastico-viscous liquids, {\em Proc. Roy. Soc. London Ser. A}, \textbf{245}, 278--297, 1958.

\bibitem{Ot} {\sc \"Ottinger, H.C.:}
{\em Stochastic Processes in Polymeric Fluids}, Springer, Berlin, 1995.


\bibitem{PNKS}  {\sc Papageorgiou, N. S., and Kyritsi-Yiallourou, S.:} {\em Habdbook of
Applied Analysis}, Springer-Verlag, Second edistion, New York, 2009.

\bibitem{PZ}  {\sc Peszat, S., and Zabczyk, J.:} {\em Stochastic
Partial Differential Equations with L\'{e}vy Noise}, Encyclopedia of
Mathematics and Its Applications 113, Cambridge University Press,
2007.

\bibitem{Ra} {\sc Razafimandimby, P.:} On Stochastic Models Describing the Motions of
Randomly Forced Linear Viscoelastic Fluids; {\em Journal of Inequalities and Applications}, {\bf 2010}, 27 pages, 2010.

\bibitem{Re} {\sc Renardy, M.:} 
Existence of slow steady flows of viscoelastic fluids with differential constitutive equations; {\em Z. Angew. Math. Mech.}, \textbf{65}, 449--451, 1985.

\bibitem{RJC}  {\sc Robinson, J. C.:}
{\em Infinite-Dimensional Dynamical Systems, An Introduction to
Dissipative Parabolic PDEs and the Theory of Global Attractors},
Cambridge University Press, UK, 2001.


\bibitem{RuZ} {\sc R\"{u}diger, B., and Ziglio, G.:} It\^{o} Formula for Stochastic Integrals w.r.t. Compensated Poisson
Random Measures on Separable Banach Spaces; {\em Stochastics} {\bf
78}(6), 377 -- 410, 2006.

\bibitem{NVW} {\sc van Neerven, Jan, V., Mark, Weis, L.:} 
Maximal $L^p$-regularity for stochastic evolution equations; 
{\em SIAM J. Math. Anal. } {\bf 44}, no. 3, 1372--1414, 2012.

\bibitem{Ya} {\sc Yamazaki, K.:} Global martingale solution for the stochastic
Boussinesq system with zero dissipation; {\em Stochastic Analysis and Applications}, {\bf 34} (3), 404--426, 2016.

\end{thebibliography}
\end{document}